\documentclass{article}

\usepackage{rotating}
\usepackage{multirow}

\usepackage{amsmath}
\usepackage{amssymb}
\usepackage{amsthm}
\usepackage{mathrsfs}
\usepackage{wasysym}
\usepackage{bbm}		
\usepackage{nccmath}	
\usepackage{scalerel}	
\usepackage{mathtools}
\usepackage[shortlabels]{enumitem}
\usepackage{braket}		
\usepackage{mdwlist}		
\usepackage[normalem]{ulem}		
\usepackage{xcolor} 
\usepackage[a4paper, left=2.7cm, right=2.7cm, top=3.4cm, bottom=2cm]{geometry}
\usepackage[pdftex=true,hyperindex=true,pdfborder={0 0 0}]{hyperref}
\usepackage{array,multirow,makecell}
\usepackage{hhline}		
\usepackage[nottoc]{tocbibind}	
\usepackage{doi}
\usepackage[numbers]{natbib}
\usepackage{nicefrac}
\usepackage{tikz,tkz-tab}
\usetikzlibrary{arrows}
\usetikzlibrary{decorations.markings}
\usetikzlibrary{decorations.pathmorphing}
\usetikzlibrary{decorations.pathreplacing}
\usetikzlibrary{decorations.shapes}
\usetikzlibrary{decorations.text}
\usetikzlibrary{shapes}
\usetikzlibrary{shapes,arrows,automata,matrix,fit}

\newcommand{\ZZ}{\mathbb{Z}}			
\newcommand{\Z}{\mathbb{Z}}				
\newcommand{\NN}{\mathbb{N}}			
\newcommand{\N}{\mathbb{N}}				
\newcommand{\RR}{\mathbb{R}}			
\newcommand{\QQ}{\mathbb{Q}}			
\newcommand{\CC}{\mathbb{C}}			
\newcommand{\GG}{\mathbb{G}}			
\renewcommand{\P}{\mathbb{P}}			

\newcommand{\symb}[1]{\mathtt{#1}}		
\newcommand{\isdef}{\triangleq}			
\newcommand{\abs}[1]{
	\left\lvert#1\right\rvert%
}
\newcommand{\bigabs}[1]{
	\big\lvert#1\big\rvert%
}

\newcommand{\norm}[1]{
	\left\lVert#1\right\rVert%
}

\newcommand{\snorm}[1]{
	\left\llangle#1\right\rrangle%
}

\newcommand{\biggsnorm}[1]{
	\bigg\llangle#1\bigg\rrangle%
}
\newcommand{\oo}{\circ}				

\newcommand{\majority}{
	\operatorname{\mathrm{majority}}%
}
\newcommand{\dd}{\mathrm{d}}			
\newcommand{\xPr}{\operatorname{\mathbb{P}}}		
\newcommand{\xExp}{\operatorname{\mathbb{E}}}		
\newcommand{\indicator}[1]{\mathbbm{1}_{#1}}			

\newcommand{\pspace}[1]{\mathcal{#1}}	
\newcommand{\xspace}[1]{\mathscr{#1}}	
\newcommand{\field}[1]{\mathfrak{#1}}	
\newcommand{\SymGroup}{\mathrm{Sym}}	
\newcommand{\TV}{\mathrm{TV}}			
\newcommand{\critical}{\mathsf{c}}		

\def \M {\xspace{M}}					
\def \X {\pspace{X}}					
\def \Neighb {\mathcal N}

\def \E {\mathbb E}

\def \qm {%
	\mathchoice
		{\mbox{\normalsize\textcircled{\raisebox{-0.4pt}{\footnotesize ?}}}}
		{\mbox{\normalsize\textcircled{\raisebox{-0.4pt}{\footnotesize ?}}}}
		{\mbox{\scriptsize\textcircled{\raisebox{-0.3pt}{\tiny ?}}}}
		{\mbox{\tiny\textcircled{\fontsize{3.5}{3.5}\selectfont ?}}}
}
\def \qO {\symb{0}} 
\def \qX {\symb{1}} 
\newcommand{\blank}{\diamond}	
\newenvironment{cellmatrix}{
	\left[\begingroup\small\begin{matrix}%
}{%
	\end{matrix}\endgroup\right]%
}
\newcommand{\pW}{\symb{W}}		
\newcommand{\pL}{\symb{L}}		
\newcommand{\pR}{\symb{R}}		

\def \t {\tilde}

\def \ave {\preceq}

\def \e {\mathbf e}

\def \u {\uline{\mathbf{f}}}	
\def \qu {\uline{\mathbf{q}}}	
\def \tu {\uline{\mathbf{n}}}	

\newcommand{\hmax}{\overline{h}}	

\newcommand{\xConfig}[1]{%
	\begin{tikzpicture}[
		baseline=-\the\dimexpr\fontdimen22\textfont2\relax,ampersand replacement=\&]
		\matrix[
			matrix of math nodes,
			nodes={
				minimum size=1.2ex,text width=1.2ex,
				text height=1.2ex,inner sep=3pt,draw={gray!20},align=center,
				anchor=base
			}, row sep=1pt,column sep=1pt
		] (config) {#1};
		\node[draw,rectangle,help lines,fit=(config), inner sep=0pt] {};
	\end{tikzpicture}
}

\makeatletter
\DeclareFontFamily{OMX}{MnSymbolE}{}
\DeclareSymbolFont{MnLargeSymbols}{OMX}{MnSymbolE}{m}{n}
\SetSymbolFont{MnLargeSymbols}{bold}{OMX}{MnSymbolE}{b}{n}
\DeclareFontShape{OMX}{MnSymbolE}{m}{n}{
    <-6>  MnSymbolE5
   <6-7>  MnSymbolE6
   <7-8>  MnSymbolE7
   <8-9>  MnSymbolE8
   <9-10> MnSymbolE9
  <10-12> MnSymbolE10
  <12->   MnSymbolE12
}{}
\DeclareFontShape{OMX}{MnSymbolE}{b}{n}{
    <-6>  MnSymbolE-Bold5
   <6-7>  MnSymbolE-Bold6
   <7-8>  MnSymbolE-Bold7
   <8-9>  MnSymbolE-Bold8
   <9-10> MnSymbolE-Bold9
  <10-12> MnSymbolE-Bold10
  <12->   MnSymbolE-Bold12
}{}

\let\llangle\@undefined
\let\rrangle\@undefined
\DeclareMathDelimiter{\llangle}{\mathopen}%
                     {MnLargeSymbols}{'164}{MnLargeSymbols}{'164}
\DeclareMathDelimiter{\rrangle}{\mathclose}%
                     {MnLargeSymbols}{'171}{MnLargeSymbols}{'171}
\makeatother

\newcommand{\ditto}[1][2em]{\rule[0.5ex]{#1}{0.4pt}~\raisebox{-0.5ex}{''}~\rule[0.5ex]{#1}{0.4pt}}
\newcommand{\xditto}[1]{\ditto[\widthof{#1}/2-\widthof{~'}]}
\newcommand{\nRoman}[1]{\textup{\uppercase\expandafter{\romannumeral#1}}}

\theoremstyle{plain}
\newtheorem{theorem}{Theorem}[section]

\newtheorem{lemma}[theorem]{Lemma}
\newtheorem{corollary}[theorem]{Corollary}
\newtheorem{proposition}[theorem]{Proposition}

\newtheorem{openproblem}{Problem}
\theoremstyle{definition}
\newtheorem{remark}[theorem]{Remark}
\newtheorem{example}[theorem]{Example}
\theoremstyle{remark}

\newcommand{\exampleqed}{\ensuremath{\ocircle}\par}
\newcommand{\remarkqed}{\ensuremath{\Diamond}\par}

\begin{document}

\title{
	Ergodicity of some classes of cellular automata subject to noise%
	\thanks{%
		The work of ST is supported by ERC Advanced Grant 267356-VARIS of Frank den Hollander.
		ST also wishes to thank the Department of Mathematics at the University of British Columbia for support.
	}
}

\author{%
	Ir\`ene Marcovici%
	\thanks{%
		Institut \'Elie Cartan de Lorraine, UMR 7502, Universit\'e de Lorraine, CNRS, France;
		Email: \href{mailto:irene.marcovici@univ-lorraine.fr}{\texttt{irene.marcovici@univ-lorraine.fr}}
	}
	\and
	Mathieu Sablik%
	\thanks{%
		Institut de Math\'ematiques de Toulouse, UMR 5219, Universit\'e Paul Sabatier, CNRS, France;
		Email: \href{mailto:mathieu.sablik@math.univ-toulouse.fr}{\texttt{mathieu.sablik@math.univ-toulouse.fr}}
	}
	\and
	Siamak Taati%
	\thanks{%
		Bernoulli Institute, University of Groningen, P.O.\ Box 407, 9700 AK Groningen, The Netherlands;
		Email: \href{mailto:siamak.taati@gmail.com}{\texttt{siamak.taati@gmail.com}}
	}
}


\maketitle

\begin{abstract}
Cellular automata (CA) are dynamical systems on symbolic configurations on the lattice. They are also used as models of massively parallel computers. As dynamical systems, one would like to understand the effect of small random perturbations on the dynamics of CA. As models of computation, they can be used to study the reliability of computation against noise. 

We consider various families of CA (nilpotent, permutive, gliders, CA with a spreading symbol, surjective, algebraic) and prove that they are highly unstable against noise, meaning that they forget their initial conditions under slightest positive noise. This is manifested as the ergodicity of the resulting probabilistic CA. The proofs involve a collection of different techniques (couplings, entropy, Fourier analysis), depending on the dynamical properties of the underlying deterministic CA and the type of noise.

\medskip

\noindent
\emph{Keywords:} Cellular automata, probabilistic cellular automata, noise, ergodicity, coupling, entropy method, Fourier analysis

\smallskip

\noindent
\emph{MSC2010:} 60K35; 60J05; 37B15; 37A50

\renewcommand{\contentsname}{\vspace{-1.5em}}
{\footnotesize\tableofcontents}

\end{abstract}

\section{Introduction}

Consider a configuration of symbols (or colors) from a finite set $S$ on the sites of the hypercubic lattice~$\ZZ^d$. 
A \emph{cellular automaton} (CA) is a dynamical system on such configurations, obtained by iterating a local update rule simultaneously at every site of the lattice.
When the updates are random, we have a Markov process called a \emph{probabilistic cellular automaton} (PCA): at each time step, the new symbol at each site is randomly updated, independently of the others, according to a distribution prescribed by the current pattern of symbols on a finite collection of neighbouring sites.

CA and PCA have been widely studied with various motivations~\cite{Wol86, TofMar87, TooVasStaMitKurPir90, Gar95, ChoDro98, Wol02, Kur03, Kar05, Ada09, CecCoo10, RozBacKok12, MaiMar14_TCS, LouNar18}. Despite the multiplicity of viewpoints, a central problem is to describe the asymptotic behaviour of the system and its dependence on the initial condition. Indeed, even when the local behaviour is simple, the global behaviour is generally difficult to predict, and there are only few CA or PCA for which we have a complete and explicit description of the asymptotic behaviour. 

The most basic question about the asymptotic behaviour of a PCA is its \emph{ergodicity}.  A PCA is said to be ergodic if it asymptotically ``forgets'' its initial condition, meaning that the distribution of its configuration always converges to one and the same distribution regardless of the initial condition.
In other words, a PCA is ergodic if its action on probability measures has a unique fixed point that attracts all the other measures.
This paper concerns the ergodicity problem for the family of PCA obtained by perturbing CA with noise.

In computer science, deterministic CA are used as models of massively parallel computers (see e.g.,~\cite{Gar95} and the relevant chapters of~\cite{Ada09} and~\cite{RozBacKok12}). In order to study the reliability of computation against noise, one is interested in the effect of small random perturbations on the dynamics of the CA. A prerequisite for the ability to perform computation reliably in presence of noise is that the system should be able to remember at least one bit of information from its input, for otherwise the output will be pure noise and independent of the input.  
Thus, a CA that becomes ergodic when perturbed by noise cannot serve as a fault-tolerant computer in presence of noise.

From the perspective of probability theory, noisy CA constitute a class of PCA that are close to being deterministic.  In models originating from statistical physics, low noise corresponds to low temperature, and the study of low-noise PCA shares the same kind of challenges as in low-temperature models.
In particular, the ergodicity question in the low-noise regime is closely related to the question of presence or absence of phase transition at low temperature~\cite{KozVas80, GolKuiLebMae89, LebMaeSpe90, DaiLouRoe02}.
From a more abstract point of view, CA form a rich class of topological dynamical systems, and the introduction of random perturbations allows one to study probabilistic notions of sensitivity and stability.

For deterministic CA, the common tools for describing the possible asymptotic behaviour require the update rule to have specific algebraic or combinatorial structure.
One approach to analyze the asymptotic distribution is to interpret the dynamics in terms of ``particles'' that move and interact~\cite{BelFer95,BelFer05,Fis90,Kur03b,HelSab17}.
An alternative approach relies on the CA to have an algebraic structure~\cite{Lin84,PivYas02,FerMaaMarNey00,HelSalThe17}.  For deterministic CA, ergodicity is equivalent to nilpotency, a property which is algorithmically undecidable~\cite{Kar92}.  Nonetheless, nilpotent CA are not so widespread and a typical CA often exhibits different asymptotic behaviour depending on its initial condition.  In fact, using the computation capabilities of CA, one can design deterministic CA having about any behaviour wanted~\cite{HelSab16}.

The case of PCA is quite different: constructing a CA whose trajectories remain distinguishable under the influence of noise is a notoriously difficult problem. Most CA seem to be highly unstable against noise, meaning that they forget their initial conditions under slightest positive noise. This is manifested as the ergodicity of the resulting PCA. The only known example of a one-dimensional CA that remains non-ergodic under sufficiently small positive noise has a sophisticated construction due to Peter G\'acs~\cite{Gac86, Gac01}. In higher dimensions, a family of examples is provided by Andrei Toom~\cite{Too80}, but the problem remains highly non-trivial. 

A variety of tools have been developed to study the ergodicity of PCA. However, most of these tools only allow to handle PCA for which all the transition probabilities are sufficiently large (i.e., the high-noise/high-temperature regime). In particular, ergodicity is often difficult to prove for noisy CA when the noise is small, even in cases where it appears clear from heuristics or simulations. Consider for instance the simplest case of a one-dimensional PCA with binary alphabet and neighbourhood of size $2$, under the additional assumption of left-right symmetry of the update rule.  Such a PCA is identified by three parameters. The standard methods can be used to handle more than 90\% of the volume of the cube $[0,1]^3$ where the parameters lie~\cite[Chap.~7]{TooVasStaMitKurPir90}.  However, when approaching some edges of the cube, none of the known criteria for ergodicity holds, although one may expect the ergodicity to be the norm, as soon as the parameters belong to the interior of the cube.

To understand the frontiers between ergodicity and non-ergodicity, we pursue the program of identifying dynamical and combinatorial properties for a CA that guarantee the ergodicity of its random perturbations.
We prove the ergodicity of various families of CA (nilpotent, permutive, gliders, CA with spreading symbols, surjective, algebraic) subject to noise, using a collection of different techniques (couplings, entropy, Fourier analysis).

The results are summarized in Section~\ref{sec:prelim:results}.
Section~\ref{sec:prelim:notation} is dedicated to notation and terminology.
In Section~\ref{sec:prelim:ergodicity}, we discuss the notion of ergodicity and prove two general results regarding the unique invariant measure of ergodic PCA.
The various models of noise considered in this paper are introduced in Section~\ref{sec:prelim:noise}.
The ergodicity results are divided into three sections based on the method of proof they use: the coupling method (Sec.~\ref{sec:coupling}), the entropy method (Sec.~\ref{sec:entropy}) and the Fourier analysis method (Sec.~\ref{sec:fourier}).  We conclude with some open problems in Section~\ref{sec:open}.

\section{Preliminaries and results}
\label{sec:prelim}

\subsection{Notation and terminology}
\label{sec:prelim:notation}


We shall generally refer to the book by K\r{u}rka~\cite{Kur03} and the survey by Kari~\cite{Kar05} for background on deterministic cellular automata and to the surveys by Toom et al.~\cite{TooVasStaMitKurPir90} and Mairesse and Marcovici~\cite{MaiMar14_TCS} for background on probabilistic cellular automata.


Let $S$ be a finite set of \emph{symbols}, and $d\geq 1$ an integer.
A \emph{configuration} on the $d$-dimensional lattice $\ZZ^d$ is a map $x:\ZZ^d\to S$ assigning a symbol~$x_k$ from~$S$ to each \emph{site}~$k$ of $\ZZ^d$.
We will often denote the set of all configurations by $\X\isdef S^{\ZZ^d}$.
A map $F:\X\to\X$ is a \emph{cellular automaton} (CA) on $\X$ if there exist $n_1,n_2,\ldots,n_m\in\ZZ^d$ and $f:S^m\to S$ such that
\begin{align}
		(Fx)_k\isdef f(x_{k+n_1},\ldots, x_{k+n_m}) 
\end{align}
for each $x\in\X$ and $k\in\ZZ^d$.
The function $f:S^m\to S$ is called the \emph{local rule} of the CA and the set $\Neighb\isdef\{n_1,n_2,\ldots,n_m\}$ its \emph{neighbourhood}.  The set $\Neighb(k)\isdef k+\Neighb = \{k+n: n\in\Neighb\}$ consists of the \emph{neighbours} of site $k$.  We also introduce $\Neighb^0\isdef\{0\}$, and $\Neighb^{t+1}\isdef\Neighb^t+\Neighb=\{a+b \; ; \; a\in \Neighb^t, \, b\in\Neighb\}$ for $t\geq 0$, so that $\Neighb^t$ can be thought of as the neighbourhood of the CA $F^t$. Similarly, we define $\Neighb^t(k)\isdef k+\Neighb^t$.

The restriction of a configuration $x$ to a set $K\subseteq\ZZ^d$ is denoted by~$x_K$. The \emph{translation} (or \emph{shift}) by $a\in\ZZ^d$ is the map $\sigma^a:\X\to\X$ defined by $(\sigma^a x)_k\isdef x_{a+k}$ for each $k\in\ZZ^d$.

The set $\X$ of configurations is equipped with the product topology.
If $K\subseteq\ZZ^d$ is a finite set and $y_K\in S^K$, we call the set
$[y_K]\isdef \{x\in S^{\ZZ^d}: \text{$x_k=y_k$ for all $k\in K$}\}$
a \emph{cylinder} with \emph{base}~$K$.
Each cylinder set is both open and closed, and the collection of all cylinder sets is a countable basis for the product topology of $\X$.
According to the Curtis--Hedlund--Lyndon theorem (see e.g.~\cite[Thm.~5.2]{Kur03}),
the CA on $\X$ are precisely identified by the maps $F:\X\to\X$ that are continuous and commute with all translations.

For a \emph{probabilistic cellular automaton} (PCA) the local rule is randomized, and is independently applied at every site.
More specifically, the local rule of a PCA is a stochastic matrix $\varphi:S^m\times S\to[0,1]$, so that $\sum_{b\in S}\varphi(a_1,a_2,\ldots,a_m)(b)=1$ for each $a_1,a_2,\ldots,a_m\in S$.
Starting from a configuration $x$, the symbol at each site $k$ is updated at random according to the distribution $\varphi(x_{k+n_1},x_{k+n_2},\ldots,x_{k+n_m})(\cdot)$, independently of the other sites.
This is described by a transition kernel $\Phi$, where
\begin{align}
	\Phi(x,[y_K]) &\isdef \prod_{k\in K} \varphi(x_{k+n_1},x_{k+n_2},\ldots,x_{k+n_m})(y_k)
\end{align}
for every configuration $x\in\X$ and each cylinder set $[y_K]$.
An \emph{evolution} (or \emph{trajectory}) of the PCA is a Markov process with transition kernel $\Phi$, that is, is a sequence $X^0,X^1,\ldots$ of random configurations satisfying
\begin{align}
	\xPr\big(X^{t+1}\in [y_K] \,\big|\,X^0,X^1,\ldots,X^t\big) &=
		\Phi\big(X^t,[y_K]\big) 
\end{align}
almost surely for every cylinder set $[y_K]$ and every $t\geq 0$.
A bi-infinite evolution $\ldots,X^{-1},X^0,X^1,\ldots$ is defined similarly.
A PCA has \emph{positive rates} if its local rule is strictly positive, meaning that for each $a_1,a_2,\ldots,a_m\in S$ and $b\in S$, we have $\varphi(a_1,a_2,\ldots,a_m)(b)>0$.

The set of all continuous \emph{observables} $h:\X\to\CC$, denoted by $C(\X)$, is a Banach space with the uniform norm $\norm{h}\isdef\sup_{x\in\X}\abs{h(x)}$.  A \emph{local} observable is an observable $h\in C(\X)$ that can be written as a linear combination of characteristic functions of cylinder sets, so that $h(x)$ depends on the symbols at only finitely many sites.  The local observables form a dense linear subspace of $C(\X)$, which we shall denote by $C_0(\X)$.

The set of all Borel probability measures on $\X$ is denoted by $\M(\X)$.  A measure $\mu\in\M(\X)$ is uniquely determined by the probabilities it associates to cylinder sets.  Furthermore, a sequence $\mu_1,\mu_2,\ldots\in\M(\X)$ of probability measures converges \emph{weakly} to another measure $\mu\in\M(\X)$ if and only if $\mu_n(E)\to\mu(E)$ for every cylinder set $E\subseteq\X$.
With the weak topology, the space $\M(\X)$ is compact and metrizable.
Let $\field{F}_A$ denote the sub-$\sigma$-algebra of the Borel sets generated
by the cylinder sets with base $A$.
We denote by
\begin{align}
	\label{eq:total-variation:window}
	\norm{\mu-\nu}_A &\isdef \sup_{E\in\field{F}_A}\abs{\mu(E)-\nu(E)}
		= \frac{1}{2}\sum_{u\in S^A}\bigabs{\mu([u])-\nu([u])}
\end{align}
the \emph{total variation distance} between the restrictions of $\mu$ and $\nu$ to $\field{F}_A$.  This is the distance between the distributions of $X_A$ and $Y_A$ where
$X$ and $Y$ are random configurations distributed according to~$\mu$ and~$\nu$.

A PCA kernel $\Phi$ naturally defines two continuous linear operators, one on $\M(\X)$ and the other on $C(\X)$.
Following the usual convention (e.g.~\cite{TooVasStaMitKurPir90}), we write $\Phi$ on the right-hand side of measures and on the left-hand side of observables.  Given a measure $\mu$, we denote by $\mu\Phi$ the measure defined by $(\mu\Phi)(A)\isdef\int \Phi(x,A)\mu(\dd x)$, that is, the distribution of $X^{t+1}$ if $X^t$ is distributed according to $\mu$.  Given an observable $h\in C(\X)$, we write $\Phi h$ for the observable defined by $(\Phi h)(x)\isdef \xExp\big[h(X^{t+1})\,\big|\,X^t=x\big] =\int h(y)\Phi(x,\dd y)$.


We will be concerned with PCA that are close to being deterministic.
We say that a PCA $\Phi$ is an \emph{$\varepsilon$-perturbation} of a deterministic CA $F$ if $\Phi$ and $F$ share the same alphabet $S$, and have a common neighbourhood $\Neighb$ for which their local rules satisfy
\begin{align}
	\varphi(a_1,a_2,\ldots,a_m)\big(f(a_1,a_2,\ldots,a_m)\big)\geq 1-\varepsilon
\end{align}
for all $a_1,a_2,\ldots,a_m\in S$, meaning that under $\Phi$, a deviation from $F$ may occur independently at each site with probability at most $\varepsilon$.  In other words, $\Phi$ is an $\varepsilon$-perturbation of $F$ if
\begin{align}
	\Phi\big(x,[(Fx)_K]\big) &\geq (1-\varepsilon)^{\abs{K}}
\end{align}
for every configuration $x\in\X$ and every finite set $K\subseteq\ZZ^d$.

\subsection{Ergodicity}
\label{sec:prelim:ergodicity}
A probability measure $\pi\in\M(\X)$ is \emph{invariant} under a PCA $\Phi$ if $\pi\Phi=\pi$.  The compactness of $\M(\X)$ ensures that every PCA has at least one invariant measure (see e.g.~\cite[Prop.~2.5]{TooVasStaMitKurPir90}).
The non-empty set of invariant measures for a PCA is closed and convex.

A PCA $\Phi$ is \emph{ergodic} if it has a unique invariant measure $\pi$ that attracts every initial measure $\mu$, in the sense that $\mu\Phi^t\to\pi$ weakly as $t\to\infty$.
Note that a PCA with a unique invariant measure may not be ergodic~\cite{ChaMai10} (see also~\cite{JahKul15}).
When the convergence is uniform among all initial measures, equivalently, when
$\Phi^t(\,\cdot\,,[u])\to\pi([u])$ uniformly
for each cylinder set $[u]$,
we say that $\Phi$ is \emph{uniformly ergodic}.
It is not known whether a PCA could exist that is ergodic but not uniformly ergodic.  We conjecture that every ergodic PCA is uniformly ergodic.\footnote{%
	Compare this with deterministic CA for which asymptotic nilpotency ($\equiv$ ergodicity)
	is equivalent to nilpotency ($\equiv$ uniform ergodicity)~\cite{GuiRic08,Sal12, MaiMar14_TCS}.
}

Observe that the unique invariant measure of an ergodic PCA is shift-invariant, that is, $\pi\oo\sigma^{-k}=\pi$ for every $k\in\ZZ^d$.
In view of the result of Goldstein et al.~\cite{GolKuiLebMae89}, it seems plausible that the unique invariant measure of a positive-rate ergodic PCA is always \emph{spatially mixing} ($\equiv$ mixing under the shift action), that is,
\begin{align}
	\lim_{k\to\infty}\pi([u]\cap\sigma^{-k}[v]) &= \pi([u])\pi([v])
\end{align}
for every two cylinder sets $[u]$ and $[v]$.
The spatial mixing of the invariant measure is known for certain classes of ergodic PCA (see e.g.~\cite{LebMaeSpe90,Ste91,MaeShl91,Lou04,ColTis10,BusMaiMar13}, as well as~\cite{Vas78, MaiMar14_IHP, CasMar15} in which the unique invariant measure is explicitly known).
However, even the weaker condition of \emph{spatial ergodicity} ($\equiv$ ergodicity under the shift action) is not known for general ergodic PCA.
We now present a more general condition that guarantees the spatial mixing of the unique invariant measure.

Let $\Phi$ be a uniformly ergodic PCA with unique invariant measure $\pi$.
Then, for every finite set $A\subseteq\ZZ^d$, the (maximum) \emph{distance from stationarity} on $A$ at time $t$,
\begin{align}
\label{eq:rate-function}
	d_A(t) &\isdef \sup_{x\in X} \norm{\pi(\cdot)-\Phi^t(x,\cdot)}_A \;,
\end{align}
decreases to $0$ as $t\to\infty$.
Roughly speaking, the next proposition shows that if the speed at which $d_A(t)$ approaches zero depends only on $\abs{A}$,
but not on the shape of $A$, then the unique invariant measure is spatially mixing.


\begin{proposition}[Spatial mixing of unique invariant measure]
\label{prop:spatial-mixing}
	Let $\Phi$ be a uniformly ergodic PCA, and for each finite set $A\subseteq\ZZ^d$,
	let $d_A(t)$ denote the distance from stationarity on $A$ at time $t$.
	Suppose there is a family of functions $\rho_n(t)$, $n\in\NN$ such that
	$d_A(t)\leq \rho_{\abs{A}}(t)$ and $\rho_n(t)\to 0$ as $t\to\infty$.
	Then, the unique invariant measure of $\Phi$ is spatially mixing.
\end{proposition}
\begin{proof}
	Let $\pi$ be the unique invariant measure of $\Phi$ and $\Neighb$
	the neighbourhood of its local rule. 
	Consider two finite patterns $u\in S^A$ and $v\in S^B$, and let $k\in\ZZ^d$.
	Then,
	\begin{align}
		\label{eq:prop:exp-ergodic:mixing:proof}
		\abs{\pi([u])-\Phi^t(x,[u])} &\leq \rho_{\abs{A}}(t) \;, \\
		\abs{\pi(\sigma^{-k}[v])-\Phi^t(x,\sigma^{-k}[v])} &\leq \rho_{\abs{B}}(t) \;, \\
		\abs{\pi\big([u]\cap\sigma^{-k}[v]\big)-\Phi^t\big(x,[u]\cap\sigma^{-k}[v]\big)}
			&\leq \rho_{\abs{A}+\abs{B}}(t) \;.
	\end{align}
	Let $r,a,b\geq 0$ be such that $\Neighb\subseteq[-r,r]^d$, $A\subseteq[-a,a]^d$
	and $B\subseteq[-b,b]^d$.
	Observe that if $t\geq 0$ is such that $\Neighb^i(A)\cap \big(k+\Neighb^i(B)\big)=\varnothing$ for each $0\leq i\leq t$, then
	\begin{align}
		\Phi^t\big(x,[u]\cap\sigma^{-k}[v]\big) &=
			\Phi^t(x,[u])\Phi^t(x,\sigma^{-k}[v]) \;.
	\end{align}
	This is because, given $x$, the random choices used to determine the patterns on $A$ and $k+B$
	at time $t$ are independent (see Fig.~\ref{fig:mixing}).
	Thus, choosing $t\leq t_k\isdef\big\lfloor(\norm{k}_\infty-a-b)/(2r)\big\rfloor$,
	we can write
	\begin{align}
		\abs{\pi\big([u]\cap\sigma^{-k}[v]\big)-\pi([u])\pi(\sigma^{-k}[v])} 
		&\leq
			\abs{\pi\big([u]\cap\sigma^{-k}[v]\big)-\Phi^t\big(x,[u]\cap\sigma^{-k}[v]\big)} \nonumber \\
		&\qquad
				+ \Phi^t(x,\sigma^{-k}[v])\abs{\pi([u])-\Phi^t(x,[u])} \\
		&\qquad
			+ \pi([u])\abs{\pi(\sigma^{-k}[v])-\Phi^t(x,\sigma^{-k}[v])} \nonumber \\
		&\leq
			\rho_{\abs{A}+\abs{B}}(t) + \rho_{\abs{A}}(t) + \rho_{\abs{B}}(t) \;.
	\end{align}
	Observe that $t_k\to\infty$ as $\norm{k}_\infty\to\infty$.
	Hence,
	\begin{align}
		\MoveEqLeft
		\limsup_{k\to\infty}\abs{\pi\big([u]\cap\sigma^{-k}[v]\big)-\pi([u])\pi([v])} \nonumber \\
		&\leq
			\lim_{t\to\infty}\left[\rho_{\abs{A}+\abs{B}}(t)
			+ \rho_{\abs{A}}(t) + \rho_{\abs{B}}(t)\right] = 0 \;. \\
		& \qedhere
	\end{align}
	
\begin{figure}[h]
\begin{center}
\begin{tikzpicture}
\draw[ultra thin] (-5.25,0) -- (5.25,0) ;
\draw[ultra thin] (-5.25,-2) -- (5.25,-2) ;
\node[left,overlay] at (-5.75,0) {time $t$} ;
\node[left,overlay] at (-5.75,-2) {time $0$} ;
\draw[ultra thick, blue] (-3,0) --  (-2,0) ;
\node[above] at (-2.5,0) {$A$};
\node[below] at (-2.5,-2) {$\Neighb^t(A)$};
\draw[blue] (-3,0) --  (-5,-2) ;
\draw[blue] (-2,0) --  (0,-2) ;
\draw[ultra thick, blue] (-5,-2) --  (0,-2) ;
\draw[ultra thick, blue] (2.5,0) --  (3,0) ;
\node[above] at (2.75,0) {$k+B$};
\node[below] at (2.75,-2) {$k+\Neighb^t(B)$};
\draw[blue] (2.5,0) --  (0.5,-2) ;
\draw[blue] (3,0) --  (5,-2) ;
\draw[ultra thick, blue] (0.5,-2) --  (5,-2) ;
\end{tikzpicture}
\end{center}
\caption{Illustration of the proof of Prop.~\ref{prop:spatial-mixing}.}\label{fig:mixing}
\end{figure}
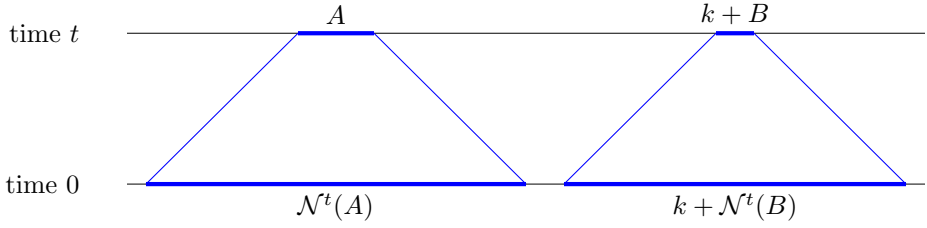
\end{proof}

All the ergodic PCA appearing in this paper satisfy the hypothesis of the above proposition.
It would be interesting to know whether the unique invariant measures of these PCA satisfy any stronger mixing property, such as the strong mixing property of extremal Gibbs measures~\cite[Sec.~7.1]{Geo88}.  A similar argument as above shows that under the hypothesis of Proposition~\ref{prop:spatial-mixing}, the unique invariant measure is spatially $k$-fold mixing for all $k$.  We conjecture that for all the classes of ergodic PCA studied in this paper (with the possible exception of those in Section~\ref{sec:gliders}), the unique invariant measure is in fact measure-theoretically isomorphic to a Bernoulli process (see~\cite{Ste91,BerSte99}). 

A rather different type of question about a probability measure on $\X$ is whether the probabilities it associates to cylinder sets can be computed by an algorithm.  It turns out that the unique invariant measure of an ergodic PCA is always computable provided the transition probabilities of the PCA are computable numbers.
A real number $x$ is said to be \emph{computable} if it can be approximated with arbitrary accuracy using an algorithm, that is if there exists a computable function $f_x:\NN\to\QQ$ such that $\abs{f_x(n)-x}<\nicefrac{1}{n}$ for all $n\in\NN$.  We say that a PCA $\Phi$ is \emph{computable} if the values of its local rule $\varphi$ are computable real numbers.
Let $S^\#$ denote the set of patterns $u\in S^A$ where $A\subseteq\ZZ^d$ is finite.
A measure $\mu\in\M(\X)$ is \emph{computable} if there exists a computable function $f_\mu:S^\#\times\NN\to\QQ$ such that $\abs{f_\mu(u,n)-\mu([u])}<\nicefrac{1}{n}$ for every $u\in S^\#$ and $n\in\NN$.
Observe that if $\Phi$ is a computable PCA and $\mu$ is a computable measure, then $\mu\Phi$ is also a computable measure.

\begin{proposition}[Computability of unique invariant measure]
\label{prop:computability}
	Let $\Phi$ be a computable PCA with a unique invariant measure $\pi$.
	Then, $\pi$ is computable.
\end{proposition}
\begin{proof}
	We first present the sketch of the proof and then get into more details.
	Let $w\in S^A$ be a finite pattern and suppose we want to approximate $\pi([w])$ within accuracy~$\nicefrac{1}{n}$.
	The idea is that for every finite set $B\supseteq A$, we can approximately identify the set of measures that are close to being invariant when restricted to the $\sigma$-algebra $\field{F}_B$ of events happening on $B$.  More specifically, for $B\supseteq A$ and $m\in\NN$, let
	\begin{align}
		\xspace{Q}_{B,m} &\isdef
			\big\{\mu\in\M(\X): \norm{\mu\Phi-\mu}_B<\nicefrac{1}{m}\big\} \;.
	\end{align}
	We will show that given $B$ and $m\geq n$, we can algorithmically generate a finite set $\xspace{R}_{B,m}$ of representatives from $\xspace{Q}_{B,m}$	such that for every $\mu\in\xspace{Q}_{B,3m}$, there is a $\nu\in\xspace{R}_{B,m}$ with $\norm{\nu-\mu}_B<\nicefrac{1}{(3m)}<\nicefrac{1}{(2n)}$.
	A compactness argument will show that for all sufficiently large $B$ and $m$, every two measures $\nu,\nu'\in\xspace{R}_{B,m}$ associate approximately the same probabilities to the cylinder set $[w]$, namely $\abs{\nu'([w])-\nu([w])}<\nicefrac{1}{(2n)}$.
	Since $\pi\in\xspace{Q}_{B,3m}$, it will then follow that for each $\nu\in\xspace{R}_{B,m}$, the value $\nu([w])$ approximates $\pi([w])$ with accuracy~$\nicefrac{1}{n}$.

	More precisely, the algorithm thus proceeds as follows.
	Denote by $I_k\isdef[-k,k]^d\cap\ZZ^d$ the centered hypercube of size $(2k+1)^d$ in $\ZZ^d$.
	We choose $m_0$ such that $I_{m_0}\supseteq A$.  For $m=m_0, m_0+1, \ldots$ we generate a set $\xspace{R}_{I_m,m}$ with the above-mentioned property and calculate $\varepsilon\isdef\max\{\abs{\nu'([w])-\nu([w])}: \nu,\nu'\in\xspace{R}_{I_m,m}\}$.
	Once $\varepsilon<\nicefrac{1}{(2n)}$, we stop and return $\nu([w])$ for an arbitrarily chosen element of $\xspace{R}_{I_m,m}$.
	
	Let us first show that $\varepsilon$ will eventually become smaller than $\nicefrac{1}{(2n)}$.
	Indeed, suppose that for every~$m$, there are two elements $\mu_m,\mu'_m\in\xspace{Q}_{I_m,m}$ such that $\abs{\mu'_m([w])-\mu_m([w])}\geq\nicefrac{1}{(2n)}$.  By compactness, there is a sequence $m_1<m_2<\ldots$ such that $\mu_{m_i}$ converges weakly to a measure $\mu$ and $\mu'_{m_i}$ converges weakly to a measure $\mu'$.  Clearly $\abs{\mu'([w])-\mu([w])}\geq\nicefrac{1}{(2n)}$ and in particular, $\mu'\neq\mu$.  On the other hand, from the definition of $\xspace{Q}_{B,m}$, it follows that both $\mu$ and $\mu'$ must be invariant under $\Phi$, hence a contradiction with the uniqueness of the invariant measure.
	
	It remains to show that for each $B$ and $m$, a set $\xspace{R}_{B,m}$ with the prescribed properties can be generated.
	The PCA $\Phi$ induces an affine mapping from probability measures on $S^{\Neighb(B)}$ to probability measures on $S^B$.
	It follows easily that $\norm{\mu'\Phi-\mu\Phi}_B\leq\norm{\mu'-\mu}_{\Neighb(B)}$ for every $\mu,\mu'\in\M(\X)$.
	Fix an arbitrary symbol $\blank\in S$, and for finite $C\subseteq\ZZ^d$ and $k\in\NN$, define
	\begin{align}
		\M_{C,k} &\isdef \bigg\{
			\frac{1}{k}\sum_{i=1}^k\delta_{x^{(i)}} \;:\;
			\text{$x^{(1)}, \ldots, x^{(k)}\in\X$ and $x^{(i)}_j=\blank$ for $j\notin C$}
		\bigg\} \;,
	\end{align}
	where $\delta_x$ denotes the Dirac measure centered at $x$.
	When restricted to $\field{F}_C$, the elements of $\M_{C,k}$ are precisely those measures whose probabilities are rational with denominator $k$.  In particular, for every $\mu\in\M(\X)$, there exists a measure $\nu\in\M_{C,k}$ such that $\norm{\nu-\mu}_C<\abs{S}^{\abs{C}}/k$.
	Given $B$ and $m$, construct the set
	\begin{align}
		\xspace{R}_{B,m} &\isdef \big\{
			\nu\in\M_{\Neighb(B),k}: \norm{\nu\Phi-\nu}_B<\nicefrac{1}{m}
		\big\}
	\end{align}
	with $k\isdef 3m\abs{S}^{\abs{\Neighb(B)}}$.
	Clearly, $\xspace{R}_{B,m}\subseteq\xspace{Q}_{B,m}$.
	Let $\mu\in\xspace{Q}_{B,3m}$.  Then, there exists a measure $\nu\in\M_{\Neighb(B),k}$ such that $\norm{\nu-\mu}_{\Neighb(B)}<\abs{S}^{\abs{\Neighb(B)}}/k=\nicefrac{1}{(3m)}$.
	For this measure, we have
	\begin{align}
		\norm{\nu\Phi-\nu}_B &\leq
			\underbrace{\norm{\mu-\nu}_B}_{<\nicefrac{1}{(3m)}} +
			\underbrace{\norm{\mu\Phi-\mu}_B}_{<\nicefrac{1}{(3m)}} +
			\underbrace{\norm{\nu\Phi-\mu\Phi}_B}_{<\nicefrac{1}{(3m)}}
		< \nicefrac{1}{m} \;,
	\end{align}
	which means $\nu\in\xspace{R}_{B,m}$.
	Hence, $\xspace{R}_{B,m}$ has the desired properties.
%
%
%
%
%
%
%
\end{proof}

\subsection{Models of noise}
\label{sec:prelim:noise}

In this article, we study the ergodicity problem for perturbations of deterministic CA.
We mainly focus on perturbations obtained when adding random and independent errors to the updates of a deterministic CA.
The transition probabilities of the resulting PCA will thus have the form $\Phi(x,E)\isdef \Theta(Fx, E)$, where $F$ is the deterministic CA and $\Theta$ a \emph{noise} kernel.  We call such a perturbation a \emph{noisy} version of $F$.  The noise kernel is itself assumed to be a PCA transition kernel (albeit a simple one) so that the updates of the symbols at distinct sites are independent.
The noise is said to be \emph{positive} if its kernel has positive rates.

\paragraph{Zero-range noise.}
A \emph{zero-range} noise is a noise with neighbourhood $\Neighb=\{0\}$: the symbol at each site is randomly modified, independently of the other sites, according to transition probabilities prescribed by a stochastic matrix $\theta:S\times S\to [0,1]$.  The local rule of the noisy CA is therefore given by $\varphi(a_1,a_2,\ldots,a_m)(b)\isdef \theta\big(f(a_1,a_2,\ldots,a_m),b\big)$, where $f$ is the local rule of the original CA.  

Most of the results in this paper (with the exception of Sections~\ref{sec:envelope}, \ref{sec:nilpotent} and~\ref{sec:spreading:perturbation}) concern zero-range noise.
Various classes of zero-range noise will be considered, each with its own interpretation.  Each of our proof techniques will be well suited for some of these noise models.


\paragraph{Memoryless noise.}
A zero-range noise is \emph{memoryless} if its noise matrix can be written as $\theta(a,b)=(1-\varepsilon)\delta_a(b)+\varepsilon q(b)$, where $0\leq\varepsilon\leq 1$, $q$ is a probability distribution on $S$, and $\delta_a$ is the distribution with unit mass at~$a$.
Under a memoryless noise, a symbol is erased with probability $\varepsilon$ and replaced with an independent random symbol drawn from distribution $q$.
We call $\varepsilon$ the \emph{error} probability and $q$ the \emph{replacement} distribution of the noise.


\paragraph{Additive noise.}
Suppose that the alphabet $S$ is identified with a finite Abelian group $(\GG,+)$.  Under an \emph{additive} noise with noise distribution $q$, each symbol $a$ is replaced with a symbol $a+N$, where $N$ is an $S$-valued random variable with distribution $q$.  The noise matrix can thus be written as $\theta(a,b)\isdef q(b-a)$ for each $a,b\in S$.

\paragraph{Permutation noise.}
The \emph{permutation} noise is an extension of additive noise, where each symbol $a$ is replaced with a symbol $\varsigma(a)$, where $\varsigma$ is a random permutation of $S$ drawn according to a fixed distribution $q$.
Observe that the noise matrix of a permutation noise can be written as a convex combination
\begin{align}
	\theta(a,b) &= \sum_{\varsigma\in\SymGroup(S)}q(\varsigma)A_\varsigma(a,b)
\end{align}
of permutation matrices $A_\varsigma(a,b)\isdef\delta_{\varsigma(a)}(b)$, and therefore is a doubly-stochastic matrix.  Conversely, the Birkhoff--von Neumann theorem implies that every zero-range noise with a doubly-stochastic matrix is in fact a permutation noise.  In particular, a permutation noise is precisely a zero-range noise that preserves the uniform distribution on $S$.
The notion of noise in a weakly symmetric communication channel (see~\cite{CovTho91}) is a special case of the permutation noise.

\paragraph{Birth-death noise.}
In some of our examples (see Sections~\ref{sec:gliders}), the alphabet has the form $S=\{\qO,\qX\}^n$, where $\qX$ and $\qO$ represent the presence and absence of ``particles'' or ``walls'' at $n$ different ``layers'' of the system.
Under a \emph{birth-death} noise particles/walls appear and disappear independently at each layer, thus the noise matrix has the form
\begin{align}
	\theta(a,b) &\isdef \prod_{i=1}^n \theta_i(a_i,b_i)
\end{align}
for $a=(a_1,a_2,\ldots,a_n)$ and $b=(b_1,b_2,\ldots,b_n)$ in $S$.

\subsection{Summary of results}
\label{sec:prelim:results}


We prove several results regarding the ergodicity of noisy CA.  Each result concerns a class of CA with a specific dynamical property subject to a specific type of noise.  The following table summarizes our results.  See Figures~\ref{fig:examples:1} and~\ref{fig:examples:2} for examples illustrating the results.

\medskip


\noindent%
\begin{center}
\begin{tabular}{|c||p{0.3\textwidth}|p{0.3\textwidth}|l|} 
	\cline{2-4}
	\multicolumn{1}{c|}{} & {\bf Type of CA} & {\bf Type of noise} & {\bf Reference} \tabularnewline
	\hhline{-::===}
	\nRoman{1} & Any CA & High noise & Thm.~\ref{thm-high} \\
	\hhline{-||---}
	\nRoman{2} & Nilpotent & Small perturbation & Thm.~\ref{thm-nilp}\\
	\hhline{-||---}
	\nRoman{3} & CA with spreading symbol
		& Memoryless noise & Thm.~\ref{thm-spreading-zrmn}\\
	\nRoman{4} & \xditto{CA with spreading symbol}\newline\phantom{.}\hfill(1d with $\Neighb=\{0,1\}$)
		& Small positive perturbation & Thm.~\ref{thm-spreading}\\
	\hhline{-||---}
	\nRoman{5} & Gliders with annihilation & Birth-death noise & Thm.~\ref{thm-glider_an} \\
	\nRoman{6} & Simple gliders with\newline reflecting walls (1d) &
		\xditto{Birth-death noise} & Thm.~\ref{thm-glider_ref} \\
	\hhline{-||---}
	\nRoman{7} & Permutive (1d) & Permutation noise & Thm.~\ref{thm-perm}\\
	\hhline{=::===}
	\nRoman{8} & Surjective & Additive noise & Thm.~\ref{thm-surj}\\
	\hhline{=::===}
	\nRoman{9} & XOR & Zero-range & Thm.~\ref{thm-xor} \\
	\hhline{-||---}
	\nRoman{10} & Binary CA with\newline spreading symbol
		& Zero-range\newline\phantom{.}\hfill(75\% of parameter range) & Thm.~\ref{thm-spreading-bin} \\
	\hhline{-||---} 
\end{tabular}
\end{center}

\medskip

The results are divided into three categories, depending on the type of tools used in their proofs: coupling arguments (Sec.~\ref{sec:coupling}), entropy (Sec.~\ref{sec:entropy}) and Fourier analysis (Sec.~\ref{sec:fourier}).
The ergodicity in the high noise regime (Thm.~\ref{thm-high}) is rather standard and can be proven using various approaches.  Here we present a coupling proof using the so-called envelope PCA (introduced in~\cite{BusMaiMar13}) which we find most elegant.  The nilpotent CA are special in that they are ergodic in absence of noise.  A coupling argument will show that the ergodicity persists for small perturbations of nilpotent CA (Thm.~\ref{thm-nilp}).
The ergodicity of a CA that has a spreading symbol is intuitively plausible.  We provide three different proofs (Thms.~\ref{thm-spreading-zrmn}, \ref{thm-spreading} and~\ref{thm-spreading-bin}) each with a different model of noise and having a different degree of generality.  Theorems~\ref{thm-glider_an} and~\ref{thm-glider_ref} concern the ergodicity of simple systems of ``particles'' (or ``gliders'') moving and interacting on the lattice, where the noise occasionally destroys particles or creates new ones.
The ergodicity of permutive CA subject to permutation noise (Thm.~\ref{thm-perm}) is a special case of a result of Vasilyev~\cite{Vas78}.  The argument is based on the identification of a certain finite-state time-inhomogeneous Markov chain that is hidden inside the model.  We also present an alternative (though similar) argument using entropy (Sec.~\ref{sec:surj:proof}).  Surjective CA constitute a broad class of CA (including e.g., those addressed in Theorems~\ref{thm-glider_ref}, \ref{thm-perm} and~\ref{thm-xor}). For general surjective CA with additive noise, we are only able to prove ``ergodicity modulo shift'' (Thm.~\ref{thm-surj}), that is the convergence towards equilibrium when the starting measure is shift-invariant.  The ergodicity of the XOR CA subject to noise (Thm.~\ref{thm-xor}) is an application of the Fourier analysis approach to the ergodicity problem developed by Toom et al.~\cite[Chap.~4]{TooVasStaMitKurPir90}.

Aside from cases~\nRoman{7} and~\nRoman{8} in which the invariant measures are explicitly known, in all the classes of ergodic PCA treated in this paper, the unique invariant measure is spatially mixing and computable.  The computability of the unique invariant measure holds in general, as demonstrated in Proposition~\ref{prop:computability}.  The spatial mixing is proven in each case with the help of Proposition~\ref{prop:spatial-mixing}.  Furthermore, Proposition~\ref{prop:coupling} below provides a perfect sampling algorithm for the unique invariant measure in cases~\nRoman{1}--\nRoman{4} and~\nRoman{6}.

Let us remark that except for Theorems~\ref{thm-spreading}, \ref{thm-glider_ref} and~\ref{thm-perm}, all the results in this paper are valid in any number of dimensions.  The proof of Theorem~\ref{thm-spreading} makes use of a result on oriented bond percolation in $1+1$ dimensions, and thus relies crucially on the CA being one-dimensional.  Nevertheless, it might be possible to use the same idea in higher dimensions. Theorems~\ref{thm-glider_ref} and~\ref{thm-perm} are restricted to the one-dimensional case for expositional convenience.  In higher dimensions, the definition of a permutive CA is more cumbersome.

For the sake of comparison, let us now recall an example of a simple CA which, in presence of small noise, remains non-ergodic.
Needless to say, this example belongs to none of the CA families~\nRoman{2}--\nRoman{10} mentioned in the above table.

\begin{example}[NEC-majority]
\label{exp:toom}
	Toom's \emph{NEC-majority} CA is the two-dimensional CA $T:S^{\ZZ^2}\to S^{\ZZ^2}$ with binary alphabet $S\isdef\{\symb{0},\symb{1}\}$, where
	\begin{align}
		(Tx)_{i,j} &\isdef \majority(x_{i,j}, x_{i+1,j}, x_{i,j+1}) \;.
	\end{align}
	In other words, in one iteration of $T$, each symbol on the lattice is replaced with the symbol	that is in majority among its northern neighbour, eastern neighbour and itself.
	Observe that $T$ is \emph{monotonic} (i.e., switching some $\symb{0}$s into $\symb{1}$s	in a configuration $x$ may turn some $\symb{0}$s in $Tx$ into $\symb{1}$s but not the other way around) and \emph{symmetric} with respect to $\symb{0}\leftrightarrow\symb{1}$ exchange.
	Moreover, it can be shown that $T$ has the \emph{erosion} property on the all-$\symb{0}$ configurations (and by symmetry, also on the all-$\symb{1}$ configuration).  Namely, $T$ keeps the all-$\symb{0}$ (resp., all-$\symb{1}$) configuration unchanged, and if $x$ is any configuration in which all but finitely many sites have symbol~$\symb{0}$ (resp., symbol~$\symb{1}$), then there is a finite time $t$ for which $T^tx$ is the all-$\symb{0}$ configuration (resp., the all-$\symb{1}$ configuration).
	
	Toom~\cite{Too80} (see~\cite[Chaps.~9 and~10]{TooVasStaMitKurPir90}) proved that for sufficiently small $\varepsilon>0$,	every $\varepsilon$-perturbation of the NEC-majority CA is non-ergodic.	In fact, he showed that in any monotonic CA $T$, any homogeneous configuration $z$ on which $T$ has the erosion property is \emph{stable} against perturbations in the sense that the trajectory of any small perturbation of $T$ starting from $z$ remains forever concentrated on configurations	that agree with $z$ on the great majority of sites.
	\hfill\exampleqed
\end{example}

\begin{figure}[!ht]
\begin{scriptsize}
\begin{center}
\begin{tabular}{@{}p{0.02\textwidth}@{}p{0.45\textwidth}@{}p{0.05\textwidth}@{}p{0.45\textwidth}@{}}
	& ~~$\varepsilon=0$	&& ~~$\varepsilon=0.01$\\

	\begin{sideways}Nilpotent CA (Sec.~\ref{sec:nilpotent})\end{sideways} 
		&\boxed{\includegraphics[width=0.435\textwidth]{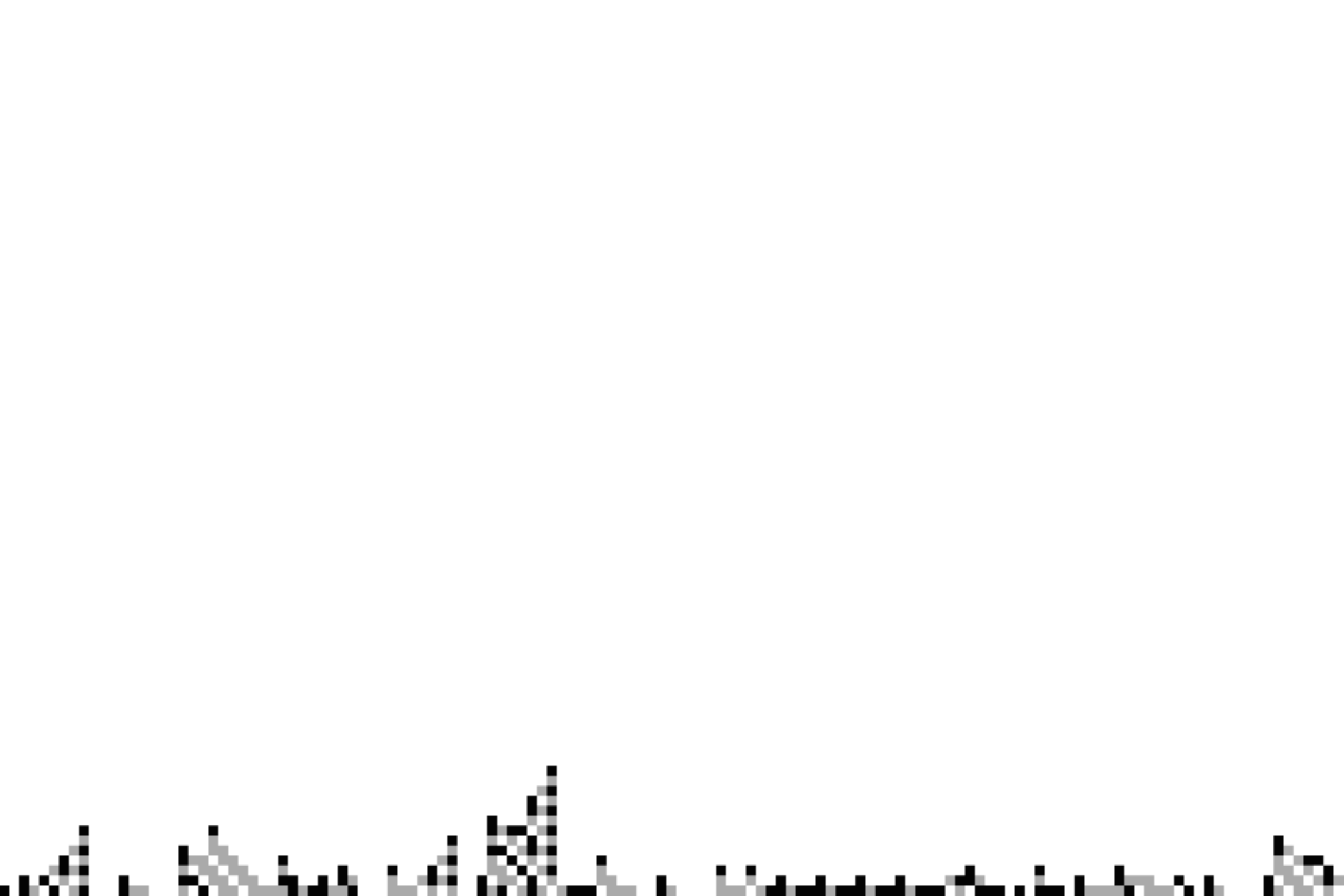}}
		&&\boxed{\includegraphics[width=0.435\textwidth]{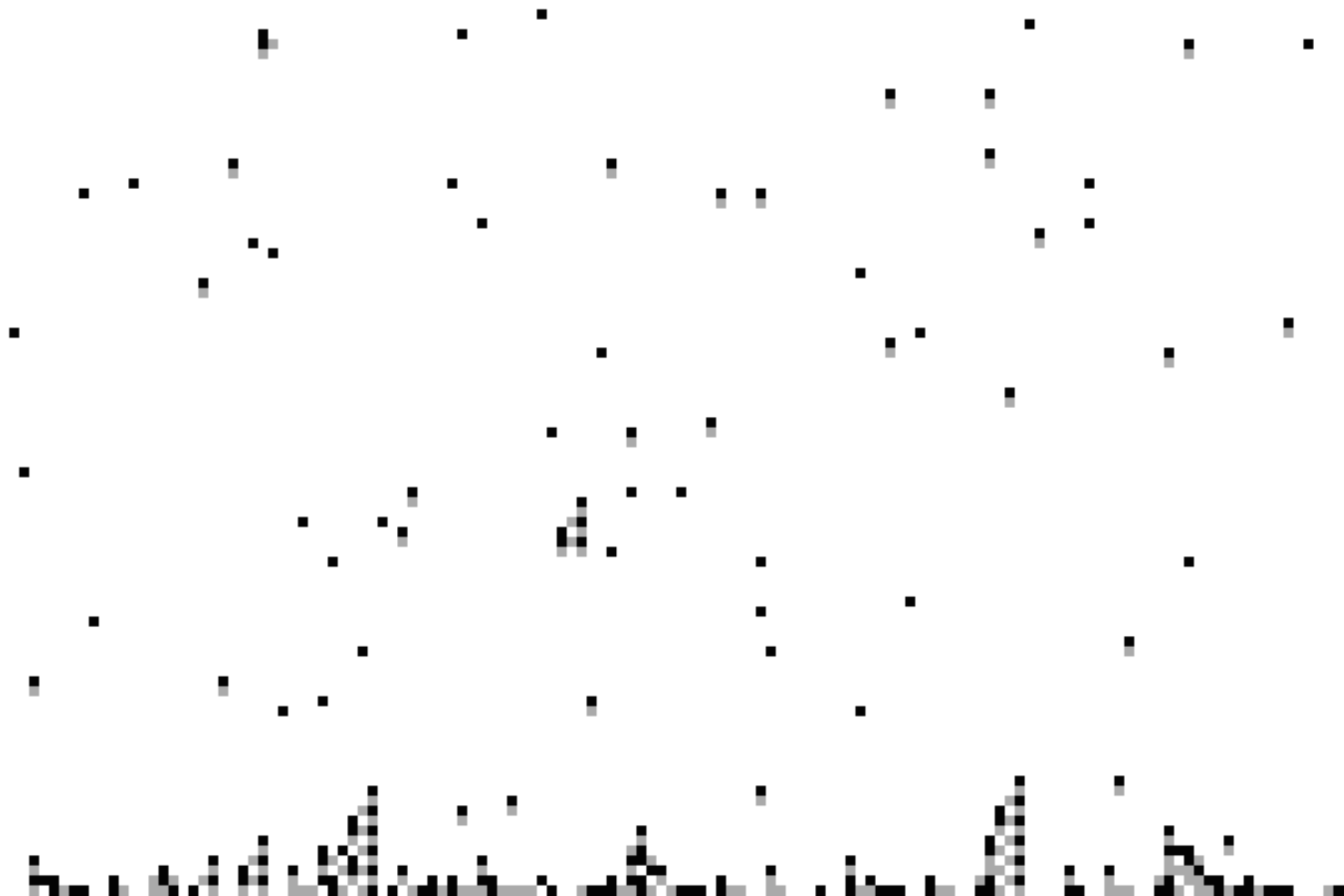}}\\ [1ex]
	&\multicolumn{3}{p{0.92\textwidth}}{This CA satisfies $F^{12}(x)=\symb{0}^\ZZ$ for all $x\in\{\symb{0},\symb{1},\symb{2}\}^\ZZ$. Without noise the system dies out; the noise adds small	local perturbations that do not propagate. Ergodicity is proven in Theorem~\ref{thm-nilp}.}
	\\ [2em]

	\begin{sideways}
		Spreading CA (Sec.~\ref{sec:spreading} and Sec.~\ref{sec:mobius:binary})
	\end{sideways} 
		&\boxed{\includegraphics[width=0.435\textwidth]{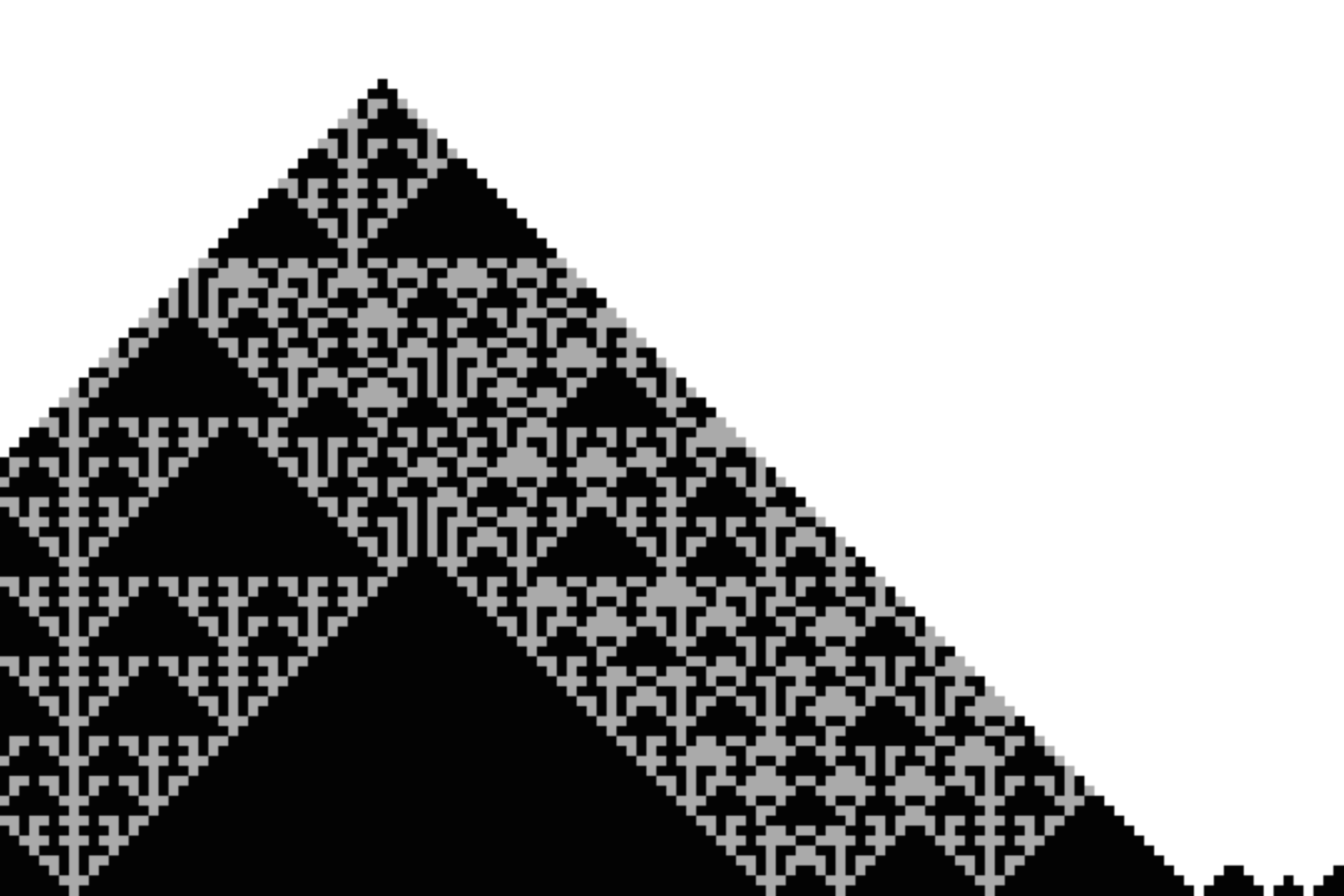}}
		&&\boxed{\includegraphics[width=0.435\textwidth]{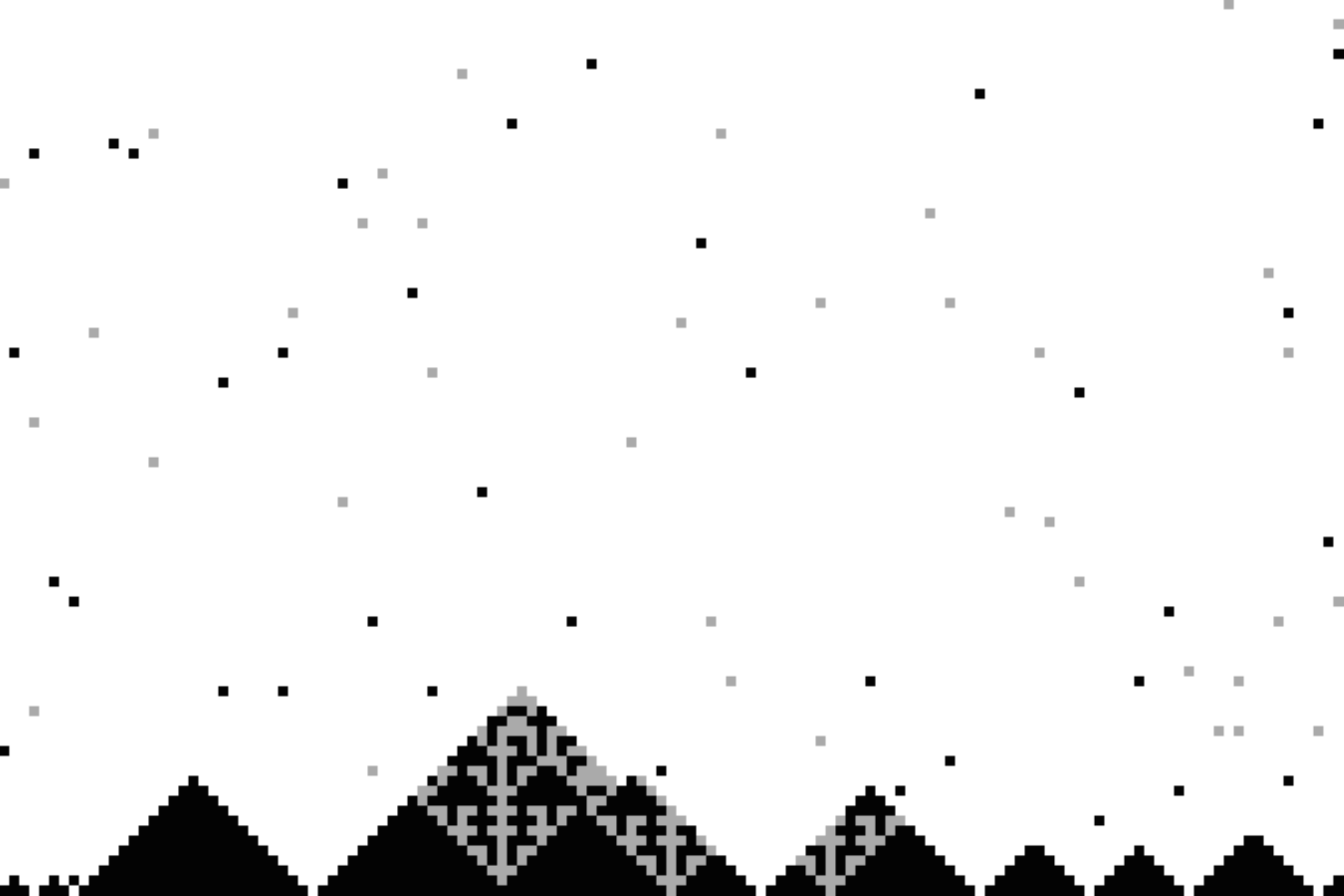}}\\ [1ex]
	&\multicolumn{3}{p{0.92\textwidth}}{The local rule is given by
		$F(x)_i\isdef x_{i-1}x_{i}x_{i+1}\bmod 3$.
	Without noise, fractal patterns can appear but they are unstable because of the spreading symbol.  Noise helps destroying these patterns by introducing the spreading symbol at random positions evenly distributed on the lattice. Ergodicity is established in Theorems~\ref{thm-spreading-zrmn} and~\ref{thm-spreading}.}
	\\ [3em]
	
	\begin{sideways}Gliders with annihilation (Sec.~\ref{sec:gliders:annihilating})\end{sideways} 
		&\boxed{\includegraphics[width=0.435\textwidth]{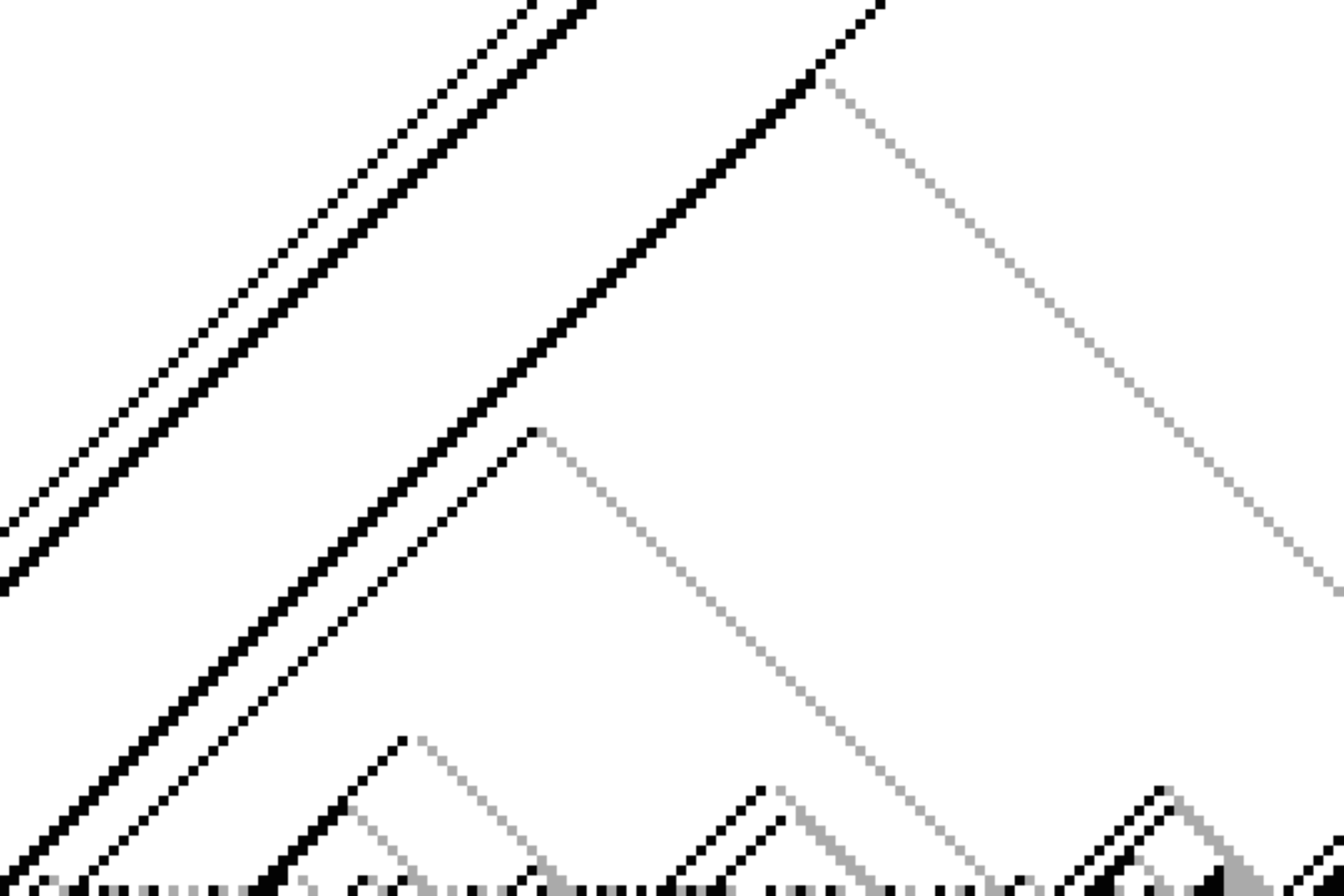}}
		&&\boxed{\includegraphics[width=0.435\textwidth]{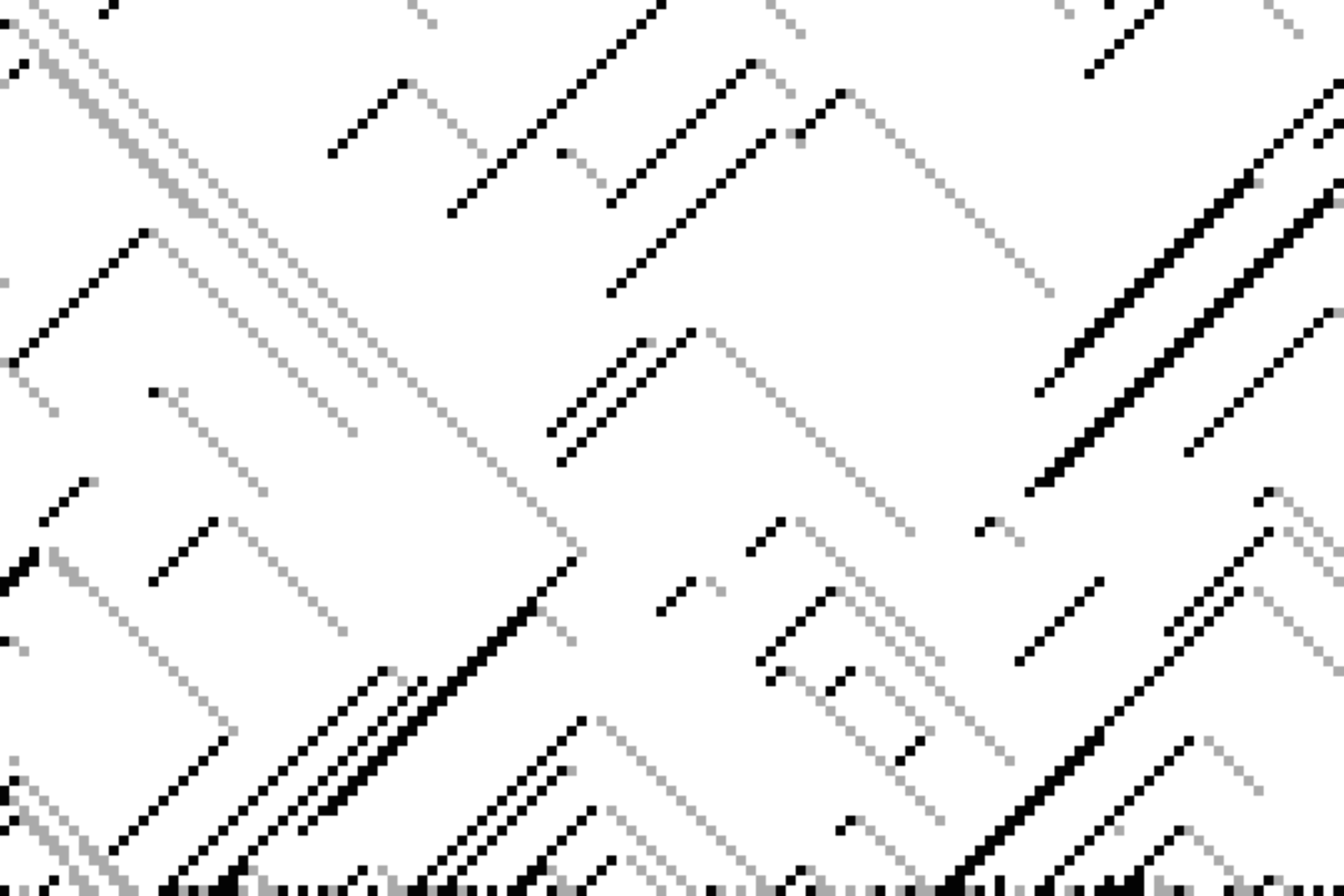}} \\ [1ex]
	&\multicolumn{3}{p{0.92\textwidth}}{This CA consists of particles moving with speed $1$ and $-1$ in an empty background. Two particles moving in opposite directions annihilate upon encounter. Without noise, there are less and less particles as time passes by. Ergodicity in presence of noise is established in Theorems~\ref{thm-glider_an} and~\ref{thm-spreading-bin}.}
\end{tabular}
\end{center}
\end{scriptsize}

\caption{%
	Space-time diagrams of some CA perturbed by a memoryless noise with uniform replacement distribution and error probability $\varepsilon$.  Time goes upwards.}
\label{fig:examples:1}
\end{figure}

\begin{figure}[!ht]
\begin{scriptsize}
\begin{center}
\begin{tabular}{@{}p{0.02\textwidth}@{}p{0.45\textwidth}@{}p{0.05\textwidth}@{}p{0.45\textwidth}@{}}
	& ~~$\varepsilon=0$	&& ~~$\varepsilon=0.01$\\
	
	\begin{sideways}Gliders with walls (Sec.~\ref{sec:gliders:walls})\end{sideways} 
	& \boxed{\includegraphics[width=0.435\textwidth]{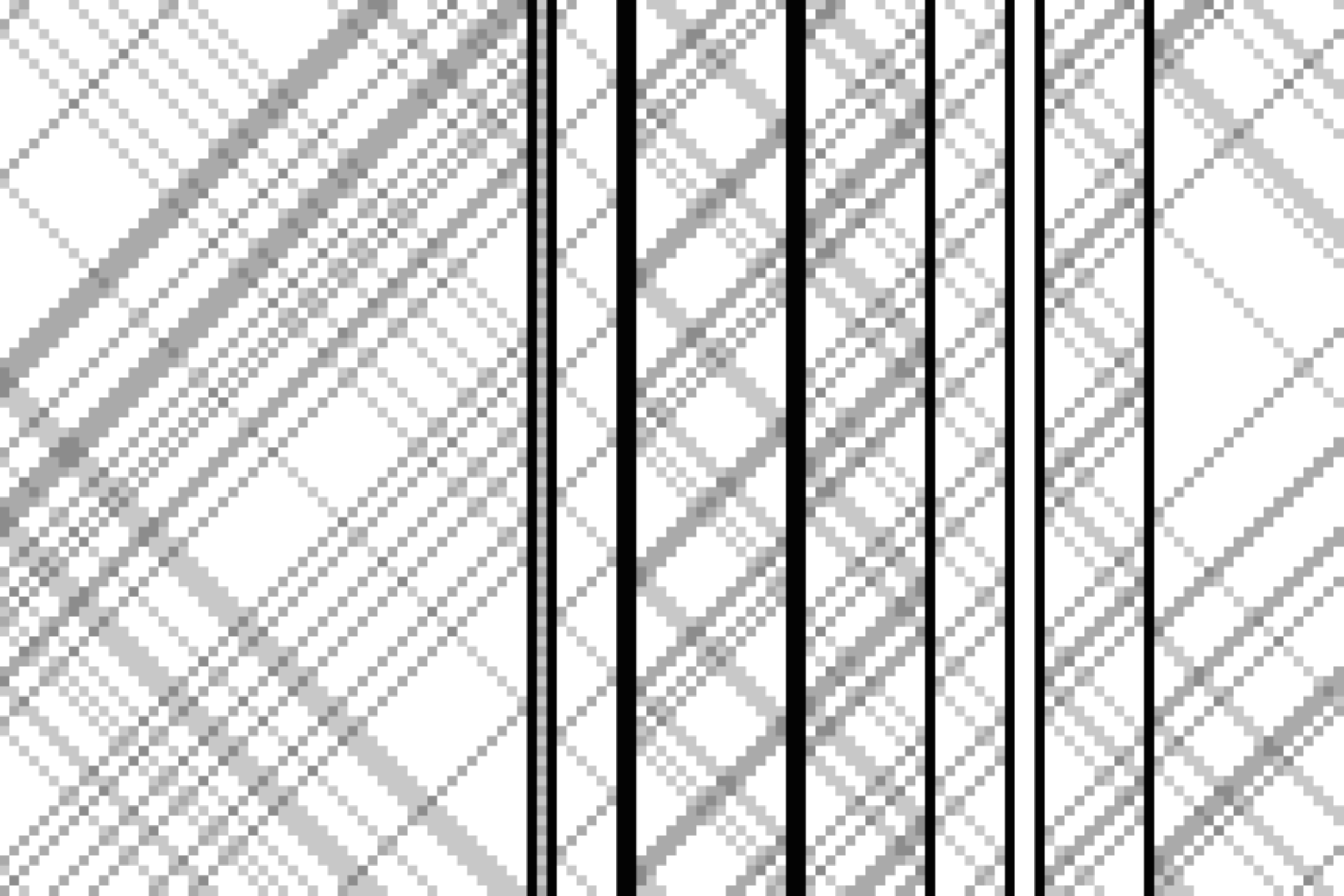}}
	&& \boxed{\includegraphics[width=0.435\textwidth]{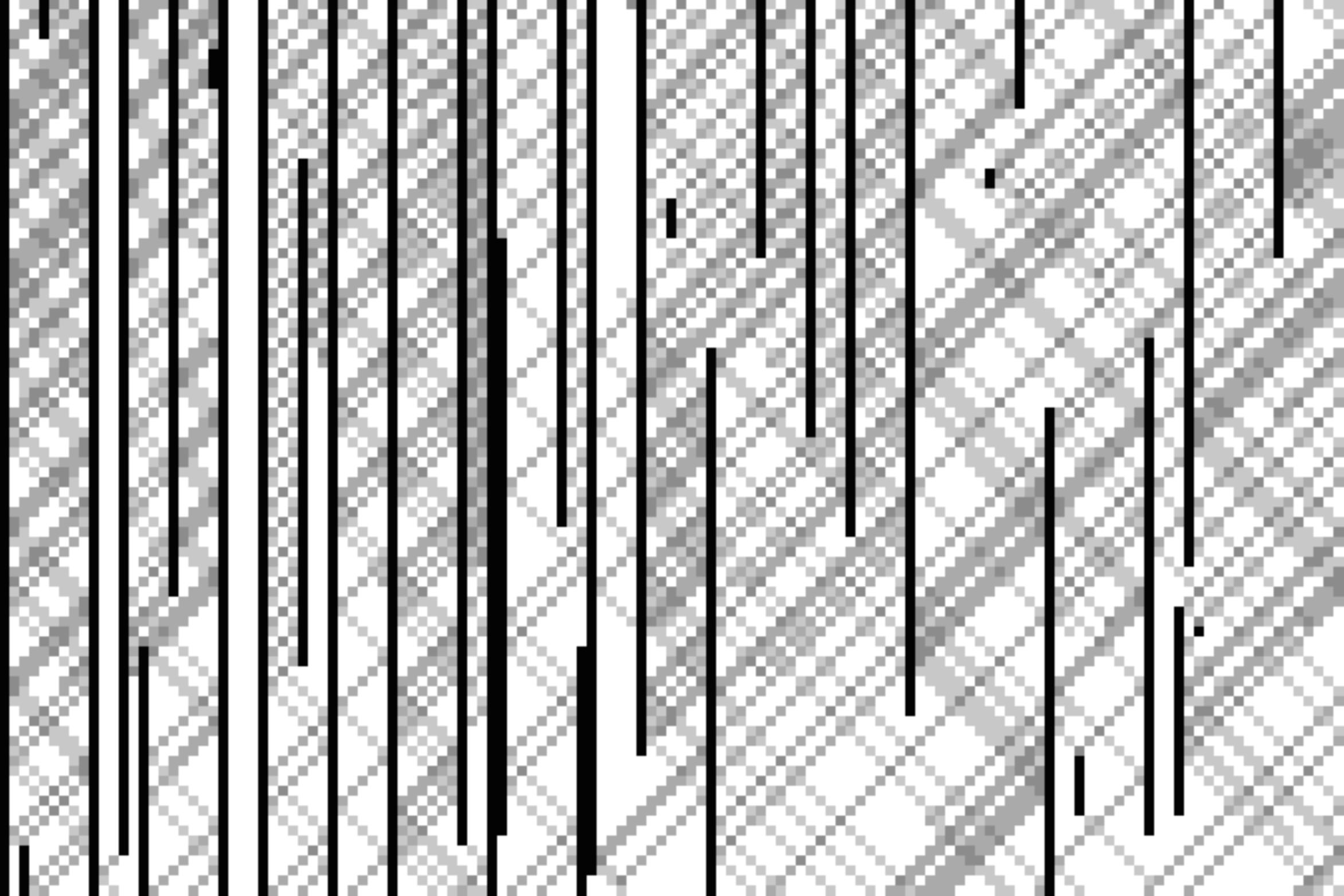}} \\ [1ex]
	&
	\multicolumn{3}{p{0.92\textwidth}}{This CA consists of non-interacting particles moving with constant speed in between walls.  The particles reflect upon hitting the walls.
	Without noise, the behaviour is very regular: the walls are static and the movement of each particle is periodic.  Noise mixes things up. Theorem~\ref{thm-glider_ref} shows the ergodicity. Since the CA is surjective, Theorem~\ref{thm:surj:ergodicity:shift-invariant} also shows the ergodicity ``modulo translations''.  The invariant measure is the uniform measure.}
	\\ [4em]
	
	\begin{sideways}Permutive CA (Sec.~\ref{sec:permutive}) \end{sideways} 
	&\boxed{\includegraphics[width=0.435\textwidth]{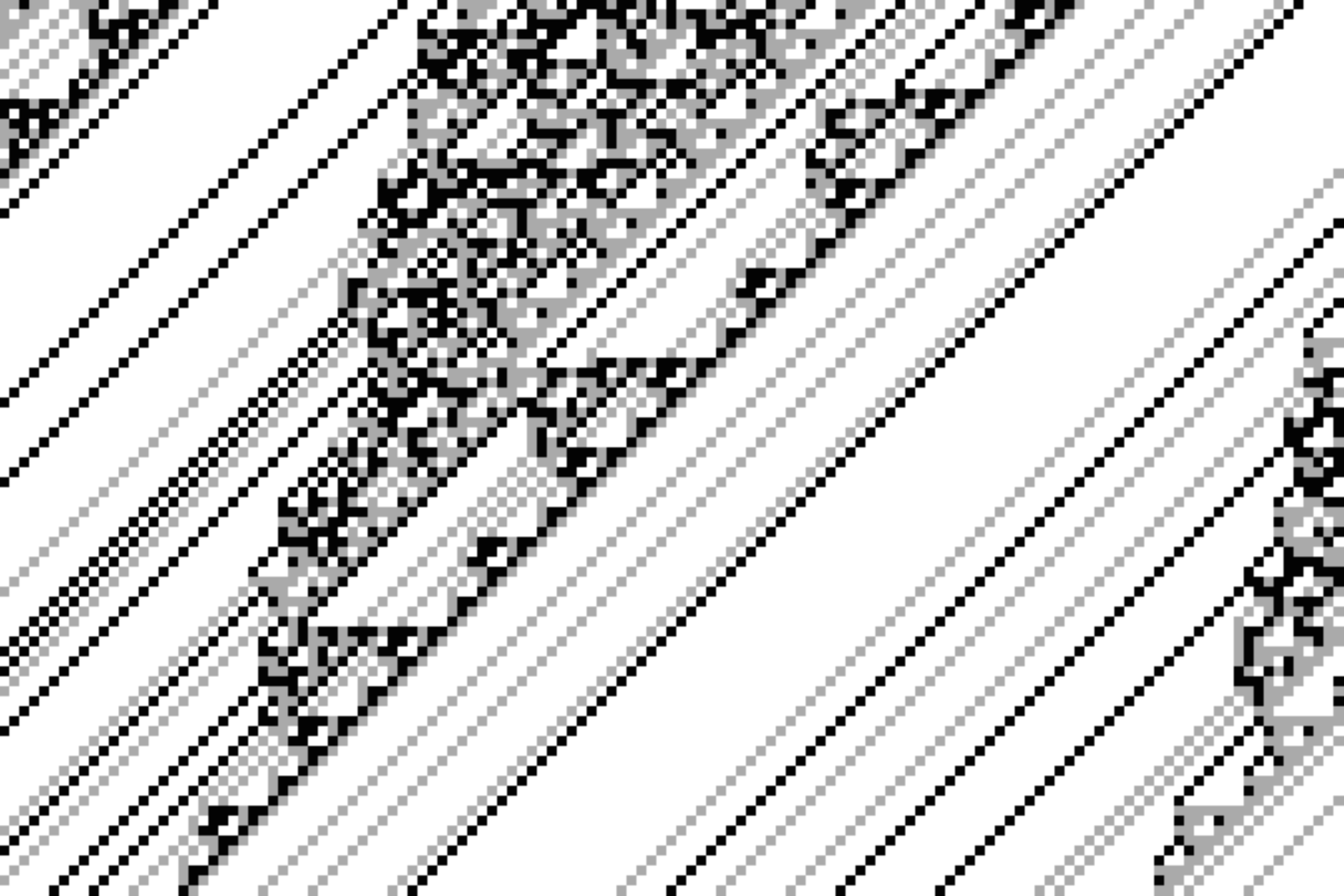}}
	&&\boxed{\includegraphics[width=0.435\textwidth]{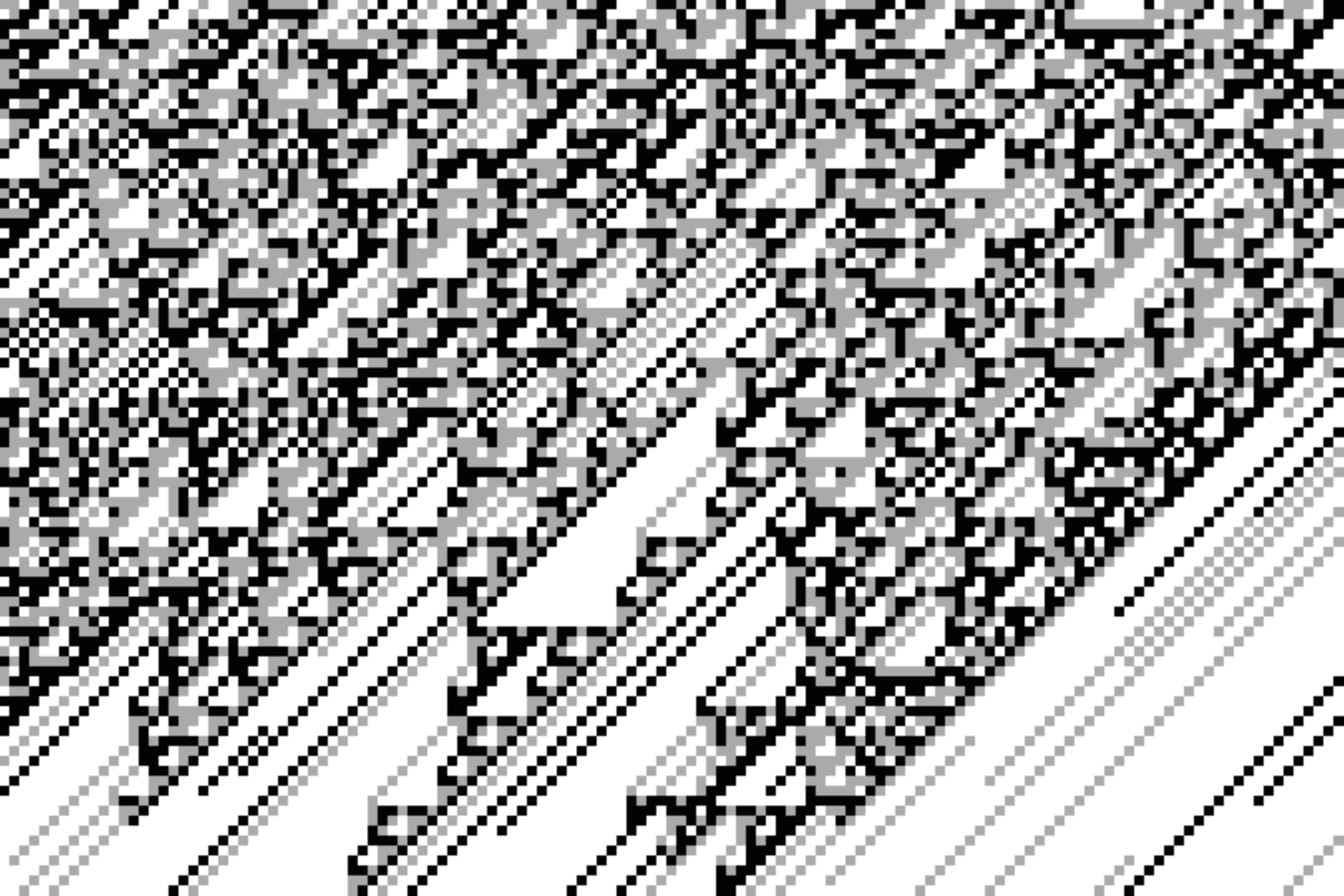}}\\ [1ex]
	&\multicolumn{3}{p{0.92\textwidth}}{The local rule is given by
		$F(x)_i\isdef x_{i-1}+x_{i}\cdot x_{i+1}\bmod 3$.
	The noisy version is ergodic by Theorems~\ref{thm-perm} or~\ref{thm:surj:ergodicity:shift-invariant}.}
	\\ [1em]

	\begin{sideways}Additive CA \end{sideways} 
	&\boxed{\includegraphics[width=0.435\textwidth]{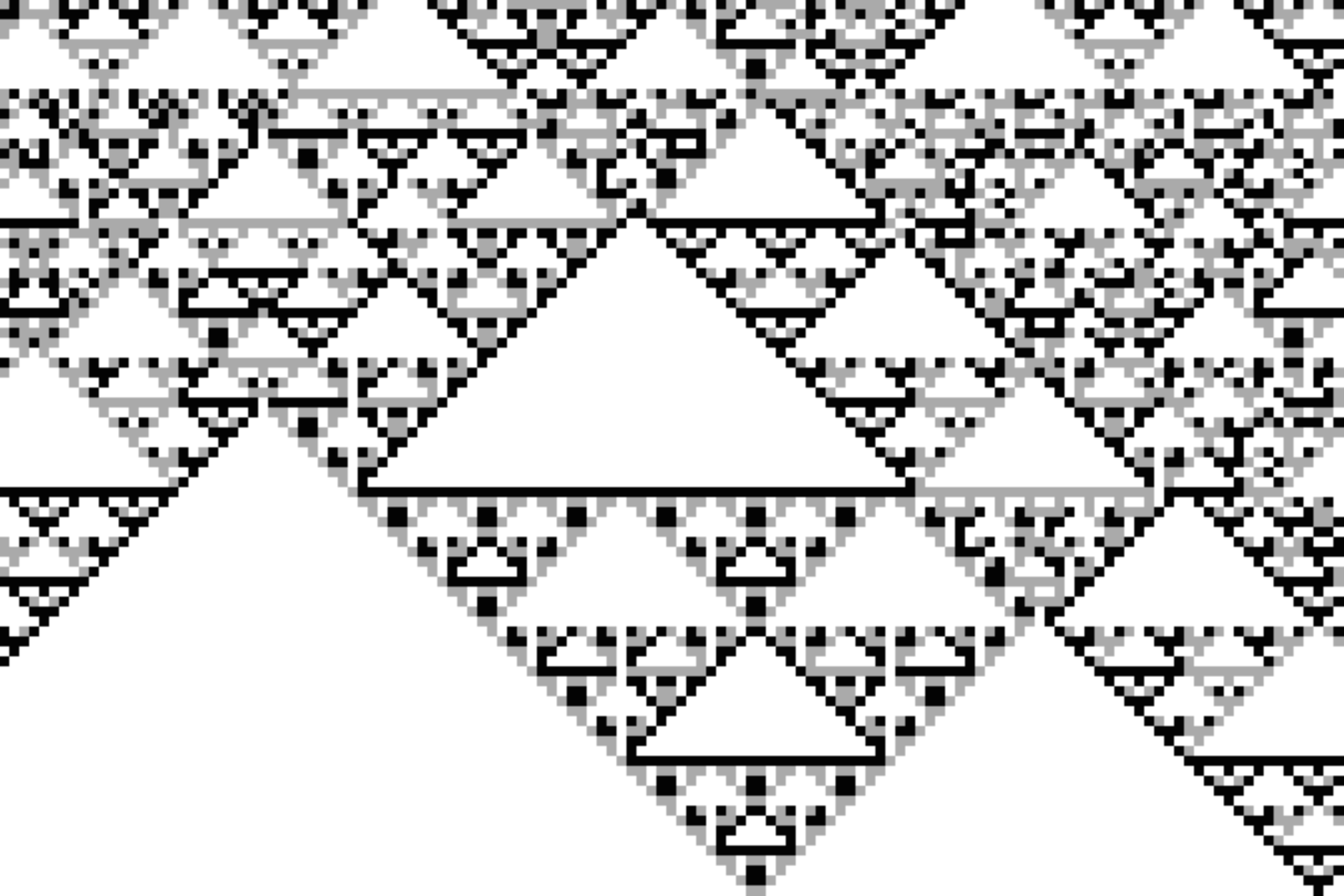}}
	&&\boxed{\includegraphics[width=0.435\textwidth]{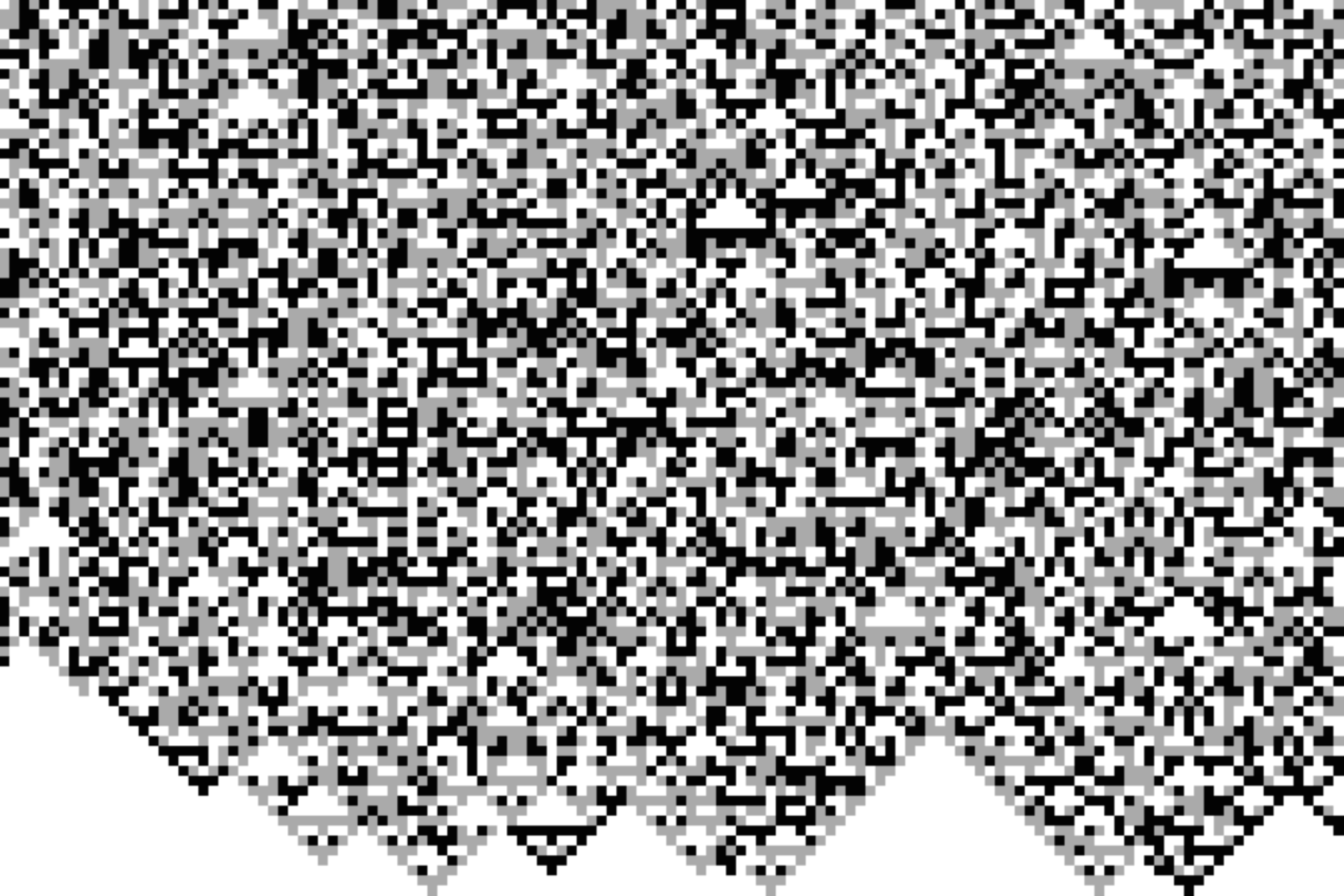}}\\ [1ex]
	&\multicolumn{3}{p{0.92\textwidth}}{The local rule is given by
		$F(x)_i\isdef x_{i-1}+x_{i}+x_{i+1}\mod 3$.
	This CA randomizes its initial condition even in absence of noise:
	starting from a sufficiently random configuration, its distribution converges to the uniform Bernoulli measure (see e.g.~\cite{Lin84}).
	The ergodicity of the noise version is given by Theorem~\ref{thm-perm} or~\ref{thm:surj:ergodicity:shift-invariant}.  See also Theorem~\ref{thm-xor}.}

\end{tabular}
\end{center}
\end{scriptsize}

\caption{%
	Space-time diagrams of some surjective CA perturbed by a memoryless noise with uniform replacement distribution and error probability $\varepsilon$. Time goes upwards.}
	\label{fig:examples:2}
\end{figure}

\section{Coupling method}
\label{sec:coupling}

Intuitively, a PCA is ergodic if it ``forgets'' its initial condition. In some cases, it is possible to prove ergodicity in a constructive fashion by means of a coupling, that is by running the process simultaneously from different initial conditions using a common source of randomness, and showing that all trajectories eventually merge.

In this section, we use coupling arguments to prove the ergodicity of some classes of noisy CA.  
The arguments for most of the results in this section (Secs.~\ref{sec:envelope}--\ref{sec:spreading},~\ref{sec:gliders:walls}) are based on ``backward'' couplings (a.k.a.\ coupling from the past).  Only in Section~\ref{sec:gliders:annihilating} we use a ``forward'' coupling.  The coupling in the last result (Sec.~\ref{sec:permutive}) is rather different and merges the trajectories only on a finite window.

\subsection{Forward and backward couplings}

Recall that a \emph{coupling} of two probability measures $\mu$ and $\nu$ is simply a pair $(X,Y)$ of random variables defined on the same probability space such that $X$ is distributed according to $\mu$ and $Y$ is distributed according to $\nu$.
Couplings can be used to obtain upper bounds on the total variation distance between two measures.
In the special case where $\mu,\nu\in\M(\X)$ are measures on the configuration space $\X$, the inequality
\begin{align}
\label{eq:inequality:coupling}
	\norm{\mu-\nu}_A &\leq \xPr(X_A\neq Y_A) \;,
\end{align}
holds for every coupling $(X,Y)$ of $\mu$ and $\nu$ and each finite set $A\subseteq\ZZ^d$.
This is known as the \emph{coupling inequality} (see e.g.~\cite{Lin02}).

By a coupling of a PCA $\Phi$ we mean a coupling of two trajectories of $\Phi$, that is, a sequence $(X^t,Y^t)_{t\geq 0}$ where both $(X^t)_{t\geq 0}$ and $(Y^t)_{t\geq 0}$ are distributed according to the evolution of the PCA~$\Phi$.

The following lemma is a basic tool for proving the ergodicity of a PCA.

\begin{lemma}\label{lemma-coupling}
	Let $(X^t,Y^t)_{t\geq 0}$ be a coupling of a PCA $\Phi$.
	Let $\mu\in\M(\X)$ denote the distribution of~$X^0$ and suppose that~$Y^0$ is distributed according to a measure $\pi\in\M(\X)$ that is invariant under $\Phi$.
	Assume that for every $k\in\ZZ^d$, $\P(X^t_k\neq Y^t_k)\to 0$ as $t\to\infty$.
	Then, $(\mu\Phi^t)_{t\geq 0}$ converges weakly to $\pi$.
\end{lemma}
\begin{proof}
	For every finite set $A\subset \Z^d$, we have, by the coupling inequality
	\begin{align}
		\norm{\mu\Phi^t-\pi}_A &\leq \P(X^t_A\neq Y^t_A)\leq\sum_{i\in A}\P(X^t_i\neq Y^t_i) \;,
	\end{align}
	which goes to $0$ as $t\to\infty$, meaning that $(\mu\Phi^t)_{t\geq 0}$ converges weakly to $\pi$.
\end{proof}

Following the same idea, we have the following criterion for uniform ergodicity.

\begin{proposition}
\label{prop:coupling:forward}
	Let $\Phi$ be a PCA.
	Let $\rho(t)$ be a real function with $\rho(t)\to 0$ as $t\to\infty$.
	Suppose that for every two configurations $x,y\in\X$, there is a coupling $(X^t,Y^t)_{t\geq 0}$ of $\Phi$ with $X^0=x$ and $Y^0=y$ such that $\P(X^t_0\neq Y^t_0)\leq \rho(t)$ for all $t\geq 0$.
	Then, the PCA is uniformly ergodic and its unique invariant measure is spatially mixing.
\end{proposition}
\begin{proof}
	Let $\pi$ be an invariant measure for $\Phi$ and $\mu$ any other measure.
	Following the argument of Lemma~\ref{lemma-coupling}, for every two configurations $x,y\in\X$ and each finite set $A\subseteq\ZZ^d$ we get $\norm{\Phi^t(x,\cdot)-\Phi^t(y,\cdot)}_A\allowbreak\leq\abs{A}\rho(t)$.
	Integrating over $x$ with respect to $\mu$ and over $y$ with respect to $\pi$, we find that $\norm{\mu\Phi^t-\pi}_A\leq\abs{A}\rho(t)$.
	Therefore, the PCA is uniformly ergodic with unique invariant measure $\pi$.
	Furthermore, $d_A(t)\leq\abs{A}\rho(t)$ and the spatial mixing of $\pi$ follows from Proposition~\ref{prop:spatial-mixing}.
\end{proof}

One way to couple the evolutions of a given PCA from two different initial configurations is to update the configurations iteratively using a common source of randomness.
Let $\Phi$ be a PCA with local function $\varphi$. An \emph{update function} for $\varphi$ is a function $\u:S^m\times[0,1]\rightarrow S$ such that for all $(a_1,a_2,\ldots,a_m)\in S^m$ and $b\in S$, we have
\begin{align}
	\xPr\big(\u(a_1,a_2,\ldots,a_m;U)=b\big) &= \varphi(a_1,a_2,\ldots,a_m)(b)
\end{align}
whenever $U$ is a random variable uniformly distributed over the unit interval $[0,1]$. 

The update function together with a collection of independent random samples uniformly drawn from $[0,1]$ can be used to simulate the PCA.  Let $\pspace{S}\isdef[0,1]^{\Z^d}$.
Given an update function $\u$, we define the \emph{global update map} $\Psi:\X\times\pspace{S}\to\X$ by
\begin{align}
	\Psi(x,u)_k &\isdef \u(x_{k+n_1},\ldots,x_{k+n_m};u_k) \;.
\end{align}
For $t\geq 1$, we recursively define $\Psi^t:\X\times\pspace{S}^t\to\X$ by $\Psi^1(x;u) \isdef \Psi(x;u)$ and 
\begin{align}
	\Psi^{t+1}(x;u^1,u^2,\ldots,u^{t+1}) &\isdef
		\Psi\big(\Psi^t(x;u^1,u^2,\ldots,u^t), u^{t+1}\big) \\
	&=
		\Psi^t\big(\Psi(x;u^1);u^2,\ldots,u^{t+1}\big) \;.
\end{align}
By construction, when $U\isdef (U_i)_{i\in\ZZ^d}$ is a collection of independent random variables uniformly distributed over $[0,1]$, the configuration $\Psi(x;U)$ is distributed according to measure $\Phi(x,\cdot)$.  More generally, if $U^1,U^2,\ldots,U^t$ are independent random configurations uniformly chosen from $\pspace{S}$, that is, if $(U^n_i)_{i\in\ZZ^d,1\leq n\leq t}$ is a collection of independent random variables uniformly distributed over $[0,1]$, then the sequence
\begin{align}
	x, \Psi^1(x;U^1), \Psi^2(x;U^1,U^2), \ldots, \Psi^t(x;U^1,U^2,\ldots,U^t)
\end{align}
is distributed according to the evolution of $\Phi$ from time $0$ to time $t$ with initial configuration~$x$.

It is sometimes useful to simulate the PCA \emph{from the past}.
Let $(U^n_i)_{i\in\ZZ^d,n\in\NN^-}$ be a collection of independent uniformly distributed random variables chosen from $[0,1]$, where $\NN^-\isdef\{0,-1,-2,\ldots\}$, and write $U^n$ for the collection $(U^n_i)_{i\in\ZZ^d}$.  The value $\Psi^t(x;U^{-t},U^{-t+1},\ldots,U^0)$ can be interpreted as the configuration at time $0$ obtained when simulating the PCA $\Phi$ from configuration $x$ at time $-t$ and using the random samples $(U^n_i)_{i\in\ZZ^d,n\in\NN^-}$.
Let us define
\begin{align}
	p_t(\Phi) &\isdef 
		\xPr\big(\text{the map $x\mapsto\Psi^t(x;U^{-t},U^{-t+1},\ldots,U^0)_0$ is constant}\big) \;.
\end{align}
In words, $p_t(\Phi)$ is the probability that, when we simulate $\Phi$ with configuration $x$ at time $-t$ and using the random samples $(U^n_i)_{i\in\ZZ^d,n\in\NN^-}$, the symbol at the origin at time $0$ is independent of $x$.

The following proposition provides another criterion for uniform ergodicity in terms of $p_t(\Phi)$.  Under the same criterion, one can algorithmically generate a perfect sample from the unique invariant measure of~$\Phi$.  This is an adaptation to PCA of the \emph{coupling-from-the-past} algorithm of Propp and Wilson~\cite{ProWil96}, which is developed in~\cite{BusMaiMar13}.  In the present setting, a \emph{perfect sampling algorithm} for a probability measure $\mu\in\M(\X)$ is an algorithm that, given a finite set $A\subseteq\ZZ^d$ and using an unbounded source of independent random samples uniformly drawn from $[0,1]$, outputs a random pattern $W_A$ such that $\xPr(W_A=w_A)=\mu([w_A])$.\footnote{%
	In general, the transition probabilities of the PCA are arbitrary real numbers and do not have finite presentations.  The sampling algorithm of Proposition~\ref{prop:coupling} also requires access to an (infinite) symbolic presentation of these real numbers.
}

\begin{proposition}\label{prop:coupling}
	Let $\Phi$ be a PCA satisfying $p_t(\Phi)\to 1$ as $t\to\infty$. Then, $\Phi$ is uniformly ergodic.  Furthermore, the unique invariant measure of $\Phi$ is spatially mixing and has a perfect sampling algorithm.
\end{proposition}

\begin{proof}
	Let us imagine simulating the PCA $\Phi$ from time $-t$ in the past up to time $0$, starting from two configurations $X^{-t}$ and $Y^{-t}$.  We can couple the configurations obtained at time $0$ by using a family $U=(U^n_i)_{i\in\ZZ^d,n\in\NN^-}$ of independent uniform random samples from $[0,1]$, and setting $X^0\isdef\Psi^t(X^{-t}; U^{-t},U^{-t+1},\ldots,U^0)$ and $Y^0\isdef\Psi^t(Y^{-t}; U^{-t},U^{-t+1},\ldots,U^0)$.


Take $X^{-t}$ to be a fixed configuration $x$ and choose $Y^{-t}$ at random, independently from $U$, according to an invariant measure $\pi$ of the PCA.
By the coupling inequality, for every finite set $A\subset \Z^d$, we have
\begin{align}
\label{eq:prop:coupling:proof}
	\norm{\pi-\Phi^t(x,\cdot)}_A &\leq \P(X^0_A\neq Y^0_A)
		\leq \sum_{i\in A}\P(X^0_i\neq Y^0_i)\leq \abs{A} \big(1-p_t(\Phi)\big) \;.
\end{align}
Since $x$ is arbitrary and $p_t(\Phi)\to 1$ as $t\to\infty$, it follows that $\Phi$ is uniformly ergodic with unique invariant measure~$\pi$.
Furthermore, from~\eqref{eq:prop:coupling:proof} we get $d_A(t)\leq\abs{A}\big(1-p_t(\Phi)\big)$.  Therefore, the conditions of Proposition~\ref{prop:spatial-mixing} are satisfied and $\pi$ is spatially mixing.



Let us now present a perfect sampling algorithm for the unique invariant measure $\pi$ of $\Phi$. We assume that we have access to a family $U=(U^n_i)_{i\in\ZZ^d,n\in\NN^-}$ of independent uniform random samples from $[0,1]$. Let $A$ be a finite subset of $\Z^d$. 
Since $p_t(\Phi)\to 1$ as $t\to\infty$, we know that almost surely, there exists an integer $T\geq 1$ depending on $U$, such that the map $x\mapsto \Psi^T(x; U^{-T},U^{-T+1},\ldots,U^0)_A$ is constant. This constant is distributed exactly according to $\pi$. More specifically, for a finite pattern $w\in S^A,$ the probability that $\Psi^T(x; U^{-T},U^{-T+1},\ldots,U^0)_A=w$ is exactly $\pi([w])$. Furthermore, since $\Psi^t(x,U^{-t},U^{-t+1},\ldots,U^0)_A$ depends only on $x_{A+\Neighb^t}$ and on $(U^n_i)_{i\in A+\Neighb^{-n},-t<n\leq 0}$, we can indeed check for each $t=1,2,\ldots$ whether the function $x\mapsto\Psi^t(x,U^{-t},U^{-t+1},\ldots,U^0)_A$ is constant or not.
\end{proof}

\subsection{The high-noise regime}
\label{sec:envelope}

In this section, we prove an ergodicity criterion holding in the high-noise regime. In particular, it gives a simple condition ensuring the ergodicity of deterministic CA when perturbed by a high enough zero-range noise (see Prop.~\ref{prop:coupling_hzr} and its two corollaries).

Let $\Phi$ be a PCA with alphabet $S$, neibhourhood $\Neighb=\{n_1,\ldots,n_m\}$ and local rule $\varphi$. In order to prove the ergodicity of $\Phi$ using Proposition~\ref{prop:coupling}, we need to design an update function $\u:S^m\times[0,1]\rightarrow S$ for which the dependence of $\u(a_1,\dots,a_m; u)$ on $(a_1,\dots,a_m)\in S^m$ is weak.  A natural idea is to choose an update function with the property that for every $b\in S$, we have
\begin{align}
	\xPr\big(\text{$\u(a_1,\ldots,a_m;U)=b$ for all $(a_1,\ldots,a_m)\in S^m$}\big) 
		&\geq \min_{a_1,\ldots,a_m\in S} \varphi(a_1,\ldots,a_m)(b)
\end{align}
whenever $U$ is a uniform sample from $[0,1]$.
In that case, with probability at least
\begin{align}
	\sum_{b\in S}\min_{a_1,\ldots,a_m\in S} \varphi(a_1,\dots,a_m)(b) \;,
\end{align}
the knowledge of $(a_1,\dots,a_m)\in S^m$ will not be used for computing the value $\u(a_1,\dots,a_m, U)$. 
The notion of \emph{envelope PCA} pursues this idea and provides a simple ergodicity criterion in the high-noise regime.

Instead of running the PCA from different initial configurations, we define a new PCA on an extended alphabet, containing a symbol $\qm$ representing sites whose values are not known (i.e., which may differ between the different copies) and we run it from a single initial configuration containing only the symbol $\qm$. Each time we are able to make the different copies match on a site, the symbol $\qm$ is replaced by a symbol $b\in S$ on which the different copies agree. 
An evolution of the envelope PCA thus encodes a coupling of different copies of the original PCA, with a symbol $\qm$ denoting sites where the copies disagree. If the density of symbol $\qm$ converges to $0$ when time goes to infinity, it means that the original PCA is forgetting its initial condition, hence it is ergodic.

Let us now go into more details. We introduce a new alphabet $\t{S}=S\cup\{\qm\}$, containing an additional question mark symbol, and we define a partial order on $\t{S}$ by declaring $a\prec\qm$ for every $a\in S$. We say that $a\in S$ is \emph{compatible} with $b\in\t{S}$ if $a\preceq b$. The \emph{envelope} of the PCA $\Phi$ is another PCA $\t{\Phi}$ with alphabet $\t{S}$, neighbourhood $\Neighb$ and local rule $\t{\varphi}:\t{S}^m\times\t{S}\rightarrow [0,1]$ defined by
\begin{align}
	\t{\varphi}(a_1,\ldots, a_m)(b) &\isdef
		\min\big\{\varphi(a'_1,\ldots, a'_m)(b) \, ; \; a'_1\ave a_1,\ldots, a'_m\ave a_m\big\}
\end{align}
for $a_1,\ldots,a_m\in\t{S}$ and $b\in S$,
where the minimum is taken over all $a'_1,\ldots,a'_m$ in $S$. The probability of transition to symbol~$\qm$ is then given by
\begin{align}
	\t{\varphi}(a_1,\ldots, a_m)(\qm) &\isdef 1-\sum_{b\in S}\t{\varphi}(a_1,\ldots, a_m)(b) \;.
\end{align}
From a configuration $x\in \t{S}^{\Z^d}$, the symbol at site $k$ is thus updated to a symbol $b\in S$ with a probability that is the minimum of transition probabilities according to $\Phi$ to symbol $b$, among all possible neighbourhood patterns for site $k$ that are compatible with $x$. With the remaining probability, the symbol at site $k$ is updated to $\qm$.

%

The envelope PCA was introduced in~\cite{BusMaiMar13} as a tool to prove the ergodicity of a PCA and to generate perfect samples from its unique invariant measure.
While it is particularly convenient for the high-noise regime, the envelope PCA
has also been successfully exploited to prove the ergodicity of some models in the low-noise regime~\cite{HolMarMar15}.
Similar ideas have been pursued by others~\cite{Fer91}.
The idea of the envelope PCA is reminiscent of the \emph{minorant} PCA introduced by Toom et al.~\cite[Chap.~3]{TooVasStaMitKurPir90}, 
which can be used in a more or less similar fashion to prove ergodicity in the high-noise regime.

The following corollary of Proposition~\ref{prop:coupling}, gives a sufficient condition for ergodicity in terms of the envelope PCA.  

\begin{lemma}
	Suppose that the density 
	$\t{\Phi}^t\big(\qm^{\Z^d},[\qm]\big)$
	of symbols $\qm$ at time $t$ starting from the initial configuration $\qm^{\Z^d}$ converges to $0$ as $t\rightarrow\infty$.  Then, the PCA $\Phi$ is uniformly ergodic, and its unique invariant measure is spatially mixing and admits a perfect sampling algorithm.
\end{lemma}

The fact that the symbol $\qm$ dies out is equivalent to the ergodicity of the envelope PCA $\t{\Phi}$, but the ergodicity of the original PCA $\Phi$ does not in general imply the ergodicity of $\t{\Phi}$. When the alphabet has more than two elements, the definition of the envelope PCA can be refined so as to keep more information about the possible values that a question mark symbol represents~\cite{BusMaiMar13}. 

%
In the evolution of the envelope PCA, at each time step, the symbol at a site is updated to $\qm$ only if at least one of its neighbours is in state $\qm$, and in that case, it becomes a $\qm$ with probability at most
\begin{align}
	p_{\qm}(\Phi) &\isdef \t{\varphi}(\qm,\ldots, \qm)(\qm) \\
	&=
		1 - \sum_{b\in S} \min_{a_1,\ldots,a_m\in S}\varphi(a_1,\ldots,a_m)(b) \;.
\end{align}
This quantity measures the dependence of the transition probabilities on the value of the neighbourhood.

Let us consider an oriented graph $G$ describing the dependence relation between the sites in the space-time diagram of the PCA. The vertices of $G$ are the elements of $\Z^d\times\N$,  and there is an edge from $(k,t)$ to $(\ell, t+1)$ if $k\in \ell+\Neighb$. For a given parameter $p\in [0,1]$, the \emph{directed site percolation} on $G$ consists in declaring each site to be \emph{open} with probability $p$ and \emph{closed} otherwise, independently for different sites. One can show that there is a critical value $p_\critical(\Neighb)\in (0,1)$, such that when $p<p_\critical(\Neighb)$, there is almost surely no infinite open (oriented) path. By comparison with a branching process, one can easily show that $p_\critical(\Neighb)\geq 1/\abs{\Neighb}$.
In one dimension, the value of $p_\critical(\Neighb)$ is known to be in $[\nicefrac{2}{3},\nicefrac{3}{4}]$ when $\Neighb=\{0,1\}$ and in $[\nicefrac{1}{2},\nicefrac{3}{4}]$ when $\Neighb=\{-1,0,1\}$ (see~\cite{PeaFle05}). 

By dominating the appearances of symbol $\qm$ in the space-time diagram of the envelope PCA by a directed site percolation with parameter $p_{\qm}(\Phi)$, one proves that when $p_{\qm}(\Phi)<p_\critical(\Neighb)$, the symbol $\qm$ dies out.

\begin{theorem}\label{thm-high}
	Let $\Phi$ be a PCA with neighbourhood $\Neighb$,
	and let $p_\critical(\Neighb)$ denote the critical value of the $(d+1)$-dimensional directed site percolation with neighbourhood $\Neighb$. If $p_{\qm}(\Phi)<p_\critical(\Neighb)$, 
then 
the PCA $\Phi$ is uniformly ergodic, and the unique invariant measure of $\Phi$ is spatially mixing and admits a perfect sampling algorithm.
\end{theorem}

As a consequence, we obtain the following proposition, and the two corollaries that follow from it.

\begin{proposition}\label{prop:coupling_hzr}
	Let $F$ be a deterministic CA with alphabet $S$ and neighbourhood $\Neighb$, and let $\theta$ be the transition matrix of a zero-range noise.  If
	\begin{align}
		\sum_{b\in S}\min_{a\in S} \theta(a,b) &> 1-p_\critical(\Neighb) \;,
	\end{align}
	then the noisy version of $F$ with noise $\theta$ is uniformly ergodic.
	Furthermore, the unique invariant measure in that case is spatially mixing and admits a perfect sampling algorithm.
\end{proposition}
\begin{proof}
	The noisy version of $F$ with noise $\theta$ satisfies $p_{\qm}\leq 1-\sum_{b\in S}\min_{a\in S} \theta(a,b)$. 
\end{proof}

\begin{corollary}
	Let $F$ be a deterministic CA with neighbourhood $\Neighb$, and let $\theta$ be a memoryless noise with error probability $\varepsilon$. If $\varepsilon>1-p_\critical(\Neighb)$, then the noisy version of $F$ with noise $\theta$ is uniformly ergodic, and has an invariant measure that is spatially mixing and which admits a perfect sampling algorithm.
\end{corollary}
\begin{proof}
	Let $q$ be the replacement distribution of the noise so that
	$\theta(a,b)=(1-\varepsilon)\delta_a(b)+\varepsilon q(b)$.
	Then, $\min_{a\in S}\theta(a,b)=\varepsilon q(b)$ and
	the claim follows from Proposition~\ref{prop:coupling_hzr}.
\end{proof}

\begin{corollary} Let $F$ be a deterministic CA with binary symbol set $S=\{\qO,\qX\}$ and neighbourhood $\Neighb$, and let $\theta$ be a zero-range noise. If $\abs{\theta(\qO,\qX)-\theta(\qX,\qX)}<p_\critical(\Neighb)$, then the noisy version of $F$ with noise $\theta$ is uniformly ergodic, and has an invariant measure that is spatially mixing and which admits a perfect sampling algorithm.
\end{corollary}
\begin{proof}
	In this case, we have 
	\begin{align}
		1-\mathop{\smash[b]{\sum_{b\in S}}}\mathop{\smash[b]{\min_{a\in S}}} \theta(a,b) &=
			1-\min\{\theta(\qO,\qO),\theta(\qX,\qO)\}-\min\{\theta(\qO,\qX),\theta(\qX,\qX)\}\\
		&=
			\max\{\theta(\qO,\qX),\theta(\qX,\qX)\}-\min\{\theta(\qO,\qX),\theta(\qX,\qX)\}\\
		&=
			\abs{\theta(\qO,\qX)-\theta(\qX,\qX)} \;,
	\end{align}
	thus the claim follows from Proposition~\ref{prop:coupling_hzr}.
\end{proof}

%
%
%

\subsection{Small perturbations of nilpotent CA}
\label{sec:nilpotent}

A CA $F$ is \emph{nilpotent} if there is a non-negative integer $N$ such that $F^N$ is a constant function.  Clearly, the unique value of $F^N$ has to be a configuration $\alpha^{\Z^d}$ with the same symbol $\alpha\in S$ at each site.
Observe that the NEC-majority CA (Example~\ref{exp:toom}) is not nilpotent, for it has two distinct fixed points.

Without noise, a nilpotent CA ``forgets'' its initial configuration in a finite number of steps.  It is therefore hard to imagine that adding noise could keep the CA from forgetting its initial configuration.  On the other hand, the envelope PCA introduced in the previous section is not directly applicable to prove the ergodicity of the noisy CA.
Indeed, suppose that $F$ is nilpotent.
If $F$ itself is not a constant function, then for an $\varepsilon$-perturbation of $F$ with small $\varepsilon$, the value $p_{\qm}$ is close to $1$, hence Theorem.~\ref{thm-high}
does not say anything about the ergodicity of such perturbations of $F$. 
Nevertheless, the ergodicity can still be shown using a different coupling-from-the-past argument.

\begin{theorem}\label{thm-nilp}
	Let $F$ be a nilpotent CA. There exists $\varepsilon_\critical>0$ such that for $\varepsilon<\varepsilon_\critical$, every  $\varepsilon$-perturbation of $F$ is uniformly ergodic.
	Furthermore, the unique invariant measure of such a perturbation is spatially mixing and admits a perfect sampling algorithm.
\end{theorem}

\begin{proof} Let $\varepsilon>0$, and let $\Phi$ be an $\varepsilon$-perturbation of $F$. We prove that if $\varepsilon$ is small enough, we can couple all the trajectories of $\Phi$. 

Since $\Phi$ is an $\varepsilon$-perturbation of $F$, its local rule can be written as
\begin{align}
	\varphi(a_1,\ldots,a_m)(b) &=
		(1-\varepsilon)\delta_{f(a_1,\ldots,a_m)}(b)
			+ \varepsilon\tilde{\varphi}(a_1,\ldots,a_m)(b)
\end{align}
where $\delta_{f(a_1,\ldots,a_m)}$ is the distribution with unit mass at $f(a_1,\ldots,a_m)$ and $\tilde{\varphi}$ is another local rule.
Let $\u:S^m\times[0,1]\to S$ be an update function for $\varphi$ with the property that $\u(a_1,\ldots,a_m;u)=f(a_1,\ldots,a_m)$ when $u>\varepsilon$; when $u\leq\varepsilon$, the value of $\u(a_1,\ldots,a_m;u)$ could be different from $f(a_1,\ldots,a_m)$.  Thus, if $U$ is a random variable uniformly distributed over $[0,1]$, then the value $\u(a_1,\ldots,a_m;U)$ may disagree with $f(a_1,\ldots,a_m)$ with probability at most $\varepsilon$.

Let $U=(U_k^n)_{k\in\ZZ^d,n\in\NN^-}$ be a collection of independent random samples uniformly drawn from $[0,1]$.  We simulate $\Phi$ from the past using the update function $\u$ and the samples $U$.
Let $K$ be a finite subset of $\ZZ^d$.
We prove that almost surely, there exists a time $T>0$ such that the trajectories from all possible starting configurations at time $-T$ provide the same pattern $X^0_K$ on $K$ at time $0$. In particular, $p_t(\Phi)\to 1$ as $t\to\infty$, and the uniform ergodicity and the spatial mixing of the invariant measure follow from Proposition~\ref{prop:coupling}.


Let $N\geq 1$ be such that $F^N$ is constant.  The value of this constant has to be a configuration $\alpha^{\Z^d}$ with the same symbol $\alpha$ at every site.  Let $\Neighb$ denote the neighbourhood of the local rule of $F$.  Consider the following subset of the space-time $\ZZ^d\times\NN^-$:
\begin{align}
	W &\isdef \{(\ell,-i) : \text{$0\leq i\leq N-1$ and $\ell\in\Neighb^i$ }\} \;.
\end{align}
We say that an \emph{error} has occurred at position $(k,-t)$ if $U^{-t}_k\leq\varepsilon$. Since $F^N$ is a constant function, if the set $(k,-t)+W$ contains no error, then we know that $X^{-t}_k=\alpha$.

For $k\in \Z^d$ and $t\geq 0$, let us define the random set
\begin{align}
	E(k,-t) &\isdef
		\begin{cases}
			\{(k,-t)+(m,-N): m\in\Neighb^N\} & \text{if $(k,-t)+W$ contains an error,} \\
			\varnothing & \text{otherwise.}
		\end{cases}
\end{align}
We recursively define a sequence of sets $A_0, A_1, \ldots$ by setting $A_0\isdef K\times\{0\}$ and
\begin{align}
	A_{i+1}\isdef E(A_i) &= \bigcup_{(k,-t)\in A_i} E(k,-t) 
\end{align}
for $i\geq 0$.
Clearly, $t=iN$ for every $(k,-t)\in A_i$.
Observe that if $A_i=\varnothing$, then running the simulation from time $-iN$ till $0$ using the samples in $U$ will lead to a pattern $X^0_K$ on $K$ at time $0$ that does not depend on the choice of the configuration $X^{-iN}$ at time $-iN$ (see Fig.~\ref{fig:nilpotent}).

\begin{figure}[h]
\begin{center}
\begin{tikzpicture}[scale=1.5]
\foreach \y in {0,1,2,3} \draw[ultra thin] (-4,-\y) --(4,-\y) ;
\node[left] at (-4.5,0) {$t=0$} ;
\node[left] at (-4.5,-1) {$t=-N$} ;
\node[left] at (-4.5,-2) {$t=-2N$} ;
\node[left] at (-4.5,-3) {$t=-3N$} ;

\draw[blue, ultra thick] (-1,0) -- (-0.25,0) ;
\draw[blue, ultra thick] (0.75,0) -- (1,0) ; 
\draw[decorate,decoration={snake,amplitude=1pt,segment length=5pt},ultra thick] (-0.25,0) -- (0.75,0) ;

\draw (-1,0)--(-2,-1) ;
\draw (1,0)--(2,-1) ;

\draw[dashed] (0.75,0)--(-0.25,-1) ;
\draw[dashed] (0.75,0)--(1.75,-1) ;
\draw[dashed] (-0.25,0)--(-1.25,-1) ;
\draw[dashed] (-0.25,0)--(0.75,-1) ;

\draw[dashed] (-0.25,0)--(-1.25,-1) ;
\draw[dashed] (-0.25,0)--(0.75,-1) ;

\draw[red, fill=red] (0.25,-0.5) circle [radius=0.05]  ;

\draw[blue, ultra thick] (-0.75,-1)--(1.75,-1) ;
\draw[decorate,decoration={snake,amplitude=1pt,segment length=5pt},ultra thick] (-1.25,-1)--(-0.75,-1) ;

\draw (-1.25,-1)--(-2.25,-2) ;
\draw (1.75,-1)--(2.75,-2) ;

\draw[dashed] (-0.75,-1)--(-1.75,-2) ;
\draw[dashed] (-0.75,-1)--(0.25,-2) ;
\draw[dashed] (-0.95,-1)--(-1.95,-2) ;
\draw[dashed] (-0.95,-1)--(0.05,-2) ;

\draw[red, fill=red] (-1.1,-1.35)  circle [radius=0.05]  ;
\draw[red, fill=red] (-1.4,-1.45) circle [radius=0.05]  ;

\draw[blue, ultra thick] (-2.25,-2)--(0.25,-2) ;

\draw (-2.25,-2)--(-3.25,-3) ;
\draw (0.25,-2)--(1.25,-3) ;

\node[above] at (0,0) {$A_0=K\times\{0\}$} ; 
\node[below] at (0.25,-1) {$A_1$} ; 
\node[below] at (-1,-2) {$A_2$} ; 
\node[below] at (-1,-3) {$A_3=\varnothing$} ; 

\end{tikzpicture}
\end{center}
\caption{Illustration of the the proof of Theorem~\ref{thm-nilp}. Errors are represented by red dots. Blue domains represent sites that are known to be in state $\alpha$ (i.e., there are no errors affecting them in the last $N$ time steps). Black (wavy) domains represent sites for which further information from the past might be needed to determine their states.}
\label{fig:nilpotent}
\end{figure}
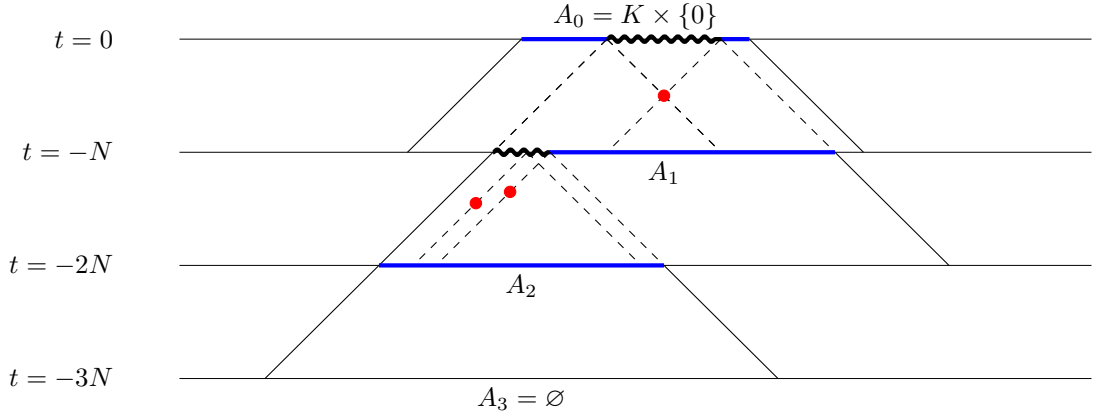

It remains to prove that if $\varepsilon$ is small enough, then almost surely, there exists an integer after which all the sets $A_i$ are empty.

We set $m_i\isdef\abs{\Neighb^i}$.
If there is an error inside $(k,-iN)+W$, then $\abs{E(k,-iN)}=m_N$.
Let $(\ell,-t)$ be a space-time position with $t=iN+j$ and $0\leq j\leq N-1$.  Then, we have $(\ell,-t)\in (k,-iN)+W$ if and only if $k\in \ell-\Neighb^j$. Thus, the number of points $(k,-iN)$ such that $(\ell,-t)$ is in $(k,-iN)+W$ is bounded by $m_j\leq m_{N-1}$.  It follows that an error at $(\ell,-t)$ has a contribution of at most $L\isdef m_{N-1}m_N$ points to $A_{i+1}$.


Let $M\isdef m_0+m_1+\ldots+m_{N-1}$, so that $\abs{W}=M$.
The number of points in $\bigcup_{(k,-iN)\in A_i} (k,-iN)+W$ is thus smaller than $\abs{A_i}\times M$, and an error occurs at each point independently with probability $\varepsilon$. Consequently, $\abs{A_{i+1}}$ is bounded by the sum of at most $\abs{A_i}\times M$ independent random variables, each taking value $L$ with probability $\varepsilon$, and $0$ with probability $1-\varepsilon$. If $\varepsilon<1/(LM)$, a comparison with a branching process shows that there is extinction: almost surely, the sets $A_i$ are eventually empty.
The claim follows.
\end{proof}

Let us remark that the bound given for $\varepsilon$ in the above proof is rough and can certainly be improved.
%

\subsection{CA with a spreading state}
\label{sec:spreading}

Let $F$ be a deterministic CA with symbol set $S$ and neighbourhood $\Neighb$. We say that a symbol $\alpha\in S$ is \emph{spreading} under $F$ if $\abs{\Neighb}\geq 2$ and
$F(x)_k=\alpha$ whenever $x_{k+n}=\alpha$ for some $n\in\Neighb$.
By definition, a CA can have at most one spreading symbol.
For comparison, let us note that in Toom's NEC-majority CA (Example~\ref{exp:toom}), neither of the two symbols $\symb{0}$ and $\symb{1}$ is spreading.
Here, we prove the ergodicity of perturbations of a CA with a spreading symbol for two classes of perturbations.  Another class of perturbations is treated in Section~\ref{sec:mobius:binary}, under the extra assumption that the alphabet is binary.

\subsubsection{Memoryless noise}

Consider a memoryless noise $\theta$ with error probability $\varepsilon$ and replacement distribution $q$, so that $\theta(a,b)=(1-\varepsilon)\delta_a(b)+\varepsilon q(b)$.
We say that the noise is \emph{$\alpha$-positive} if $q(\alpha)>0$.

\begin{theorem}\label{thm-spreading-zrmn} Let $F$ be a CA with spreading state $\alpha$. Then, every perturbation of $F$ by an $\alpha$-positive memoryless noise is uniformly ergodic.
Furthermore, the unique invariant measure of the perturbation is spatially mixing and admits a perfect sampling algorithm.
\end{theorem}

The proof we propose below has the same flavour as the one of Theorem~\ref{thm-nilp} for nilpotent CA, and uses the idea of coupling from the past. Observe however that unlike for nilpotent CA, in some sense, the errors that are introduced here by the random noise favour ergodicity.

\begin{proof} Let $\Phi$ be a perturbation of $F$ by a memoryless noise, defined by the matrix $\theta(a,b)=(1-\varepsilon)\delta_a(b)+\varepsilon q(b)$, where $\varepsilon>0$ and $q(\alpha)>0$. 

Let $\qu:[0,1]\to S$ be a function with the property that if $U$ is a random variable uniformly distributed over $[0,1]$, then $\xPr(\qu(U)=b)=q(b)$.
We use an update function $\u:S^m\times[0,1]\to S$ for $\Phi$ defined by
\begin{align}
	\u(a_1,\ldots,a_m;u) &=
		\begin{cases}
			\qu(u/\varepsilon)		& \text{if $u\leq\varepsilon$,} \\
			f(a_1,\ldots,a_m)		& \text{otherwise,}
		\end{cases}
\end{align}
where $f$ denotes the local rule of $F$.
Observe that if $U$ is a random variable uniformly distributed over $[0,1]$, then $\xPr\big(\u(a_1,\ldots,a_m;U)=b\;\big|\,U\leq\varepsilon\big)=q(b)$ and $\xPr\big(\u(a_1,\ldots,a_m;U)=f(a_1,\ldots,a_m)\;\big|\,U>\varepsilon\big)=1$.

As in the proof of Theorem~\ref{thm-nilp},
we let $U=(U_k^n)_{k\in\ZZ^d,n\in\NN^-}$ be a collection of independent random samples uniformly drawn from $[0,1]$.  We simulate $\Phi$ from the past using the update function $\u$ and the samples $U$.
We prove that almost surely, there exists a time $T>0$ such that the trajectories from all possible starting configurations at time $-T$ provide the same value $X^0_0$ for site $0$ at time $0$. It follows that $p_t(\Phi)\to 1$ as $t\to\infty$, and the uniform ergodicity of $\Phi$ and the spatial mixing of its invariant measure follow from Proposition~\ref{prop:coupling}.


We say that an \emph{error} has occurred at space-time position $(k,-t)$ if $U_{k}^{-t}\leq\varepsilon$. By construction, we know that if there is an error at position $(k,-t)$, then the value $X_{k}^{-t}$ does not depend on the past: it is only a function of $U_{k}^{-t}$.

For $k\in \Z^d$ and $t\geq 0$, let us define the set
\begin{align}
	E(k,-t) &=
		\begin{cases}
			\{(k+m,-t-1): m\in\Neighb\}		& \text{if there is no error at position $(k,-t)$,} \\
			\varnothing						& \text{otherwise.}
		\end{cases}
\end{align}
We recursively define sets $A_0,A_1,\ldots$ by setting $A_0\isdef\{(0,0)\}$ and
\begin{align}
	A_{i+1} &\isdef E(A_i)=\bigcup_{(k,-t)\in A_i} E(k,-t)
\end{align}
for $i\geq 0$.
The set $A=\bigcup_{i\geq 0} A_i$ can be seen as an oriented tree, that is, a directed acyclic graph with edges from each $(k,-t)\in A$ to the points of $E(k,-t)$. Observe that a point $(k,-t)\in A$ is a leaf of the tree if and only if there is an error at position $(k,-t)$. 

Now, let us distinguish two cases (see Fig.~\ref{fig:spreading}):
\begin{enumerate}[label={\rm (\Roman*)}]
	\item The tree $A$ is finite. In this case, there exists an integer $T\geq 0$ such that $A_T=\varnothing$ (hence $A_i=\varnothing$ for all $i\geq T$), and the value $X_{0}^0$ is only a function of the finite family of samples $U_{k}^{-t}$ with $k \in \Neighb^t$ and $0\leq t\leq T-1$.
	\item The tree $A$ is infinite. In this case, almost surely the tree contains an infinite number of leaves. Indeed, each point $(k,-t)$ is an error with probability $\varepsilon$, independently for different points. Furthermore, conditioned on the event that $(k,-t)$ is a leaf, the symbol $X_{k}^{-t}$ takes value $\alpha$ with probability $q(\alpha)>0$, independently for different leaves. Thus, almost surely, the tree $A$ contains at least one leave labeled by the symbol $\alpha$, at some time $-T$. Using the fact that $\alpha$ is a spreading symbol, we can then trace the tree up to time $0$ to find that $X_0^0=\alpha$.
\end{enumerate}

\begin{figure}[h]
\begin{center}
\begin{tikzpicture}[scale=1.5]
\begin{scope}
\foreach \y in {0,1,2,3} \draw[ultra thin] (-1.5,-0.5*\y) --(1.5,-0.5*\y) ;
\foreach \x in {0} \node[left,overlay] at (-2,-0.5*\x) {$t=\x$} ;
\foreach \x in {1,2,3} \node[left,overlay] at (-2,-0.5*\x) {$t=-\x$} ;
\node[left] at (-2,-2) {$\vdots$} ;
\node[above] at (0,0) {$X_0^0$} ;
\foreach \x in {0} {\draw[fill] (\x,0) circle [radius=0.05]  ; \draw[dashed] (\x,0)--(\x,-0.5) ; \draw[dashed] (\x,0)--(\x-0.5,-0.5) ; \draw[dashed] (\x,0)--(\x+0.5,-0.5) ; }
\foreach \x in {0,0.5} {\draw[fill] (\x,-0.5) circle [radius=0.02]  ; \draw[dashed] (\x,-0.5)--(\x,-1) ; \draw[dashed] (\x,-0.5)--(\x-0.5,-1) ; \draw[dashed] (\x,-0.5)--(\x+0.5,-1) ; }
\foreach \x in {-0.5} \draw[fill, red] (\x,-0.5) circle [radius=0.05]  ;
\foreach \x in {0,0.5} {\draw[fill] (\x,-1) circle [radius=0.02]  ; \draw[dashed] (\x,-1)--(\x,-1.5) ; \draw[dashed] (\x,-1)--(\x-0.5,-1.5) ; \draw[dashed] (\x,-1)--(\x+0.5,-1.5) ; }
\foreach \x in {-0.5,1} \draw[fill, red] (\x,-1) circle [radius=0.05]  ;
\foreach \x in {0,0.5,1} \draw[fill, red] (\x,-1.5) circle [radius=0.05]  ;
\foreach \x in {-0.5,0,0.5,1} \draw[fill, red] (\x,-1.5) circle [radius=0.05]  ;
\end{scope}
\begin{scope}[xshift=4cm]
\foreach \y in {0,1,2,3} \draw[ultra thin] (-1.5,-0.5*\y) --(1.5,-0.5*\y) ;
\node[above] at (0.2,0) {$X_0^0{\color{red} =\alpha}$} ;
\foreach \x in {0} {\draw[fill] (\x,0) circle [radius=0.05]  ; \draw[dashed] (\x,0)--(\x,-0.5) ; \draw[dashed] (\x,0)--(\x-0.5,-0.5) ; \draw[dashed] (\x,0)--(\x+0.5,-0.5) ; }
\foreach \x in {0,0.5} {\draw[fill] (\x,-0.5) circle [radius=0.02]  ; \draw[dashed] (\x,-0.5)--(\x,-1) ; \draw[dashed] (\x,-0.5)--(\x-0.5,-1) ; \draw[dashed] (\x,-0.5)--(\x+0.5,-1) ; }
\foreach \x in {-0.5} \draw[fill, red] (\x,-0.5) circle [radius=0.05]  ;
\foreach \x in {0,0.5} {\draw[fill] (\x,-1) circle [radius=0.02]  ; \draw[dashed] (\x,-1)--(\x,-1.5) ; \draw[dashed] (\x,-1)--(\x-0.5,-1.5) ; \draw[dashed] (\x,-1)--(\x+0.5,-1.5) ; }
\foreach \x in {-0.5,1} \draw[fill, red] (\x,-1) circle [radius=0.05]  ;
\foreach \x in {-0.5,1} {\draw[fill] (\x,-1.5) circle [radius=0.02]  ; \draw[dashed] (\x,-1.5)--(\x,-2) ; \draw[dashed] (\x,-1.5)--(\x-0.5,-2) ; \draw[dashed] (\x,-1.5)--(\x+0.5,-2) ; }
\foreach \x in {0,0.5} \draw[fill, red] (\x,-1.5) circle [radius=0.05]  ;
\foreach \x in {0.5} \draw[red] (\x,-1.5) circle [radius=0.1]  ;
\node[below] at (0.5,-1.6) {${\color{red} \alpha}$} ;
\end{scope}
\end{tikzpicture}
\end{center}
\caption{Illustration of the proof of Theorem~\ref{thm-spreading-zrmn}. Errors are represented by red dots. In the first case, the tree is finite, and $X_0^0$ is a function of the values given by the memoryless noise $q$ at the errors. In the second case, the tree is infinite: then it contains an infinite number of leaves, and there is almost surely one leaf having value $\alpha$, so that $X_0^0=\alpha$.}
\label{fig:spreading}
\end{figure}
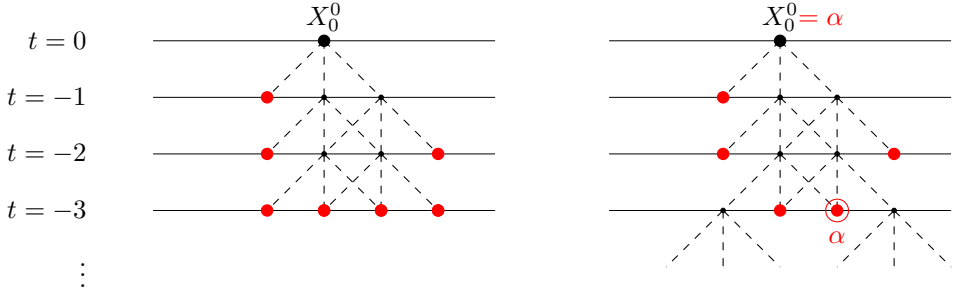

In both cases, the value $X_0^0$ is almost surely uniquely determined by a finite number of samples in the family~$U$.  In particular, almost surely there is a time $T>0$ such that if we simulate the PCA from time $-T$ using the samples in $U$, all possible choices of the configuration $X^{-T}$ lead to the same value $X_0^0$ for site $0$ at time $0$.
\end{proof}

\subsubsection{Positive perturbation}
\label{sec:spreading:perturbation}

In this section, we consider $\varepsilon$-perturbations of a CA $F$ with spreading symbol $\alpha$.
Recall that an $\varepsilon$-perturbation of a CA with local rule $f$ is a PCA whose local rule $\varphi$ satisfies $\varphi(a_1,\dots,a_m)(f(a_1,\dots,a_m))\allowbreak\geq 1-\varepsilon$ for all $a_1,\dots,a_m\in S$. We consider perturbations that are \emph{$\alpha$-positive}, meaning that $\varphi(a_1,\ldots,a_m)(\alpha)\allowbreak>0$ for all $a_1,\ldots,a_m\in S$.

\begin{theorem}\label{thm-spreading} Let $F$ be a one-dimensional CA with neighbourhood $\Neighb=\{0,1\}$ and spreading state~$\alpha$. There exists an $\varepsilon_\critical>0$ such that for $\varepsilon<\varepsilon_\critical$, every $\alpha$-positive $\varepsilon$-perturbation of $F$ is uniformly ergodic, with an invariant measure that is spatially mixing and admits a perfect sampling algorithm.
\end{theorem}


\begin{proof}
	Let $\Phi$ be an $\alpha$-positive $\varepsilon$-perturbation of $F$.
	The local rule of $\Phi$ can be written as
	\begin{align}
		\varphi(a_0,a_1)(b) &= (1-2\varepsilon)\delta_{f(a_0,a_1)}(b) + 2\varepsilon\tilde{\varphi}(a_0,a_1)(b)
	\end{align}
	where $f$ is the local rule of $F$ and $\tilde{\varphi}$ is another local rule.
	We have used $2\varepsilon$ instead of $\varepsilon$ to make sure that $\tilde{\varphi}$ is also $\alpha$-positive.  Let $\delta\isdef 2\varepsilon\cdot\min\{\tilde{\varphi}(a_0,a_1)(\alpha): a_0,a_1\in S\}$ and note that $\delta>0$.  Let $\u:S^2\times[0,1]\to S$ be an update function for $\varphi$ with the property that $\u(a_0,a_1;u)=f(a_0,a_1)$ when $u>2\varepsilon$ and $\u(a_0,a_1;u)=\alpha$ when $u\leq\delta$. 
	
	Let $U\isdef(U_i^n)_{i\in\ZZ,n\in\ZZ}$ be a collection of independent uniform samples from $[0,1]$.
	We use the update function $\u$ and the collection $U$ to simulate $\Phi$ from a time far in the past.  Let $X^t$ denote the configuration at time $t$.

	Since $\alpha$ is a spreading symbol for $F$ and at each space-time point the local rule is applied with probability at least $1-2\varepsilon$, the spread of $\alpha$ dominates an oriented site percolation with parameter $1-2\varepsilon$.  More specifically, consider the ``space-time'' graph with vertex set $\ZZ\times\ZZ$ and oriented edges $(i,n-1)\to(i,n)$ and $(i+1,n-1)\to(i,n)$ for all $i,n\in\ZZ$.  Declare a point $(i,n)$ \emph{open} if $U_i^n>2\varepsilon$ and \emph{closed} otherwise.  The \emph{open cluster} of the point $(0,0)$ is the set $C$ of all points $(i,n)$ that can be reached from $(0,0)$ by an oriented open path.  Clearly, if $X_0^0=\alpha$ (in particular, if $U_0^0<\delta$), then for every point $(i,n)$ in the open cluster of $(0,0)$, we necessarily have $X_i^n=\alpha$.  But even more is true.  Let $Q^n\isdef\{i: (i,n)\in C\}$ be the set of descendants of $(0,0)$ at time $n$ and denote by $L^n\isdef\inf Q^n$ and $R^n\isdef\sup Q^n$ the leftmost and rightmost elements of $Q^n$ (with the convention $\inf\varnothing\isdef+\infty$ and $\sup\varnothing\isdef-\infty$).  Observe that if $X_0^0=\alpha$, then for every $n>0$ and $i$ with $L^n\leq i\leq R^n$, the value $X_i^n$ is uniquely determined by the samples $U_j^m$ with $0<m\leq n$ and $-m\leq j\leq 0$.
	Let us call the set $\overline{C}\isdef\{(i,n): \text{$n>0$ and $L^n\leq i\leq R^n$}\}$ the \emph{cone} of $(0,0)$.  The cone of a point $(k,t)$ is defined in a similar fashion and is denoted by $\overline{C}(k,t)$.
	
	In order to prove ergodicity, we claim that when $\varepsilon$ is small enough (in particular, when $2\varepsilon<(1-p_\critical)^2$, where $p_\critical$ is the critical value for oriented \emph{bond} percolation on $\ZZ\times\ZZ$), the point $(0,0)$ is almost surely in the cone of a point $(k,-t)$ with $U_k^{-t}<\delta$ (see Fig.~\ref{fig:spreading2}).  This implies that $p_t(\Phi)\to 1$ as $t\to\infty$, and the uniform ergodicity of $\Phi$ and the spatial mixing of its invariant measure follow from Proposition~\ref{prop:coupling}.
	
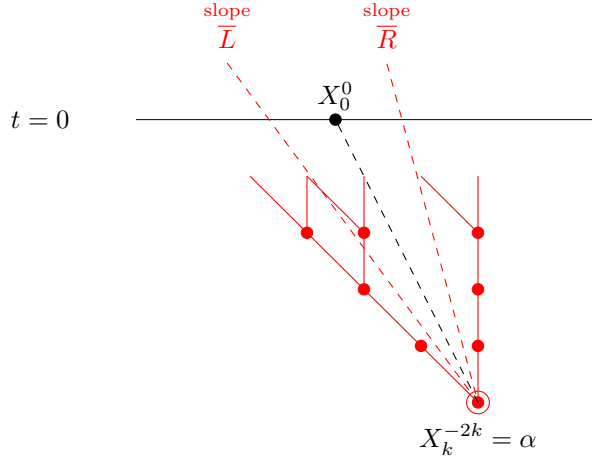
\begin{figure}[h]
	\begin{center}
	\begin{tikzpicture}[scale=1.5]
	\draw[ultra thin] (-5,3) --(-1,3) ;
	\node[left,overlay] at (-5.5,3) {$t=0$} ;
	\foreach \x in {-2} {\draw[fill, red] (\x,0.5) circle [radius=0.05]  ; \draw[red] (\x-0.5,1)--(\x,0.5)--(\x,1);}
	\foreach \x in {-2.5} {\draw[fill, red] (\x,1) circle [radius=0.05]  ; \draw[red] (\x-0.5,1.5)--(\x,1);}
	\foreach \x in {-2} {\draw[fill, red] (\x,1) circle [radius=0.05]  ; \draw[red] (\x,1)--(\x,1.5);}
	\foreach \x in {-3} {\draw[fill, red] (\x,1.5) circle [radius=0.05]  ; \draw[red] (\x-0.5,2)--(\x,1.5)--(\x,2);}
	\foreach \x in {-2} {\draw[fill, red] (\x,1.5) circle [radius=0.05]  ; \draw[red] (\x,1.5)--(\x,2);}
	\foreach \x in {-3.5,-3,-2} {\draw[fill, red] (\x,2) circle [radius=0.05]  ; \draw[red] (\x-0.5,2.5)--(\x,2)--(\x,2.5);}
	\foreach \x in {-2} {\draw[red] (\x,0.5) circle [radius=0.1]  ; \draw[red, dashed] (\x-2.2,3.5)--(\x,0.5)--(\x-0.8,3.5); \draw[dashed] (\x,0.5)--(\x-1.25,3);}

	\node[above] at (-2-2.2,3.55) {{\color{red}$\overline{L}$}} ;	
	\node[above] at (-2-2.2,3.8) {{\color{red}{\scriptsize slope}}} ;	

	\node[above] at (-2-0.8,3.55) {{\color{red}$\overline{R}$}} ;	
	\node[above] at (-2-0.8,3.8) {{\color{red}{\scriptsize slope}}} ;	

%
	
	\node[below] at (-2,0.4) {$X_{k}^{-2k}=\alpha$} ;	
	\node[above] at (-2-1.25,3) {$X_0^0$} ;
	
	\foreach \x in {-2} {\draw[fill] (\x-1.25,3) circle [radius=0.05]  ;}

	\end{tikzpicture}
	\end{center}
	\caption{Illustration of the proof of Theorem.~\ref{thm-spreading}. With probability $1$, the point $(0,0)$ belongs to the cone of a point $(k,-2k)$.
	}
	\label{fig:spreading2}
\end{figure}
	
	To prove the latter claim, we invoke a result of Durrett~\cite[Sec.~3]{Dur84} on oriented bond percolation.  In the oriented bond percolation, each edge of the above-mentioned space-time graph is declared \emph{open} with probability $p$, independently of the other edges.  Observe that when $p=1-\sqrt{2\varepsilon}$, the oriented bond percolation with parameter $p$ and the oriented site percolation with parameter $1-2\varepsilon$ can be coupled in such a way that a point $(i,n)$ is open if and only if at least one of its two incoming edges are open.
	With such a coupling, the open bond-cluster of $(0,0)$ will be included in the open site-cluster of $(0,0)$. 
	Let $\overline{L}\isdef\limsup_{n\to\infty}L^n/n$ and $\overline{R}\isdef\liminf_{n\to\infty}R^n/n$.  It follows from the result of Durrett that when $p>p_\critical$, on the event that the open bond-cluster of $(0,0)$ is infinite, we almost surely have $\overline{L}<-\nicefrac{1}{2}<\overline{R}$.
	
	As a consequence, when $2\varepsilon<(1-p_\critical)^2$, there exists a value $i_0>0$ such that, with positive probability, every point $(-i,2i)$ with $i\geq i_0$ is in the cone of $(0,0)$.  Observe that the cone of $(0,0)$ is independent of the variable $U_0^0$.  Therefore, with positive probability, $U_0^0<\delta$ and every point $(-i,2i)$ with $i\geq i_0$ is in the cone of $(0,0)$.
	Let $E(k,t)$ denote the event that $U_k^t<\delta$ and every point $(k-i,t+2i)$ with $i\geq i_0$ is in the cone of $(k,t)$.
	Since the process $(U_i^n)_{i\in\ZZ,n\in\ZZ}$ is ergodic with respect to the shift along $(-1,2)$, we find that with probability~$1$, the events $E(k,-2k)$ occur for infinitely many $k>0$.  In particular, almost surely, there exists a point $(k,-2k)$ with $k\geq i_0$ for which $U_k^{-2k}<\delta$ and the cone of $(k,-2k)$ includes $(0,0)$.  This concludes the proof.
\end{proof}

The assumption $\Neighb=\{0,1\}$ is not essential, and the proof can be extended to the case where $\Neighb=\{\ell,\ell+1,\ldots,r\}$ is an interval in $\ZZ$.  Extending the result to more general neighbourhoods would require additional technical details.

\subsection{Interacting gliders with birth-death noise}
\label{sec:gliders}

\subsubsection{Gliders with annihilation}
\label{sec:gliders:annihilating}


A gliders CA is a deterministic CA describing the movement of particles of different types according to given velocities.  More specifically, a \emph{gliders} CA with $N\geq 1$ particle types and particle velocities $v_1,\ldots,v_N\in\ZZ^d$ is a CA $G$ with alphabet $S\isdef\{\qO,\qX\}^N$ defined by
\begin{align}
	(G x)_{k,i} &\isdef x_{k+v_i,i}
\end{align}
for every $x\in S^{\ZZ^d}$, $k\in\ZZ^d$ and $i\in\{1,\ldots,N\}$.  Here, $x_{k,i}$ denotes the $i$th component of the symbol at site~$k$ in~$x$, and $x_{k,i}=\qX$ indicates the presence of a particle of type $i$ at site $k$.  Thus, $G$ simply shifts the particles of type $1$ with vector $v_1$, the particles of type $2$ with vector $v_2$ and so forth.  The neighbourhood of $G$ is clearly $\Neighb_G\isdef\{v_1,\ldots,v_N\}$.

For $i\in\{1,\ldots,N\}$, let $\e_i\in \{\qO,\qX\}^N$ denote the symbol representing the presence of a particle of type $i$ and absence of all the other types of particles, that is, $(\e_i)_j\isdef\indicator{i}(j)$. 
An \emph{elementary annihilation} rule is a function $h_{i,j}:S\to S$ defined by
\begin{align}
	h_{i,j}(a) &\isdef
		\begin{cases}
			a -\e_i-\e_j		& \text{if $a_i=a_j=\qX$,} \\
			a					& \text{otherwise.}
		\end{cases}
\end{align}
An \emph{annihilation} rule is a composition  $h\isdef h_{i_n,j_n}\oo\cdots\oo h_{i_1,j_1}$ of elementary annihilation rules.  Observe that elementary annihilation rules may not commute.
An annihilation CA is a CA $A$ with neighbourhood $\Neighb_A\isdef\{0\}$ whose local rule is an annihilation rule.  A \emph{gliders with annihilation} is a composition $F\isdef A\oo G$ of a gliders CA $G$ followed by an annihilation CA $A$.  In words, a gliders with annihilation represents the movement of $N$ types of particles where certain pairs of particles annihilate upon encounter at the same position.  Note that, due to the discrete nature of time, particles moving in opposite directions can possibly pass each other without encountering at the same position.


Recall that a birth-death noise on $S=\{\qO,\qX\}^N$ is a zero-range noise under which particles of different type appear and disappear independently from one another.  The matrix of a birth-death noise can therefore be written as
\begin{align}
	\theta(a,b) &\isdef \prod_{i=1}^n \theta_i(a_i,b_i) \;.
\end{align}
Each matrix $\theta_i$ has the form
\begin{align}
	\theta_i &=
		\begin{pmatrix}
			1-\beta_i	& \beta_i \\
			\delta_i	& 1-\delta_j
		\end{pmatrix} \;,
\end{align}
where $\beta_i\in[0,1]$ and $\delta_i\in[0,1]$ respectively represent the \emph{birth rate} and \emph{death rate} of particles of type~$i$.
A birth-death noise is positive if $\beta_i,\delta_i\in(0,1)$ for each $i\in\{1,\ldots,N\}$.

\begin{figure}[h]
\begin{center}
\begin{tikzpicture}[scale=0.75]
\begin{scope}
\foreach \y in {0} \draw[ultra thin] (-4,-0.5*\y) --(4,-0.5*\y) ;
\draw[red, thick] (-3.5,0)--(4,3.75) ;
\draw[red, thick] (-3,0)--(-2,0.5) ;
\draw[blue,decorate,decoration={zigzag,amplitude=0.5pt,segment length=3pt},thick] (-1.5,0)--(-2,0.5) ;
\draw[green, decorate,decoration={snake,amplitude=0.5pt,segment length=7pt}, thick] (-1,0)--(4,5) ;
\draw[green, decorate,decoration={snake,amplitude=0.5pt,segment length=7pt}, thick] (1,0)--(2,1) ;
\draw[blue,decorate,decoration={zigzag,amplitude=0.5pt,segment length=3pt},thick] (3,0)--(2,1) ;
\node[left,overlay] at (-4.5,0) {$t=0$} ;
\end{scope}
\begin{scope}[xshift=9cm]
\foreach \y in {0} \draw[ultra thin] (-4,-0.5*\y) --(4,-0.5*\y) ;
\draw[red, thick] (-3,0.25)--(4,3.75) ;
\draw[fill, red] (-3,0.25) circle [radius=0.05]  ; 
\draw[red, thick] (-3,0)--(-2,0.5) ;
\draw[blue,decorate,decoration={zigzag,amplitude=0.5pt,segment length=3pt},thick] (-1.5,0)--(-2,0.5) ;
\draw[green,decorate,decoration={snake,amplitude=0.5pt,segment length=7pt}, thick] (-1,0)--(2.5,3.5) ;
\draw[fill, green] (2.5,3.5) circle [radius=0.05]  ; 
\draw[green, decorate,decoration={snake,amplitude=0.5pt,segment length=7pt}, thick] (-3,2)--(-2,3) ;
\draw[green, decorate,decoration={snake,amplitude=0.5pt,segment length=7pt}, thick] (1,0)--(2,1) ;
\draw[fill, green] (-3,2) circle [radius=0.05]  ; 
\draw[fill, green] (-2,3) circle [radius=0.05]  ; 
\draw[blue,decorate,decoration={zigzag,amplitude=0.5pt,segment length=3pt},thick] (2.5,0.5)--(2,1) ;
\draw[fill, blue] (2.5,0.5) circle [radius=0.05]  ; 
\draw[blue,decorate,decoration={zigzag,amplitude=0.5pt,segment length=3pt},thick] (-0.5,2.5)--(-4,6) ;
\draw[fill, blue] (-0.5,2.5) circle [radius=0.05]  ; 
\end{scope}
\end{tikzpicture}
\end{center}
\caption{Examples of space-time diagrams of a gliders with annihilation (left) and of a gliders with annihilation subject to noise (right). In these examples, there are three types of particles: blue (zigzag) particles have speed $-1$, green (wavy) particles have speed $1$, and red ones have speeds $2$. Red and green particle do not interact, while red and blue particles annihilate upon meeting, and so do green and blue.}
\label{fig:glidersexa}
\end{figure}
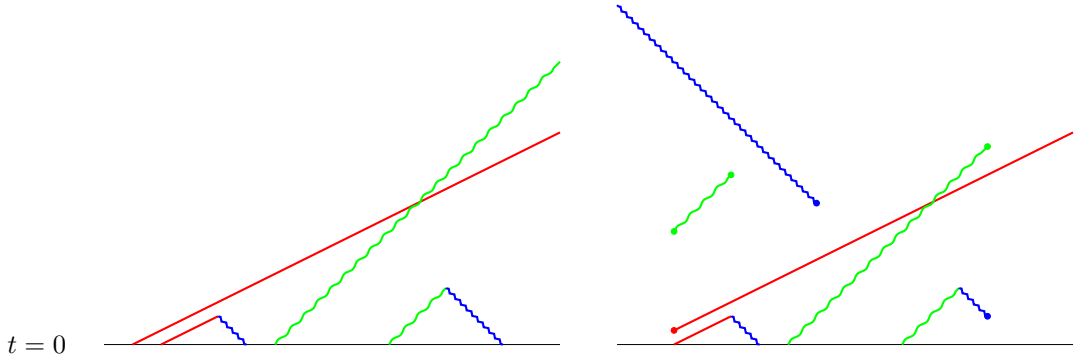

\begin{theorem}\label{thm-glider_an}
	Let $F\isdef A\oo G$ be a gliders with annihilation, and let $\theta$ be a positive birth-death noise. The noisy version of $F$ with noise $\theta$ is uniformly ergodic, with a spatially mixing invariant measure.
\end{theorem}

\begin{proof}
We couple the action of the noise $\theta$ on two configurations $x$ and $y$ in the following manner. For each site $k\in\Z$ and each $i\in \{1,\ldots,N\}$, we draw independently a random number $U_{k,i}$, uniformly distributed on $[0,1]$. We update $x$ and $y$ using the same samples, and according to the following rule:
\begin{itemize}
	\item if $x_{k,i}=\qO$ (resp.\ $y_{k,i}=\qO$) and $1-\beta_i\leq U_{k,i}\leq 1$, we add a particle of type $i$ at position $k$ in configuration $x$ (resp.\ in $y$),
	\item if $x_{k,i}=\qX$ (resp.\ $y_{k,i}=\qX$) and $0\leq U_{k,i}\leq\delta_i$, we remove the particle of type $i$ at position $k$ in configuration $x$ (resp.\ in $y$),
	\item otherwise, $x_{k,i}$ (resp.\  $y_{k,i}$) remains unchanged.
\end{itemize}

Let us first assume that $\delta_i<1-\beta_i$. Then, if $U_{k,i}\in[0,\delta_i]$, whatever the values of $x_{k,i}$ and $y_{k,i}$ are, we know that after the update, there is no particle of type $i$ at position $k$ in either configuration. On the other hand, if $U_{k,i}\in [1-\beta_i,1]$, we know that after the update, there is a particle of type $i$ at position $k$ in both configurations. Thus, if $U_{k,i}\in [0,\delta_i]\cup[1-\beta_i,1]$, then the two updated configurations coincide at component $i$ of position $k$. 
If we now assume that $\delta_i>1-\beta_i$, we can check in the same fashion that if $U_{k,i}\in[0,1-\beta_i]\cup[\delta_i,1]$, the two updated configurations coincide at component $i$ of position $k$.
This shows that in all cases, after the action of the noise, the two configurations coincide at component $i$ of position $k$ with probability at least $\varepsilon_i\isdef\min\{\beta_i+\delta_i, 2-(\beta_i+\delta_i)\}>0$.

Let us make a coupling $(X^t,Y^t)_{t\geq 0}$ of the PCA recursively as follows.  Let $U\isdef(U_{k,i}^t)_{k\in\ZZ^d,1\leq i\leq N,t\in\NN}$ be a collection of independent random samples uniformly drawn from $[0,1]$.
Starting with arbitrary configurations $X^0\isdef x^0$ and $Y^0\isdef y^0$,
at each time step, we first apply the deterministic CA $F=A\oo G$ and then perturb the two configurations with the noise, using the random samples in $U$ and the coupling strategy sketched above.


We say that two configurations $x$ and $y$ have a \emph{disagreement} of type $i$ at position $k$ if $x_{k,i}\not=y_{k,i}$. For a finite subset $K\subset \Z^d$, let $D_K(x,y)=\sum_{k\in K}\norm{x_k-y_k}_1$ be the number of disagreements between $x$ and $y$ in $K$. Note that $D_K(x,y)\leq N\cdot\abs{K}$.

In the two configurations $G(x)$ and $G(y)$, there can be a disagreement $G(x)_{k,i}\neq G(y)_{k,i}$ of type~$i$ at position~$k$ if and only if $x_{k+v_i, i}\neq y_{k+v_i, i}$. Let us recall that $G$ has neighbourhood $\Neighb_G\isdef\{v_1,\ldots,v_N\}$. It follows that $D_K\big(G(x),G(y)\big)\leq D_{K+\Neighb_G}(x,y)$. 
Next, observe that the action of the annihilating rule $A$ does not increase the number of disagreements. Indeed, when applying an annihilation rule $A_{i,j}$ at position $k$,
\begin{itemize}
	\item if there is no disagreement of types $i$ and $j$, then after the action of the annihilation rule, there is still no disagreement,
	\item if exactly one of the two components $i$ and $j$ contains a disagreement, then in the updated configuration, still exactly one of the two components contains a disagreement,
	\item if there are two disagreements of types $i$ and $j$, then in the updated configuration, there are either no disagreement (if there were particles both types in one of the configuration, and none in the other) or still two disagreements (if one configuration has only a particle of type $i$ and the other only a particle of type $j$).
\end{itemize}
The other components are not affected by the annihilation rule.
Combining the effects of the glider $G$ and the annihilation $A$, we find that
\begin{align}
	D_K\big(F(x),F(y)\big) &\leq D_{K+\Neighb_G}(x,y)
\end{align}
for each two configurations $x$ and $y$.

Applying the noise, the expected number of disagreements decreases by a factor at least $1-\varepsilon$, where $\varepsilon\isdef\min_{i\in\{1,\ldots,N\}} \varepsilon_i$. It follows that
\begin{align}
	\E\big[D_K(X^{t+1},Y^{t+1})\,\big|\, X^t,Y^t\big]
		&\leq (1-\varepsilon) D_{K+\Neighb_G}(X^t,Y^t) \;,
\shortintertext{and thus}
	\E\big[D_K(X^{t+1},Y^{t+1})\big] &\leq (1-\varepsilon) \E\big[D_{K+\Neighb_G}(X^t,Y^t)\big] \;,
\end{align}
Consequently, for every $k\in\Z^d$ and $t\geq 0$, we have
\begin{align}
	\E\big[D_{\{k\}}(X^t,Y^t)\big] &\leq (1-\varepsilon)^t D_{k+\Neighb^t_G}(x^0,y^0) \;.
\end{align}
Let $r=\max_{i\in\{1,\ldots,N\}}\abs{v_i}$ be the neighbourhood radius of $G$. The cardinality of the set $\Neighb^t_G$ is bounded by $(2rt+1)^d$.
Thus, we obtain
\begin{align}
	\E\big[D_{\{k\}}(X^{t},Y^{t})\big] &\leq (1-\varepsilon)^t(2rt+1)^dN \;,
\end{align}
It follows that $\xPr(X^t_k\neq Y^t_k)\to 0$ as $t\to\infty$, uniformly in the position $k$ and the choice of the initial configurations $x^0$ and $y^0$.  The uniform ergodicity of the PCA and the spatial mixing of its unique invariant measure follow from Proposition~\ref{prop:coupling:forward}.
\end{proof}

\begin{remark}
	Let us highlight the essence of the above argument.
	\begin{itemize}
		\item We have a \emph{discrepancy} function $\delta:S\times S\to\RR^+$ with the property that 				
		\begin{align}
			\text{$\delta(a,b)=0$ \qquad if and only if \qquad $a=b$.}
		\end{align}
			For a finite set $K\subset\ZZ$ and two configurations $x,y$, we define $D_K(x,y)\isdef\sum_{i\in K}\delta(x_i,y_i)$.
		\item We have a CA $F$ that is \emph{almost contractive}, meaning that 
			\begin{align}
				D_K(Fx,Fy) &\leq D_{K+\Neighb}(x,y)
			\end{align}
			for all $x,y\in S^{\ZZ^d}$ and $K\subset \Z^d$.
		\item We have a zero-range noise, identified by a matrix $\theta$, that is \emph{contractive} in the sense that there exists an $\varepsilon>0$ with the following property: for every $a,b\in S$, there is a coupling $(U,V)$ of $\theta(a,\cdot)$ and $\theta(b,\cdot)$ such that $\xExp[\delta(U,V)]\leq (1-\varepsilon)\delta(a,b)$.
	\end{itemize}
	If all these conditions are fulfilled, then the argument above shows that the noisy version of $F$ with noise $\theta$ is uniformly ergodic.
	For instance, the uniform ergodicity of Theorem~\ref{thm-glider_an} persists if we replace the annihilation rule with any other interaction rule $h:S\to S$ satisfying $\norm{h(b)-h(a)}_1\leq\norm{b-a}_1$.


	In the next section, we show how the coupling presented in the proof of Theorem~\ref{thm-glider_an} can be used to prove the ergodicity of another type of gliders with noise, even in a case where the approach via discrepancy functions is not sufficient.
	\hfill\remarkqed
\end{remark}

\subsubsection{Simple gliders with reflecting walls}
\label{sec:gliders:walls}

Let us consider a one-dimensional gliders CA $G$ with three types of particles: 
\begin{itemize}
	\item particles of type `$\pW$' have velocity $0$; they play the role of \emph{walls}, 
	\item particles of type `$\pR$' have velocity $1$; they move one unit to the \emph{right} at each time step, 
	\item particles of type `$\pL$' have velocity $-1$; they move one unit to the \emph{left} at each time step. 
\end{itemize}
The set of symbols is thus $S=\{\qO,\qX\}^3$ and the neighbourhood is $\Neighb=\{-1,0,1\}$.  We keep the same notations as in the previous section: for $x\in S^{\Z}$, $k\in\ZZ$ and $i\in\{\pW,\pR,\pL\}$, $x_{k,i}=\qX$ means that in $x$, there is a particle of type $i$ at position $k$.

We combine $G$ with a reflection rule $I$ modeling the reflection of left and right particles on walls (see Fig.~\ref{fig:gliderwalls}).
The \emph{reflection} rule $I$ is the CA of neighbourhood $\{0\}$ defined on the same configuration space $S^{\Z}$ by
\begin{align}
	I(x)_k &=
		\begin{cases}
			\begin{cellmatrix}\qX\\ \qO\\ \qX \end{cellmatrix}\quad{}
				& \text{if $x_k=\begin{cellmatrix}\qX\\ \qX\\ \qO\end{cellmatrix}$,} \smallskip \\
			\begin{cellmatrix}\qX\\ \qX\\ \qO\end{cellmatrix}
				& \text{if $x_k=\begin{cellmatrix}\qX\\ \qO\\ \qX\end{cellmatrix}$,} \\
			x_k & \text{otherwise,}
		\end{cases}
\end{align}
for each $x\in S^\ZZ$ and $k\in\ZZ$.
We call the composition $I\oo G$ the (one-dimensional) \emph{gliders with reflecting walls}.

As in the previous section, we consider a birth-death noise $\theta$, defined by some parameters $\beta_{\pW},\beta_{\pR},\beta_{\pL}\in[0,1]$ and $\delta_{\pW},\delta_{\pR},\delta_{\pL}\in[0,1]$ respectively representing the birth and death rates of the three types of particles.

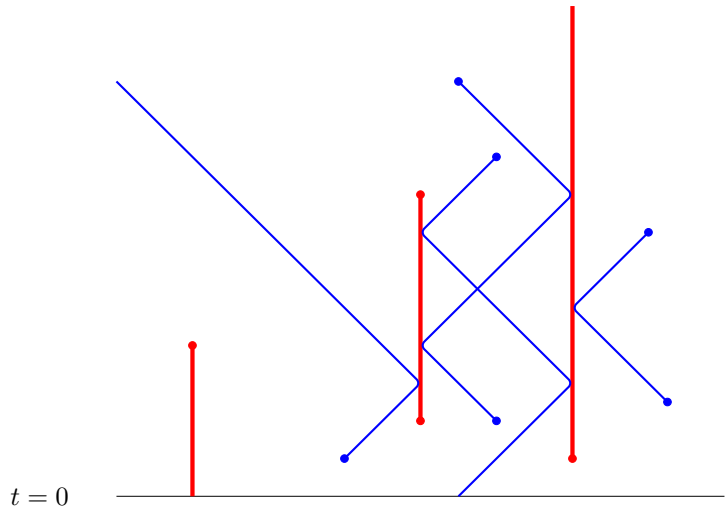
\begin{figure}[h]
\begin{center}
\begin{tikzpicture}[scale=1,rounded corners=2pt]
\foreach \y in {0} \draw[ultra thin] (-4,-0.5*\y) --(4,-0.5*\y) ;
\draw[red, ultra thick] (-3,0)--(-3,2) ;
\draw[fill, red] (-3,2) circle [radius=0.05]  ; 
\draw[red, ultra thick] (0,1)--(0,4) ;
\foreach \y in {1,4} \draw[fill, red] (0,\y) circle [radius=0.05]  ; 
\draw[red, ultra thick] (2,0.5)--(2,6.5) ;
\foreach \y in {0.5} \draw[fill, red] (2,\y) circle [radius=0.05]  ; 
\draw[blue, thick] (0.5,0)--(2,1.5)--(0,3.5)--(1,4.5) ;
\draw[fill, blue] (1,4.5) circle [radius=0.05]  ; 
\draw[blue, thick] (-1,0.5)--(0,1.5)--(-4,5.5) ;
\draw[fill, blue] (-1,0.5) circle [radius=0.05]  ; 
\node[left,overlay] at (-4.5,0) {$t=0$} ;
\draw[blue, thick] (3.25,1.25)--(2,2.5)--(3,3.5) ;
\draw[fill, blue] (3.25,1.25) circle [radius=0.05]  ; 
\draw[fill, blue] (3,3.5) circle [radius=0.05]  ; 
\draw[blue, thick] (1,1)--(0,2)--(2,4)--(0.5,5.5) ;
\draw[fill, blue] (1,1) circle [radius=0.05]  ; 
\draw[fill, blue] (0.5,5.5) circle [radius=0.05]  ; 
\end{tikzpicture}
\end{center}
\caption{Example of space-time diagram of a noisy gliders with reflecting walls.}
\label{fig:gliderwalls}
\end{figure}

\begin{theorem}\label{thm-glider_ref} Let $F=I\oo G$ be the gliders with reflecting walls, and let $\theta$ be a positive birth-death noise. The noisy version of $F$ with noise $\theta$ is uniformly ergodic, with an invariant measure that is spatially mixing and admits a perfect sampling algorithm.
\end{theorem}

\begin{proof}
We couple the action of the noise on configurations in the same manner as in the proof of Theorem~\ref{thm-glider_an}.  However, unlike in the previous result, we couple the PCA from the past.

To be more specific, we use an update function of the form $\tu:S\times[0,1]^3\to S$ for the noise $\theta$, where
\begin{align}
	\tu(a;u)_i &\isdef
		\begin{cases}
			\qX		& \text{if $a_i=\qO$ and $1-\beta_i\leq u_i\leq 1$,} \\
			\qO		& \text{if $a_i=\qX$ and $0\leq u_i\leq \delta_i$,} \\
			a_i		& \text{otherwise,}
		\end{cases}
\end{align}
for each $a\in S$, $u\in[0,1]^3$ and $i\in\{\pW,\pR,\pL\}$.
If $U$ is uniformly drawn from $[0,1]^3$, then for every $a\in S$, the value $\tu(a,U)$ is distributed according to $\theta(a,\cdot)$.

We use a family of independent samples $(U_k^n)_{k\in\ZZ,n\in\NN^-}$ uniformly drawn from $[0,1]^3$ to simulate the PCA from the past.  To determine $X^n$, we first apply the CA $F$ on $X^{n-1}$ and then update the value at each site $k$ using the update function $\tu$ and the sample $U_k^n$.

First, observe that the evolution of the walls at different sites are independent and are not affected by the other types of particles.  Namely, walls have velocity $0$ and are not affected by the reflection rule, and moreover, the noise is zero-range and acts on walls independently of the other two types of particles.  It follows that the presence or absence of a wall at position $0$ and time $0$ is almost surely uniquely determined by a finite (though random) number of samples $U_{0,\pW}^m$ with $m\leq 0$.

We claim that the presence of left- or right-moving particles at position $0$ and time $0$ is also almost surely a function of a finite number of random samples $U_k^m$.
In order to know if there is a right-moving particle at position $0$ and time $0$, we trace back the possible trajectory of the particle in time.  Each time we take a step back, we first determine the presence or absence of a wall at the current position so as to know whether the particle has changed direction or not.  The potential ancestor at time $-t$ can either be a right-moving particle or a left-moving particle, depending on whether the backward trajectory has met an even or odd number of walls.

Let $\varepsilon_{\pR}=\min\{\beta_{\pR}+\delta_{\pR}, 2-(\beta_{\pR}+\delta_{\pR})\}$ and $\varepsilon_{\pL}=\min\{\beta_{\pL}+\delta_{\pL}, 2-(\beta_{\pL}+\delta_{\pL})\}$ and set $\varepsilon\isdef\min\{\varepsilon_{\pR},\varepsilon_{\pL}\}>0$.  When tracing back the trajectory of a potential right-moving particle, at each step, we have a probability at least $\varepsilon$ of learning whether there is indeed an ancestor particle or not.  Therefore, almost surely, we eventually learn about the presence or absence of an ancestor.  If so, when going up again in time, we can determine whether there is a right-moving particle at position $0$ and time $0$ or not.  In the same fashion, we can almost surely determine the presence or absence of a left-moving particle at position $0$ and time $0$ by exploring a finite part of the samples in $U$.

It follows that $p_t\to 1$ as $t\to\infty$, and Proposition~\ref{prop:coupling} concludes the proof.
\end{proof}

\begin{remark}
	The two-dimensional version of gliders with reflecting walls is often called the \emph{mirror model} (or the \emph{discrete Lorentz gas} model)~\cite{RuiCoh88}.
	In the mirror model, mirrors are placed at some sites of the lattice $\ZZ^2$ in either of the two diagonal directions.  Particles (or beams of light) travel with speed~$1$ vertically or horizontally and are reflected upon hitting the mirrors.
	A similar argument as above shows the ergodicity of the mirror model in presence of positive birth-death noise.
	\hfill\remarkqed
\end{remark}

\subsection{Permutive CA with permutation noise}
\label{sec:permutive}

In this section, the kind of coupling is quite different, since it involves only finite Markov chains: for permutive CA with permutation noise, it is indeed possible to couple the evolution of all trajectories in any finite window.  For the simplicity of the presentation, we focus on the one-dimensional setting.  Analogous results can be obtained in higher dimensions.

Let $F$ be a CA of neighbourhood $\Neighb = \{\ell,\ell+1, \dots, r\}$ and local function $f:S^m \rightarrow S$, with $m=r-\ell+1\geq 2$. We say that $F$ is \emph{left-permutive} (resp. \emph{right-permutive}) if, for all $w\in S^{m-1}$, 
the mapping
$\tau_w: S\to S$ given by $\tau_w(a)\isdef f(aw)$ (resp., $\tau_w(a)\isdef f(wa)$) is bijective.
A CA is \emph{permutive} if it is either left- or right-permutive; it is \emph{bipermutive} if it is both left- and right-permutive. For example, when $S$ is the ring $\Z_n$ of integers modulo $n$, the \emph{affine} CA defined by $f(x,y)\isdef ax+by+c$ for $a,b,c\in\Z_n$ is left-permutive (resp., right-permutive) if $a$ (resp. $b$) is invertible in $\Z_n$.

Let $F$ be a permutive CA. Using the bijections $\tau_w$ one can prove that $F$ is surjective.  Every surjective CA with configuration space $\X$ preserves the uniform Bernoulli measure $\lambda$ on $\X$ (see e.g.~\cite[Thm.~5.21]{Kur03}).
The next proposition shows that when a permutive CA is subjected to a zero-range noise that preserves $\lambda$, the resulting PCA indeed converges to $\lambda$. The proof below is adapted from a work of Vasilyev~\cite{Vas78,TooVasStaMitKurPir90}.
An alternative proof (for additive noise) is provided at the end of Section~\ref{sec:surj:proof}.

\begin{theorem}\label{thm-perm}
	Every PCA resulting from adding positive permutation noise to a permutive CA is uniformly ergodic with the uniform Bernoulli measure as its unique invariant measure.
\end{theorem}
\begin{proof}
	Let $F$ be a permutive CA with local rule $f$, and $\Theta$ a permutation noise with noise matrix $\theta$.  Let $\Phi$ denote the resulting noisy CA.
	We will prove that for every $n\in\N$ and every initial measure $\mu$ on $\X$, the marginal distribution of $\mu\Phi^t$ on $K=\{-n,-n+1,\ldots,n\}$ converges exponentially to the uniform Bernoulli distribution on $S^K$, which we denote by $\lambda_K$.  More specifically, we will prove that for each $n\in\N$, there exists a real number $\rho<1$ such that for every $\mu\in\M(\X)$ and each $t\in\NN$, we have $\norm{\mu\Phi^t-\lambda}_K\leq\rho^t$, where as before, $\norm{\nu'-\nu}_K$ denotes the total variation distance between the marginal distributions of $\nu$ and $\nu'$ on $K$.

Let us first assume that $F$ is left-permutive with neighbourhood $\Neighb=\{0,1,\ldots,r\}$.  
By permutivity of $F$, for every $w\in S^{r}$ we have a bijection 
$\tau^{(K)}_w:S^K\to S^K$ given by
\begin{align}
	\tau^{(K)}_w(x) &\isdef f^{(K)}(xw) \;,
\end{align}
where $f^{(K)}$ denotes the map $S^{\Neighb(K)}\to S^K$ induced by the local rule $f$.

When fixing the word $w$ as a boundary condition on the right of $K$, the PCA $\Phi$ transforms a word $x$ in $S^K$ to a random word $Z$ in $S^K$ distributed according to a product distribution with marginal distribution $\theta(y_k,\cdot)$ at site $k\in K$, where $y=\tau^{(K)}_w(x)$.
We denote by $P_w(x,z)$ the probability that $x\in S^K$ is transformed into $z\in S^K$, that is,
$P_w(x,z)=\prod_{k\in K}\theta(y_k,z_k)$.

Since the map $\tau^{(K)}_w$ is bijective, it preserves the uniform distribution $\lambda_K$. By assumption, the noise matrix $\theta$ also preserves the uniform distribution on $S$, so we obtain $\lambda_K P_w=\lambda_K$.

For each $w\in S^r$, the matrix $P_w$ is a positive stochastic matrix. Therefore, there exists $\rho_w<1$ such that for every two probability distributions $q,q'$ on $S^K$, we have 
\begin{align}
	\norm{q' P_w - q P_w}_{\TV} &\leq \rho_w\norm{q'-q}_{\TV} \;,
\end{align}
where $\norm{q'-q}_{\TV}$ denotes the total variation distance between $q$ and $q'$.
Let us set $\rho\isdef\max\{\rho_w \, ; \; w\in S^r\}$. It follows that for any sequence $(w^t)_{t\geq 0}$ of words of $S^r$, we have
\begin{align}
	\norm{q'P_{w^0}P_{w^1}\cdots P_{w^{t-1}}-q P_{w^0}P_{w^1}\cdots P_{w^{t-1}}}_{\TV}
		&\leq \rho^t\norm{q'-q}_{\TV} \;.
\end{align}
In particular, for $q'=\lambda_K$, we obtain that for every distribution $q$ on $S^K$ and every sequence $(w^t)_{t\geq 0}$ of words in $S^r$,
$\norm{q P_{w^0}P_{w^1}\cdots P_{w^{t-1}} - \lambda_K}_{\TV}\leq\rho^t\norm{q-\lambda_K}_{\TV}\leq\rho^t$.

\begin{figure}[h]
	\begin{center}
	\begin{tikzpicture}[xscale=1.2, yscale=0.6] 
	\foreach \y in {0,1,2,3,4} \draw (-3.5,\y)--(3.5,\y) ;
	\foreach \y in {0,1,2,3} {
		\node[left,overlay] at (-3.6,0.5+\y) {$t=\y$} ;
		\node at (2.4,0.5+\y) {$w^{\y}$} ;
	} ;
	\foreach \x in {-2, 2.8, 2} \draw (\x,0)--(\x,4.5) ;
	\node[right] at (-1.5,0.5) {$U^0\sim q$} ;
	\node[right] at (-1.5,1.5) {$U^1\sim q P_{w^0}$} ;
	\node[right] at (-1.5,2.5) {$U^2\sim q P_{w^0}P_{w^1}$} ;
	\node[right] at (-1.5,3.5) {$U^3\sim q P_{w^0}P_{w^1}P_{w^2}$} ;
	\node[below] at (-2,0) {$-n$} ;
	\node[below] at (2,0) {$+n$} ;
	\node[below] at (2.8,0) {$n+r$} ;
	\end{tikzpicture}
	\end{center}
	\caption{Illustration of the proof of Theorem \ref{thm-perm}.}
	\label{fig:thm-perm}
\end{figure}

Let now $\mu$ be a distribution on $\X$. When iterating $\Phi$, it induces a random sequence of words $(W^t)_{t\geq 0}$ on $\{n+1,\ldots,n+r\}$. Conditioning on this sequence and using the above inequality, we get
\begin{align}
	\norm{\mu\Phi^t-\lambda}_{K} &\leq
		\max_{w^0,\ldots,w^{t-1}\in S^r}
			\norm{\mu_K P_{w^0}P_{w^1}\cdots P_{w^{t-1}} - \lambda_K}_{\TV}
		\leq\rho^t
\end{align}
for every $t\in\NN$ (see Fig.~\ref{fig:thm-perm}).

If the neighbourhood of $F$ is not of the form $\Neighb=\{0,1,\ldots,r\}$, then there exists a number $s\in\Z$ such that $F\circ\sigma^s$ is a left-permutive CA having a neighbourhood of that form.
If we denote the noisy version of $F\circ\sigma^s$ by $\Phi_s$, the above inequality yields
$\norm{\mu'\Phi_s^t-\lambda}_K\leq\rho^t$
for every distribution $\mu'$, in particular, for $\mu'\isdef\sigma^{-st}\mu$.
With this choice, $\mu'\Phi_s^t=\mu\Phi^t$ and we obtain
$\norm{\mu\Phi^t-\lambda}_K\leq\rho^t$, which concludes the proof.
The right-permutive case is analogous.
\end{proof}

\section{Entropy method: surjective CA with additive noise}
\label{sec:entropy}
The purpose of this section is to prove that under the action of a surjective CA perturbed by positive additive noise, every shift-invariant probability measure is attracted towards the uniform Bernoulli measure.  This does not settle the ergodicity question because we do not know if other non-shift-invariant measures are attracted towards the same measure, and we do not know if the uniform Bernoulli measure is the only invariant measure.

The idea of the proof is as follows: we know that a surjective CA preserves the entropy per site of shift-invariant probability measures.  On the other hand, positive additive noise
increases the entropy unless the measure has maximal entropy.
Combining these two, we get that a surjective CA followed by positive additive noise
increases the entropy unless the measure has maximal entropy.
This however is not quite enough to prove convergence to the measure of maximal entropy because entropy per site is not a continuous function of the measure and hence cannot serve as a simple Lyapunov function; we need to control how much the entropy increases.

The analysis of finite-state Markov chains via entropy is classic and goes back to the ideas of Boltzmann (see e.g.~\cite[Sec.~II.7]{Pen69} or~\cite[Sec.~II.4]{Lig85}).
The use of entropy to describe the asymptotic behaviour of continuous-time interacting particle systems was pioneered by Holley~\cite{Hol71,Lig85} and has been very successful.
For applications of the entropy method to PCA see~\cite{KozVas80,Yag00,DaiLouRoe02}.


In this section, we prove the following result.

\begin{theorem}\label{thm-surj}
\label{thm:surj:ergodicity:shift-invariant}
	Let $\Phi$ be a PCA on configuration space $\X$ obtained by perturbing a surjective CA with a positive additive noise.
	Then, the uniform Bernoulli measure $\lambda$ on $\X$ is invariant under $\Phi$ and $\mu\Phi^t\to\lambda$ weakly as $t\to\infty$ for every shift-invariant measure $\mu$ on $\X$.
\end{theorem}

Before entering the proof, let us note that the NEC-majority CA of Example~\ref{exp:toom} is not surjective.
The non-surjectivity in that example follows easily from the Garden-of-Eden theorem, which is discussed below in the proof of Lemma~\ref{lem:surjective:entropy}.
For a more direct argument, one can verify that, for instance, any configuration that has an occurrence of the pattern
\begin{align}
	\xConfig{
		\symb{1} \& \symb{0} \& \symb{0} \& \symb{0} \\
		\symb{0} \& \symb{0} \& \symb{0} \& \symb{0} \\
		\symb{1} \& \symb{1} \& \symb{0} \& \symb{0} \\
		\symb{0} \& \symb{1} \& \symb{0} \& \symb{1} \\
	}
\end{align}
has no pre-image under the NEC-majority CA.

For clarity, we present the proof of Theorem~\ref{thm:surj:ergodicity:shift-invariant} in the one-dimensional setting, but everything goes through similarly in the higher-dimensional case.  The notion of additive noise requires that the set of symbols for the PCA is identified with a finite Abelian group.  This identification is however arbitrary.  In fact, the theorem remains true if the additive noise is replaced with any positive permutation noise.  We stick to the additive noise to keep the presentation simple.
At the end of this section, we also use the entropy method to give an alternate proof of Theorem~\ref{thm-perm} in case of additive noise.

\subsection{Entropy}

Let us fix the notation and terminology for entropy.
The entropy of a random variable $A$ taking values from a finite set $\Sigma$ will be denoted by
\begin{align}
	H(A) &\isdef -\sum_{a\in\Sigma} \xPr(A=a)\log\xPr(A=a) \;.
\end{align}
We recall that $H(A)\leq\log\abs{\Sigma}$ and the equality holds if and only if $A$ is uniformly distributed over~$\Sigma$.
If $B$ is another random variable on the same probability space, we write
\begin{align}
	\widehat{H}(A\,|\,B) &\isdef
		-\sum_{a\in\Sigma} \xPr(A=a\,|\,B)\log\xPr(A=a\,|\,B)
\end{align}
for the entropy of the conditional distribution of $A$ given $B$.
Note that this is a random variable, and is not the same as the usual notion of conditional entropy which is a number.
The usual conditional entropy of $A$ given $B$ is given by
\begin{align}
	H(A\,|\,B) \isdef \xExp\big[\widehat{H}(A\,|\,B)\big] \;.
\end{align}
Entropies satisfy the chain rule $H(A,B)=H(B) + H(A\,|\,B)$, where $H(A,B)$ denotes the entropy of the pair $(A,B)$.
As a consequence, if a random variable $B\isdef g(A)$ is a function of another random variable $A$, then $H(A)=H(B)+H(A\,|\,B)\geq H(B)$.
The mutual information
\begin{align}
	I(A;B) &\isdef H(A)-H(A\,|\,B) = H(B)-H(B\,|\,A)
\end{align}
between two random variables $A$ and $B$ is always non-negative and takes value $0$ if and only if the two variables are independent.

The \emph{entropy per site} of a shift-invariant probability measure $\mu$ refers to the limit
\begin{align}
	h(\mu) &\isdef
		\lim_{n\to\infty}\frac{H\left(X_{[-n,n]}\right)}{2n+1} \;,
\end{align}
where $X$ is a (one-dimensional) random configuration with distribution $\mu$.
Among the shift-invariant measures on $\X\isdef S^\ZZ$, the uniform Bernoulli measure is the unique measure with maximum entropy per site $\hmax\isdef\log\abs{S}$.

For background on the entropy, we refer to the book of Cover and Thomas~\cite{CovTho91} in the context of information theory and to the book of Denker, Grillenberger and Sigmund~\cite{DenGriSig76} in the context of dynamical systems.

\subsection{The effect of a surjective CA on entropy}

We start by looking at how a surjective CA affects the entropy of a finite region.

\begin{lemma}
\label{lem:surjective:entropy}
	Let $F$ be a one-dimensional surjective CA.  There is a constant $c>0$ such that for every random configuration $X$ and every finite interval $J\subseteq\Z$, we have
	\begin{align}
		H\big((FX)_J\big) &\geq H(X_J) - c \;.
	\end{align}
\end{lemma}
\begin{proof}
	Without loss of generality, we may assume that the neighbourhood of the local rule of $F$ is of the form $\Neighb\isdef\{-r,-r+1,\ldots,r\}$.
	We write $\partial\Neighb(J)\isdef\Neighb(J)\setminus J$ for the external \emph{boundary} of a set $J\subseteq\ZZ$ with respect to $\Neighb$.  Similarly, we write $\partial\Neighb^2(J)\isdef\Neighb^2(J)\setminus J$.
	
	Let $x\in\X$ be an arbitrary configuration.
	For an interval $J$, the pattern $(F(x))_J$ is uniquely determined by the patterns $x_J$ and $x_{\partial\Neighb(J)}$.
	Conversely, since by the Garden-of-Eden theorem (see e.g.~\cite[Theorem~5.3.1]{CecCoo10}), every surjective CA is pre-injective, the pattern $x_J$ is uniquely determined by the patterns $(F(x))_J$, $(F(x))_{\partial\Neighb(J)}$ and $x_{\partial\Neighb^2(J)}$.
	
	To see the latter, let $y$ be any configuration such that $y_J\not=x_J$ and $y_{\partial\Neighb^2(J)}=x_{\partial\Neighb^2(J)}$.
	Define a configuration $y'$ that agrees with $y$ on $\Neighb^2(J)$ and with $x$ outside $J$.
	Then $x$ and $y'$ are asymptotic to each other.
	Since $x$ and $y$ disagree on $J$, so do $x$ and $y'$.
	By pre-injectivity, $F(x)$ and $F(y')$ must be different from each other.
	Since $x$ and $y'$ disagree only on $J$, $F(x)$ and $F(y')$ can only disagree on $\Neighb(J)$.
	On the other hand, $F(y)$ and $F(y')$ agree on $\Neighb(J)$.
	Therefore, $F(x)$ and $F(y)$ must disagree on $\Neighb(J)=J\cup\partial\Neighb(J)$.
		
	Now consider the random configuration $X$.
	Since $X_J$ is uniquely determined by $(FX)_J$, $(FX)_{\partial\Neighb(J)}$ and $X_{\partial\Neighb^2(J)}$, we have the inequality
	\begin{align}
		H(X_J) &\leq
			H\left((FX)_J,(FX)_{\partial\Neighb(J)},X_{\partial\Neighb^2(J)}\right) \\
		&= H\left((FX)_J\right)
			+ H\left((FX)_{\partial\Neighb(J)},X_{\partial\Neighb^2(J)}\,|\, (FX)_J\right)
	\end{align}
	for the entropy.
	Since $\abs{\partial\Neighb(J)}=2r$ and $\abs{\partial\Neighb^2(J)}=4r$, the second term on the right-hand side is bounded from above by $6r\log\abs{S}$.  Therefore,
	\begin{align}
		H((FX)_J) &\geq H(X_J) - c
	\end{align}
	with $c\isdef 6r\log\abs{S}$.
\end{proof}

\begin{remark}
	The same argument is used in~\cite{KarTaa15} to show that $h(F\mu)=h(\mu)$ for every shift-invariant measure $\mu$ on $\X$.
Indeed, for a random configuration $X$ with distribution $\mu$ one has
	\begin{align}
		h(\mu) &=
			\lim_{n\to\infty}\frac{H\left(X_{[-n,n]}\right)}{2n+1}
		\leq
			\lim_{n\to\infty}\frac{H\left((FX)_{[-n,n]}\right)+c}{2n+1}
		=
			h(F\mu) \;.
	\end{align}
	The opposite inequality is true in general.
	\hfill\remarkqed
\end{remark}

\subsection{The effect of noise on entropy}

Lemma~\ref{lem:surjective:entropy} says that a one-dimensional surjective CA reduces the entropy of
a finite window by at most a constant $c$, uniformly on the size of the window.
We now show that if the window is large,
the extra entropy added by the noise is large enough
to compensate the lost entropy, at least if the entropy
of the window is not too close to maximal.
We divide the argument into a few lemmas.

Recall that in order to describe an additive noise, we identify the alphabet $S$ with a finite Abelian group $(\GG,+)$.  Under an additive noise, each symbol $a$ is replaced with a symbol $a+N$, where $N$ is $\GG$-valued random variable.  The noise variables at different sites are independent and all have distribution~$q$.  We are assuming that the noise is positive, hence $q(b)>0$ for each $b\in S$.
We denote by $\hmax\isdef\log\abs{S}$ the maximum possible entropy carried by a single site.

\begin{lemma}
\label{lem:sum:entropy:observation}
	For every $\varepsilon>0$, there is a $\delta(\varepsilon)>0$ with the following property.
	If $A$ and $N$ are independent $\GG$-valued random variables and $N$ is distributed according to $q$, then
	\begin{align}
		H(A)\leq \hmax - \varepsilon
			\qquad\Longrightarrow\qquad
		H(A+N)\geq H(A)+\delta(\varepsilon) \;.
	\end{align}
	The inequality $H(A+N)\geq H(A)$ holds in general as long as $A$ and $N$ are independent.
\end{lemma}
\begin{proof}
The entropy of a $\GG$-valued random variable $A$ and its noisy version $A+N$
(where $N$ is independent of $A$) are related in the following way:
\begin{align}
	H(N,A+N) &= H(N) + H(A+N\,|\, N) \;, \\
	H(N,A+N) &= H(A+N) + H(N\,|\,A+N)\;.
\end{align}
Since $A$ and $N$ are independent, we have $H(A+N\,|\,N)=H(A)$.  It follows that
\begin{align}
	H(A+N) &= H(A) + \underbrace{H(N)-H(N\,|\,A+N)}_{I(N;A+N)} \;.
\end{align}
The mutual information $I(N;A+N)$ is non-negative and takes value $0$ if and only if $N$ and $A+N$ are independent, which happens if and only if $A$ is uniform on $\GG$, that is to say $H(A)=\hmax$.

The claim follows from the continuity of entropy and convolution and the compactness of the set of probability measures on $\GG$.
\end{proof}

\begin{lemma}
\label{lem:sum:entropy:conditional}
	For every $\varepsilon>0$, there is a $\rho(\varepsilon)>0$ with the following property.
	Let $A$ and $N$ be $\GG$-valued random variables and $C$ another random variable.
	Suppose that $N$ is distributed according to $q$, and is independent of $A$ and $C$.
	Then,
	\begin{align}
		H(A\,|\,C)\leq \hmax - \varepsilon
			\qquad\Longrightarrow\qquad
		H(A+N\,|\,C)\geq H(A\,|\,C)+\rho(\varepsilon) \;.
	\end{align}
	The inequality $H(A+N\,|\,C)\geq H(A\,|\,C)$ holds in general
	as long as $A$ and $N$ are independent conditioned on $C$.
\end{lemma}
\begin{proof}
	For each $\varepsilon>0$, denote $\delta(\varepsilon)$ the number whose existence is guaranteed by Lemma~\ref{lem:sum:entropy:observation}.
	Lemma~\ref{lem:sum:entropy:observation} immediately
	gives a corresponding almost sure statement about the entropy of
	conditional distributions $\widehat{H}(A\,|\,C)$ and $\widehat{H}(A+N\,|\,C)$.
	Namely, if conditioned on $C$, the random variables $A$ and $N$ are independent
	and $N$ has distribution $q$, then
	\begin{align}
	\label{eq:sum:entropy:conditional:proof:start}
		\widehat{H}(A\,|\,C)\leq\hmax-\varepsilon
		\qquad\Longrightarrow\qquad
		\widehat{H}(A+N\,|\,C)-\widehat{H}(A\,|\,C)\geq\delta(\varepsilon)
	\end{align}
	with probability $1$.
	(In the proof of Theorem~\ref{thm:surj:ergodicity:shift-invariant}, we will only need Lemma~\ref{lem:sum:entropy:conditional} in situations where $C$ is a discrete variable and the conditional distributions are elementary.)
	
	Now, suppose that
	\begin{align}
		\xExp\big[\widehat{H}(A\,|\,C)\big]
			&= H(A\,|\,C) \leq \hmax-\varepsilon \;.
	\end{align}
	Using Markov's inequality, we get
	\begin{align}
		\xPr\left(\widehat{H}(A\,|\,C) \geq \hmax-\nicefrac{\varepsilon}{2}\right)
			&\leq \frac{\xExp\big[\widehat{H}(A\,|\,C)\big]}{\hmax-\nicefrac{\varepsilon}{2}}
			\leq \frac{\hmax-\varepsilon}{\hmax-\nicefrac{\varepsilon}{2}} < 1 \;.
	\end{align}
	Therefore,
	\begin{align}
		\xPr\left(\widehat{H}(A\,|\,C) < \hmax-\nicefrac{\varepsilon}{2}\right)
			&\geq \frac{\nicefrac{\varepsilon}{2}}{\hmax-\nicefrac{\varepsilon}{2}} > 0 \;,
	\end{align}
	that is, with probability at least $\varepsilon/(2\hmax-\varepsilon)$, we have $\widehat{H}(A\,|\,C) < \hmax-\nicefrac{\varepsilon}{2}$.
	Hence, with probability at least $\varepsilon/(2\hmax-\varepsilon)$, we have
	\begin{align}
		\widehat{H}(A+N\,|\,C) - \widehat{H}(A\,|\,C)
			&\geq \delta(\nicefrac{\varepsilon}{2}) > 0\;.
	\end{align}
	Taking expectation and using the non-negativity of $\widehat{H}(A+N\,|\,C) - \widehat{H}(A\,|\,C)$, we get
	\begin{align}
		H(A+N\,|\,C)-H(A\,|\,C) &\geq
			\frac{\varepsilon}{2\hmax-\varepsilon}\delta(\nicefrac{\varepsilon}{2}) > 0\;,
	\end{align}
	which proves the claim with
	$\rho(\varepsilon)\isdef[\varepsilon/(2\hmax-\varepsilon)]\delta(\nicefrac{\varepsilon}{2})$.
\end{proof}

\begin{lemma}
\label{lem:sum:entropy:vector}
	Let $s>0$.
	For every $\varepsilon>0$, there exists an integer $n_0>0$ such that for all $n\geq n_0$,
	if $A_1,A_2,\ldots, A_n$ are $\GG$-valued random variables
	and $N_1,N_2,\ldots, N_n$ are i.i.d.\ $\GG$-valued random variables
	with distribution $q$ and independent of $A_1,A_2,\ldots,A_n$, then
	\begin{align}
		H(\underline{A}) &\leq n(\hmax-\varepsilon)
			\qquad\Longrightarrow\qquad
		H(\underline{A}+\underline{N}) \geq H(\underline{A}) + s \;,
	\end{align}
	where $\underline{A}\isdef (A_1,A_2,\ldots,A_n)$ and $\underline{N}\isdef(N_1,N_2,\ldots,N_n)$ for brevity.
	The inequality $H(\underline{A}+\underline{N}) \geq H(\underline{A})$
	holds in general as long as $\underline{A}$ and $\underline{N}$ are independent.
\end{lemma}
\begin{proof}
	We have
	\begin{align}
		H(A_1,A_2) &= H(A_1) + H(A_2\,|\,A_1) \;, \\
	    H(A_1+N_1,A_2+N_2) &= H(A_1+N_1) + H(A_2+N_2\,|\,A_1+N_1) \;.
	\end{align}
	Since conditioning on more information does not increase the entropy, we have
	\begin{align}
		H(A_2+N_2\,|\,A_1+N_1) &\geq H(A_2+N_2\,|\,A_1,N_1) = H(A_2+N_2\,|\,A_1)
	\end{align}
	where the last equality is by the independence of $N_1$ and $A_2+N_2$.
	In a similar fashion, we obtain
	\begin{align}
		H(\underline{A}) &= H(A_1) + H(A_2\,|\,A_1) + \cdots + H(A_n\,|\,A_1,\ldots,A_{n-1}) \\
		H(\underline{A}+\underline{N}) &\geq H(A_1+N_1) + H(A_2+N_2\,|\,A_1) + \cdots \nonumber\\
			&{} \qquad\qquad + H(A_n+N_n\,|\,A_1,\ldots,A_{n-1}) \;.
	\end{align}
	Hence,
	\begin{align}
	\label{eq:decomposition}
		H(\underline{A}+\underline{N}) - H(\underline{A}) &\geq
			\sum_{i=1}^n \big[H(A_i+N_i\,|\,A_1,\ldots,A_{i-1})-H(A_i\,|\,A_1,\ldots,A_{i-1})\big] \;.
	\end{align}
	
	Choose $k$ large enough so that $k\rho(\nicefrac{\varepsilon}{2})\geq s$, where $\rho(\cdot)$ is as in Lemma~\ref{lem:sum:entropy:conditional}.  Take $n_0$ large enough so that $(\nicefrac{\varepsilon}{2})n_0\geq(k-1)\big(\hmax-\nicefrac{\varepsilon}{2}\big)$. Let $n\geq n_0$ and assume that $H(\underline{A}) \leq n(\hmax-\varepsilon)$.
	By the pigeonhole principle, there must be $k$ distinct indices $1\leq i_1,i_2,\ldots,i_k\leq n$ such that
	\begin{align}
		H(A_{i_\ell}\,|\,A_1,\ldots,A_{i_\ell-1}) &\leq \hmax-\nicefrac{\varepsilon}{2} \;.
	\end{align}
	Indeed, if this is not the case, there can exist at most $k-1$ indices $i\in\{1,2,\ldots,n\}$ for which $H(A_{i}\,|\,A_1,\ldots,A_{i-1}) \leq \hmax-\nicefrac{\varepsilon}{2}$, hence
	\begin{align}
		H(A) = \mathop{\smash[b]{\sum_{i=1}^n}} H(A_i\,|\,A_1,\ldots,A_{i-1})
			&> (n-k+1)\big(\hmax-\nicefrac{\varepsilon}{2}\big) \\
			&\geq n(\hmax-\varepsilon)+(\nicefrac{\varepsilon}{2})n-(k-1)\big(\hmax-\nicefrac{\varepsilon}{2}\big) \\
			&\geq n(\hmax-\varepsilon) \;,
	\end{align}
	which contradicts the assumption $H(\underline{A}) \leq n(\hmax-\varepsilon)$.
	
	By Lemma~\ref{lem:sum:entropy:conditional}, for each of these $k$ indices we have
	\begin{align}
		H(A_{i_\ell}+N_{i_\ell}\,|\,A_1,\ldots,A_{i_\ell-1})-H(A_{i_\ell}\,|\,A_1,\ldots,A_{i_\ell-1})
			&\geq \rho(\nicefrac{\varepsilon}{2}).
	\end{align}
	Since all the other terms in~\eqref{eq:decomposition} are non-negative, we get
	\begin{align}
		H(\underline{A}+\underline{N}) - H(\underline{A})
			&\geq k\rho(\nicefrac{\varepsilon}{2}) \geq s\;. \\
		& \qedhere
	\end{align}
\end{proof}

\subsection{Proof of Theorem~\ref{thm:surj:ergodicity:shift-invariant}}
\label{sec:surj:proof}

\begin{proof}[Proof of Theorem~\ref{thm:surj:ergodicity:shift-invariant}]
	For clarity, we focus on the one-dimensional case.
	See Remark~\ref{rem:entropy:higher-dimensional} for the general case.

	Let $\pi$ be an accumulation point of the measure orbit
	$\mu\to\mu\Phi\to\mu\Phi^t\to\cdots$ starting from a shift-invariant measure $\mu$.
	We show that $\pi$ is the uniform Bernoulli measure.
	In order to do that, we show that $h(\pi)\geq\hmax-\varepsilon$ for every $\varepsilon>0$, and use the fact that the uniform Bernoulli measure is the only shift-invariant measure with entropy $\hmax$.
	
	To be specific, let us use the following construction of a trajectory of the noisy CA with initial distribution $\mu$.
	Let $X^{(0)}$ be a configuration with distribution $\mu$.
	Let $Z^{(1)},Z^{(2)},\ldots$ be a sequence of independent random configurations independent of $X^{(0)}$, each distributed according to the product measure with marginal $q$ at each site.
	Construct $X(t)$ recursively by setting $X^{(t+1)}\isdef FX^{(t)}+Z^{(t+1)}$.
	
By Lemma~\ref{lem:surjective:entropy}, for every finite interval $J\subseteq\Z$ and every $t\in\NN$, we have
	\begin{align}
		H\left((FX^{(t)})_J\right) &\geq H(X^{(t)}_J)-c \;.
	\end{align}
	Let $\varepsilon>0$.
	By Lemma~\ref{lem:sum:entropy:vector}, 
	there is an $n_0>0$
	(corresponding to $s\leftarrow 2c$ and $\varepsilon$) 
	such that for every finite interval $J\subseteq\Z$
	of size at least $n_0$ and every $t\in\NN$, either
	\begin{align}
		\label{eq:surjective:proof:option:1}
		H\big(X^{(t+1)}_J\big) &\geq H\big((FX^{(t)})_J\big)
			> \abs{J}(\hmax-\varepsilon) \\
	\shortintertext{or}
		\label{eq:surjective:proof:option:2}
		H\big(X^{(t+1)}_J\big)
			&\geq H\big((FX^{(t)})_J\big) + 2c \geq H(X^{(t)}_J)+c \;.
	\end{align}
	It follows that for every $t\geq(\abs{J}\cdot\hmax)/c$,
	\begin{align}
	\label{eq:surjective:proof:final-bound}
		H\big(X^{(t)}_J\big) &> \abs{J}(\hmax-\varepsilon)-c \;,
	\end{align}
	provided $\abs{J}\geq n_0$.
	Indeed, observe that once~\eqref{eq:surjective:proof:final-bound} is satisfied for some $t=t_0$, it remains satisfied for all $t\geq t_0$.  On the other hand, within $\big\lceil(\abs{J}\cdot\hmax)/c\big\rceil$ steps, inequality~\eqref{eq:surjective:proof:option:1} is bound to be satisfied at least once. 
	Letting $\abs{J}\to\infty$, we get
	\begin{align}
		h(\pi) &\geq
			\lim_{\abs{J}\to\infty}\,\liminf_{t\to\infty}\frac{H(X^{(t)}_J)}{\abs{J}}
		\geq
			\lim_{\abs{J}\to\infty}\frac{\abs{J}(\hmax-\varepsilon)-c}{\abs{J}}
		=
			\hmax - \varepsilon \;.
	\end{align}
	Since $\varepsilon>0$ is arbitrary, the claim follows.
\end{proof}

\begin{remark}\label{rem:entropy:higher-dimensional}
	For a $d$-dimensional surjective CA, Lemma~\ref{lem:surjective:entropy} remains true except that rather than a constant $c$,	we need a function $c(J)$ that is $o(\abs{J})$ (for hypercubic $J$) as $\abs{J}\to\infty$.
	More specifically, with $\Neighb\isdef[-r,r]^d\cap\ZZ^d$, the statement holds for $c(J)\isdef\left(\abs{\partial\Neighb(J)}+\abs{\partial\Neighb^2(J)}\right)\log\abs{S}$.
	The rest of the argument goes through in the same fashion.
	In fact, the theorem remains true if the lattice $\Z^d$ is replaced	with a countable amenable group.
	\hfill\remarkqed
\end{remark}

\begin{remark}
	The proof of Theorem~\ref{thm:surj:ergodicity:shift-invariant} can be adapted to encompass the broader scenario in which the noise is a (positive) permutation noise.  Indeed, Lemma~\ref{lem:sum:entropy:observation} remains true if the noise variable $N$ is a random permutation chosen according to a distribution $q$ and the sum $A+N$ is replaced with the application $N(A)$, provided that the distribution $q$ has the property that for every $a,a'\in S$, there is a permutation $\varsigma\in\SymGroup(S)$ with $q(\varsigma)>0$ such that $\varsigma(a')=a$.  The latter condition is easily seen to be equivalent to the condition that the noise is positive.  The adapted variants of Lemmas~\ref{lem:sum:entropy:conditional} and~\ref{lem:sum:entropy:vector} and the rest of the proof then follow similarly.
	\hfill\remarkqed
\end{remark}

\begin{remark}
	Applying the argument of Theorem~\ref{thm:surj:ergodicity:shift-invariant}
	to non-shift-invariant measures, we still get a weaker statement:
	every accumulation point of the orbit of the noisy CA has well-defined
	uniform entropy per site $\hmax$.
	More specifically, let $\Gamma_0$ denote the set of probability measures on $\X$
	(not necessarily shift-invariant)
	that have well-defined \emph{uniform} entropy per site $\hmax$, that is,
	the measures $\mu$ for which the limit
	\begin{align}
		\breve{h}(\mu) &\isdef \lim_{\abs{J}\to\infty}\frac{H(X_J)}{\abs{J}}
	\end{align}
	(for a random configuration $X\sim\mu$) exists and equals $\hmax$.
	The limit is taken over intervals.
	The argument of Theorem~\ref{thm:surj:ergodicity:shift-invariant} shows that
	the iterates of the noisy CA $\Phi$ on any probability measure $\mu$
	converge weakly to the set $\Gamma_0$.
%
	\hfill\remarkqed
\end{remark}

Let us conclude this section by giving an alternate proof of Theorem~\ref{thm-perm} in case the noise is additive.
For permutive CA under positive additive noise, the entropy argument can be easily formulated in terms of conditional entropy, hence providing convergence for every (not necessarily shift-invariant) measure.  The argument is however not entirely different from the Markov chain proof given in Section~\ref{sec:permutive}; the Markov chain interpretation is implicit in the following proof.

\begin{proof}[Alternate proof of Theorem~\ref{thm-perm} with additive noise]
	Let $F$ be a right-permutive CA with neighbourhood $\Neighb\isdef\{l,l+r,\ldots,r\}$.
	Let $X$ be a random configuration with arbitrary distribution and set $Y\isdef FX$.
	Then, for every $k\in\Z$,
	\begin{align}
		H(X_{k+r}\,|\, X_{(-\infty,k+r)}) &=
			H(Y_{k}\,|\, X_{(-\infty,k+r)}) \\
		&\leq
			H(Y_k\,|\, Y_{(-\infty,k)}) \;.
	\end{align}
	The first equality is by permutiveness, and the second inequality is by the fact that $Y_{(-\infty,k)}$ is a function of $X_{(-\infty,k+r)}$.

	Next, let $Z$ be a noise configuration independent of $X$, and distributed according to a product measure with marginal $q$ at each site.
	Then,
	\begin{align}
		H(Y_k + Z_k \,|\, Y_{(-\infty,k)} + Z_{(-\infty,k)}) &\geq
			H(Y_k + Z_k \,|\, Y_{(-\infty,k)}, Z_{(-\infty,k)}) \\
		&=
			H(Y_k + Z_k \,|\, Y_{(-\infty,k)}) \;,
	\end{align}
	where the last equality follows from the independence of $Z_{(-\infty,k)}$ and $Y_k+Z_k$.
	
	Combining these two with Lemma~\ref{lem:sum:entropy:conditional}, we get that for every $\varepsilon>0$,
	\begin{gather}
		H(X_{k+r}\,|\, X_{(-\infty,k+r)}) \leq \hmax-\varepsilon \\
			\Downarrow \nonumber \\
		H(Y_k + Z_k \,|\, Y_{(-\infty,k)} + Z_{(-\infty,k)})
			\geq \left(H(X_{k+r}\,|\, X_{(-\infty,k+r)}) + \rho(\varepsilon)\right) \land (\hmax-\varepsilon) \;.
	\end{gather}
	In particular, if $X^{(0)},X^{(1)},\ldots$ represents the evolution of the noisy CA, then
	\begin{align}
		H(X^{(t)}_k \,|\, X^{(t)}_{(-\infty,k)}) &\to \hmax
	\end{align}
	as $t\to\infty$, uniformly in $k$.  This implies convergence to the uniform Bernoulli measure of the distribution of $X^{(t)}$.
\end{proof}

\section{Fourier analysis method}
\label{sec:fourier}

In this section, we apply (generalized) Fourier analysis to establish ergodicity under noise of CA with certain algebraic properties.  
For clarity and brevity, we focus on two concrete examples (the XOR CA and the binary spreading CA) and prove ergodicity under zero-range noise.  Further development of this approach will be left to another paper.

Our exposition is based on Chapter 4 of the survey by Toom et al.~\cite{TooVasStaMitKurPir90}.
The idea is to show that the action of the PCA on local observables is ``contractive'' in an appropriate sense.  When the CA has an algebraic property (e.g., additive), it is sometimes possible to choose a basis for the space of observables (e.g., the Fourier basis) with respect to which the CA maps each basis element into another basis element.  Proving the ergodicity of the noisy CA would then be reduced to showing that the action of noise on the same basis is contractive.

\subsection{XOR CA with zero-range noise}
\label{sec:fourier:binary}

Let $S\isdef\{\symb{0},\symb{1}\}$ be the binary alphabet.  We identify $S$ with the cyclic group $\ZZ/2\ZZ$.
The \emph{XOR} CA with neighbourhood $\Neighb\subseteq\ZZ^d$ is identified with the map $x\mapsto Fx$ on $\X\isdef S^{\ZZ^d}$, where
\begin{align}
	(Fx)_i &\isdef \sum_{j\in \Neighb} x_{i+j} \pmod{2} \;.
\end{align}
We consider the PCA $\Phi$ obtained by combining $F$ with a zero-range noise kernel $\Theta$, identified by the matrix
\begin{align}
	\theta &\isdef
		\begin{pmatrix}
			1-p & p \\
			q	& 1-q
		\end{pmatrix},
\end{align}
which modifies each symbol independently according to transition probabilities $\symb{0}\xrightarrow{p}\symb{1}$ and $\symb{1}\xrightarrow{q}\symb{0}$.

Since $F$ is permutive, we already know (Theorem~\ref{thm-perm}) the ergodicity of the noisy version as long as the noise is positive and preserves the uniform distribution, that is, if $q=p\in(0,1)$.
In the case $q=p\in(0,1)$, the ergodicity also follows by a classic application of Fourier analysis (see~\cite[Example~1.3]{TooVasStaMitKurPir90}) or by coupling from the past (see~\cite[Sec.~5d]{Dur88}).
In this case, the convergence to the limit measure is super-exponentially fast (i.e., the probability of each cylinder set converges super-exponentially fast to its limit value). 
In the degenerate case, that is, when $p\in\{0,1\}$ or $q\in\{0,1\}$, Bramson and Neuhauser~\cite{BraNeu94} have proved that the system is not ergodic, at least in the one-dimensional case with $\Neighb=\{-1,0,1\}$.

Following \cite[Chap.~4]{TooVasStaMitKurPir90}, Fourier analysis can in fact be used to prove ergodicity in the entire domain $0<p,q<1$.

\begin{theorem}
\label{thm-xor}
	The XOR CA with positive zero-range noise is uniformly ergodic.
	Moreover, its unique invariant measure is spatially mixing.
\end{theorem}

\begin{proof}
Define the function $\chi:\ZZ_2\to\CC$ by $\chi(a)\isdef(-1)^a$
(i.e., $\chi(\symb{0})\isdef 1$ and $\chi(\symb{1})\isdef -1$).
This is a \emph{character} of the group $\ZZ_2$
(i.e., a homomorphism into the multiplicative group of $\CC$),
and along with the constant $1$
(the trivial character), forms a basis for the two-dimensional space of functions $\ZZ_2\to\CC$.
For a finite set $A\subseteq\ZZ^d$, define $\chi_A:\X\to\CC$ by
\begin{align}
	\chi_A(x) &\isdef \prod_{i\in A}\chi(x_i) \;.
\end{align}
(In particular, $\chi_\varnothing\equiv 1$.)
The collection of all functions $\chi_A$ (for finite $A\subseteq\ZZ^d$)
is a basis (the \emph{Fourier basis}) for the linear space $C_0(\X)$,
which is orthonormal with respect to the inner product
$\langle g,h\rangle\isdef\pi(g\overline{h})$,
where $\overline{h}$ is the complex conjugate of $h$
and $\pi$ is the uniform Bernoulli measure on $\X$ (a.k.a.~the \emph{Haar measure}).

The basis $\{\chi_A:\text{$A\subseteq\ZZ^d$ finite}\}$
is particularly convenient, because the XOR CA $F$ maps each character $\chi_A$
into another character $\chi_{F^*A}$.
Namely,
\begin{align}
	\chi_A(Fx) &=
		\prod_{i\in A}\chi\big((Fx)_i\big) =
		\prod_{i\in A}\chi\Big(\sum_{\quad\mathclap{j\in i+\Neighb}\quad}x_j\Big) =
		\prod_{i\in A}\,\prod_{j\in i+\Neighb}\chi(x_j) =
		\chi_{F^*A}(x) \;,
\end{align}
where $F^*A$ denotes the set of all $j\in\ZZ^d$ for which
the set $\{i\in A: j\in i+\Neighb\}$ has an odd number of elements.
(If we represent $A$ as a configuration $c:\ZZ^d\to\ZZ_2$
with $c_i=1$ if and only if $i\in A$, then $F^*A$ will be represented by
$F^*c$ where $(F^*c)_k\isdef \sum_{j\in \Neighb} c_{k-j} \pmod{2}$.)

To calculate the effect of noise,
let $x$ be an arbitrary configuration and $Y$
a random configuration chosen according $\Theta(x,\cdot)$,
so that each $Y_i$ is obtained from $x_i$ independently at random
with transition probabilities prescribed by $\theta$.
We have
\begin{align}
	(\Theta\chi_A)(x) &= \xExp_x[\chi_A(Y)]
	=
		\xExp_x\Big[\prod_{i\in A}\chi(Y_i)\Big]
	=
		\prod_{i\in A}\xExp_x[\chi(Y_i)]
	=
		\prod_{i\in A} (\theta\chi)(x_i) \;.
\end{align}
Note how the multiplicative form of $\chi_A$ and the independence
of noise at different sites reduce
the calculation of $\Theta\chi_A$ to the calculation of $\theta\chi$.
For the latter, we have
\begin{align}
	(\theta\chi)(a) &=
		\begin{cases}
			1-2p	& \text{if $a=\symb{0}$,}\\
			2q-1	& \text{if $a=\symb{1}$,}
		\end{cases}
\end{align}
which can be written as the linear combination $\theta\chi=(q-p) + (1-p-q)\chi$.
It follows that
\begin{align}
	(\Theta\chi_A)(x) &=
		\prod_{i\in A}\big((q-p) + (1-p-q)\chi(x_i)\big) \\
	&=
		\sum_{I\subseteq A}
			(q-p)^{\abs{A\setminus I}}(1-p-q)^{\abs{I}}\chi_I(x) \;.
\end{align}

Combining the effect of the CA $F$ and the noise $\Theta$, we get
the representation
\begin{align}
	\Phi\chi_A = \Theta (\chi_A\oo F) &=
		\sum_{I\subseteq F^*A}
			(q-p)^{\abs{(F^*A)\setminus I}}(1-p-q)^{\abs{I}}\chi_I
\end{align}
in the Fourier basis.

In order to prove the ergodicity of a PCA $\Phi$, we show that for each local function $h\in C_0(\X)$, the sequence $\Phi^t h$ converges exponentially fast to a constant.
In particular, ergodicity follows if we are able to show that $\Phi$ contracts the non-constant component of $h$.
The non-constant part of $h$ can, for instance, be measured by
\begin{align}
	\snorm{h} &\isdef \sum_{\varnothing\neq A\subseteq\ZZ^d}\bigabs{\widehat{h}_A} \;,
\end{align}
where $h=\sum_{A\subseteq\ZZ^d}\widehat{h}_A\chi_A$ is the representation
of $h$ in the Fourier basis.
This is a semi-norm satisfying $\snorm{h}=0$ if and only if $h$ is constant.
Suppose that $\Phi$ is \emph{contractive} with respect to $\snorm{\cdot}$,
in the sense that there is a constant $0\leq\rho<1$ such that
$\snorm{\Phi h}\leq\rho\snorm{h}$ for all $h\in C_0(\X)$.
Then, $\snorm{\Phi^t h}\leq\snorm{h}\rho^t$ for every $h\in C_0(\X)$ and $t\geq 0$.
In particular,
\begin{align}
\label{eq:xor-noise:bound}
	\abs{\Phi^t(y,[u]) - \Phi^t(x,[u])} &\leq 2\snorm{\Phi^t \indicator{[u]}} \leq 2\snorm{\indicator{[u]}}\rho^t
\end{align}
for every cylinder set $[u]$, every two configurations $x,y\in\X$ and each $t\geq 0$.
Hence, we obtain the uniform ergodicity of $\Phi$.

In order to verify that $\Phi$ is contractive,
it is sufficient to verify that $\snorm{\Phi\chi_A}\leq\rho$
for each non-empty finite $A\subseteq\ZZ^d$.
Namely, for an arbitrary $h\in C_0(\X)$, the latter condition gives
\begin{align}
	\snorm{\Phi h} &=
		\biggsnorm{\sum_{A\subseteq\ZZ^d}\widehat{h}_A \Phi\chi_A}
	\leq
		\sum_{\varnothing\neq A\subseteq\ZZ^d}\bigabs{\widehat{h}_A}\snorm{\Phi\chi_A}
	\leq \rho \snorm{h} \;.
\end{align}

For the PCA $\Phi(x,E)\isdef \Theta(Fx,E)$, we have
\begin{align}
	\snorm{\Phi\chi_A} &=
		\biggsnorm{
			\sum_{I\subseteq F^*A}
			(q-p)^{\abs{(F^*A)\setminus I}}(1-p-q)^{\abs{I}}\chi_I
		} \\
	&=
		\sum_{\varnothing\neq I\subseteq F^*A}
			\abs{q-p}^{\abs{(F^*A)\setminus I}}\abs{1-p-q}^{\abs{I}} \\
	&=
		\left(\abs{q-p} + \abs{1-p-q}\right)^{\abs{F^*A}}
		- \abs{q-p}^{\abs{F^*A}} \;.
\end{align}
Note that $\rho\isdef\abs{q-p}+\abs{1-p-q}<1$ for $p,q\in(0,1)$.  Therefore,
$\snorm{\Phi\chi_A}\leq\rho$ for every finite $\varnothing\neq A\subseteq\ZZ^d$,
and the uniform ergodicity of $\Phi$ follows.

To see the spatial mixing of the unique invariant measure $\pi$ of $\Phi$, observe that for $u\in S^A$, we have $\snorm{\indicator{[u]}}=1-2^{-\abs{A}}\leq 1$, because
\begin{align}
	\indicator{[u]} &= \prod_{k\in A}\frac{1}{2}\big(1+\chi(u_k)\chi_k\big)
		= 2^{-\abs{A}}\sum_{B\subseteq A} \chi_B(u)\chi_B \;.
\end{align}
Integrating~\eqref{eq:xor-noise:bound} over $y$ with respect to $\pi$, we therefore get $\bigabs{\pi([u])-\Phi^t(x,[u])}\leq 2\rho^t$.  Now, using~\eqref{eq:total-variation:window}, we obtain that $d_A(t)\leq 2^{\abs{A}}\rho^t$ for every finite set $A\subseteq\ZZ^d$ and $t\geq 1$.
The spatial mixing of the invariant measure thus follows from Proposition~\ref{prop:spatial-mixing}.
\end{proof}

\begin{remark}
	Observe that $\snorm{\Phi\chi_A}<1$ even in the degenerate (but non-deterministic) case,
	for instance, when $p=0$ and $q\in(0,1)$.
	Namely, in the latter case we have $\snorm{\Phi\chi_A}=1-\abs{q-p}^{\abs{F^*A}}<1$.
	However, this is not sufficient for ergodicity, as the upper bound for $\snorm{\Phi\chi_A}$
	depends on $A$ and approaches $1$ as $A$ grows.
	\hspace*{\fill}\remarkqed
\end{remark}

\subsection{Binary spreading CA with zero-range noise}
\label{sec:mobius:binary}

Consider a non-constant CA $F$ with binary alphabet $S\isdef\{\symb{0},\symb{1}\}$ in which
$\symb{0}$ is \emph{spreading}.
Namely, $x\mapsto Fx$ is given by
\begin{align}
	(Fx)_i &\isdef
		\begin{cases}
			\symb{0}	& \text{if $x_{i+j}=\symb{0}$ for some $j\in \Neighb$,} \\
			\symb{1}	& \text{otherwise,}
		\end{cases}
\end{align}
where $\Neighb\subseteq\ZZ^d$ is a finite set.
As in the case of the XOR CA, we consider a general zero-range noise kernel $\Theta$
defined by the transition matrix
\begin{align}
	\theta &\isdef
		\begin{pmatrix}
			1-p & p \\
			q	& 1-q
		\end{pmatrix}.
\end{align}

When $q=0$, we recover Stavskaya's PCA (a.k.a.\ directed site percolation), which is non-ergodic
for sufficiently small $p\geq 0$ (see~\cite[Chap.~1]{TooVasStaMitKurPir90}).
Using coupling arguments, we already know the ergodicity of
a CA with spreading symbol with either memoryless noise (Theorem~\ref{thm-spreading-zrmn}) or sufficiently weak positive perturbation (Theorem~\ref{thm-spreading}).
In the binary case, we get an alternative argument
via (generalized) Fourier analysis, covering most of the parameter space.

\begin{theorem}
\label{thm-spreading-bin}
	The binary CA with spreading $\symb{0}$
	combined with a zero-range noise with
	transition probabilities $\symb{0}\xrightarrow{p}\symb{1}$
	and $\symb{1}\xrightarrow{q}\symb{0}$
	is uniformly ergodic if $p+\abs{1-p-q}<1$.
	Moreover, under the same condition, the unique invariant measure of
	the system is spatially mixing.
\end{theorem}

\begin{proof}
The proof is similar to that of Theorem~\ref{thm-xor} except that we use a different basis for $C_0(\X)$.
Define $\chi:S\to\CC$ by $\chi(\symb{0})\isdef 0$ and $\chi(\symb{1})\isdef 1$.
Clearly, $\{1,\chi\}$ is a basis for the linear space $\CC^S$.
For a finite $A\subseteq\ZZ^d$, define $\chi_A:\X\to\CC$ by
\begin{align}
	\chi_A(x) &\isdef \prod_{i\in A}\chi(x_i) 
	=
		\begin{cases}
			1 & \text{if $x_i=\symb{1}$ for each $i\in A$,} \\
			0 & \text{otherwise.}
		\end{cases} 
\end{align}
It is easy to verify (e.g., using the inclusion-exclusion principle) that the functions $\chi_A$ (for finite $A\subseteq\ZZ^d$) form a basis for $C_0(\X)$.
We call this basis the \emph{M\"obius basis} and each $\chi_A$ a \emph{character} of $\X$.

The advantage of the above basis is that the CA $F$
maps characters into characters.  Namely,
$\chi_A(Fx)=1$ if and only if $(Fx)_i=1$ for every $i\in A$,
which is in turn the case if and only if $x_{i+j}=1$ for every $i\in A$ and $j\in \Neighb$.
Therefore, $F\chi_A=\chi_{F^*A}$, where $F^*A\isdef A+\Neighb$.

As in the case of the Fourier basis, calculating the effect of
the noise $\Theta$ on characters boils down to calculating
the effect of the transition matrix $\theta$ on $\chi$.
For the latter, we obtain
\begin{align}
	(\theta\chi)(a) &=
		\begin{cases}
			p		& \text{if $a=\symb{0}$,} \\
			1-q		& \text{if $a=\symb{1}$,}
		\end{cases}
\end{align}
which gives $\theta\chi=p + (1-p-q)\chi$.
It follows, as in the previous case, that
\begin{align}
	\Theta\chi_A &=
		\sum_{I\subseteq A} p^{\abs{A\setminus I}}(1-p-q)^{\abs{I}}\chi_I \;.
\end{align}
For the combination of the CA $F$ and noise $\Theta$, we get
\begin{align}
	\Phi\chi_A = \Theta(\chi_A\oo F) &=
		\sum_{I\subseteq A+\Neighb} p^{\abs{(A+\Neighb)\setminus I}}(1-p-q)^{\abs{I}}\chi_I \;.
\end{align}

Each local function $h\in C_0(\X)$
has a unique representation $h=\sum_{A\subseteq\ZZ^d}\widehat{h}_A\chi_A$
as a linear combination of characters.  We define a semi-norm on $C_0(\X)$ by
\begin{align}
\label{eq:spreading-bin:seminorm}
	\snorm{h} &\isdef \sum_{\varnothing\neq A\subseteq\ZZ^d}\bigabs{\widehat{h}_A}
\end{align}
for each $h\in C_0(\X)$.
Following the same argument as in the case of the XOR CA, a sufficient condition for the uniform ergodicity of $\Phi$ is that $\Phi$ is contractive with respect to $\snorm{\cdot}$, in the sense that there is a constant $0\leq\rho<1$ such that $\snorm{\Phi h}\leq\rho\snorm{h}$ for every $h\in C_0(\X)$.
The property $\snorm{\Phi h}\leq\rho\snorm{h}$ for every $h\in C_0(\X)$ in turn is equivalent to the condition that $\snorm{\Phi\chi_A}\leq\rho$ for each non-empty finite $A\subseteq\ZZ^d$.

Clearly, $\Phi\chi_{\varnothing}=\chi_{\varnothing}$, hence $\snorm{\Phi\chi_{\varnothing}}=\snorm{\chi_{\varnothing}}=0$.
For a non-empty finite $A\subseteq\ZZ^d$, we have
\begin{align}
	\snorm{\Phi\chi_A} &=
		\biggsnorm{
			\sum_{I\subseteq A+\Neighb}
				p^{\abs{(A+\Neighb)\setminus I}}(1-p-q)^{\abs{I}}\chi_I
		} \\
	&=
		\sum_{\varnothing\neq I\subseteq A+\Neighb}
			p^{\abs{(A+\Neighb)\setminus I}}\abs{1-p-q}^{\abs{I}} \\
	&=
		\left(p + \abs{1-p-q}\right)^{\abs{A+\Neighb}}
		- p^{\abs{A+\Neighb}} \;.
\end{align}
We get uniform ergodicity if $p+\abs{1-p-q}<1$, that is if either $p+q\leq 1$ and $q>0$, or $p+q>1$ and $p+\frac{1}{2}q<1$.

The spatial mixing of the unique invariant measure follows in a similar fashion as in Theorem~\ref{thm-xor}.  Note that 
$\snorm{\indicator{[u]}}<2^{\abs{A}}$ for a cylinder with base $A$, because
\begin{align}
	\indicator{[u]}=\chi_{u^{-1}(\symb{1})}\prod_{k\in u^{-1}(\symb{0})}(1-\chi_k)
		= \sum_{B\subseteq u^{-1}(\symb{0})}(-1)^{\abs{B}}\chi_{u^{-1}(\symb{1})\cup B} \;.
\end{align}
Integrating~\eqref{eq:xor-noise:bound} over $y$ with respect to $\pi$, we therefore get $\bigabs{\pi([u])-\Phi^t(x,[u])}\leq 2\times 2^{\abs{A}}\rho^t$.  Now, using~\eqref{eq:total-variation:window}, we obtain that $d_A(t)\leq 2^{2\abs{A}}\rho^t$ for every finite set $A\subseteq\ZZ^d$ and $t\geq 1$.
The spatial mixing of the invariant measure hence follows from Proposition~\ref{prop:spatial-mixing}.
\end{proof}

\section{Open problems}
\label{sec:open}

We conclude with several open problems, some of which are already mentioned in the text.

\begin{openproblem}
	Is every ergodic PCA uniformly ergodic?
\end{openproblem}
\noindent For deterministic CA, ergodicity and uniform ergodicity are known to be equivalent~\cite{GuiRic08,Sal12, MaiMar14_TCS}.  We conjecture that the same is true for general PCA.

The ergodic PCA discussed in this article are all \emph{exponentially} ergodic, in the sense that, the probability of each cylinder set converges exponentially fast to its stationary value.
We do not know any example of an ergodic PCA that is not exponentially ergodic.
\begin{openproblem}
	Find an example of a (uniformly) ergodic PCA that is not exponentially ergodic.
\end{openproblem}
\noindent For the class of PCA that are monotonic with respect to a total ordering of the alphabet, Louis~\cite{Lou04} has provided a necessary and sufficient condition for exponential ergodicity in terms of a spatial mixing condition.

Proposition~\ref{prop:computability} above established the computability of the unique invariant measure for every ergodic PCA.
However, for the PCA discussed in this article, one can exploit the exponential ergodicity to give a ``fast'' algorithm for computing the unique invariant measure.
\begin{openproblem}
	Give an example of (uniformly) ergodic PCA for which the unique invariant measure is not computable by a ``fast'' algorithm.
\end{openproblem}

\begin{openproblem}
	Is the unique invariant measure of every (uniformly) ergodic PCA spatially mixing?
	Find an example of a (uniform) ergodic PCA whose unique invariant measure is not measure-theoretically isomorphic to a Bernoulli process.
\end{openproblem}
\noindent Proposition~\ref{prop:spatial-mixing} above provides a sufficient condition for the unique invariant measure of a uniformly ergodic PCA.
In view of the result of Goldstein et al.~\cite{GolKuiLebMae89}, we conjecture that the unique invariant measure of a positive-rate uniformly ergodic PCA is always spatially mixing.

For perturbations of a nilpotent CA with noise, we know ergodicity when noise is sufficiently high (Thm.~\ref{thm-high}) or sufficiently low (Thm.~\ref{thm-nilp}).  When the noise has zero range, one may expect ergodicity to hold for all the parameter range.
\begin{openproblem}
	Is every perturbation of a nilpotent CA with a positive zero-range noise ergodic?
\end{openproblem}

The complete ergodicity of surjective CA under positive permutation noise remains open.
\begin{openproblem}
	Is every perturbation of a surjective CA with a positive permutation noise ergodic?  How about perturbations with other types of noise?
\end{openproblem}

One of the simplest CA for which the ergodicity under noise is unknown is the majority rule.
A \emph{majority} CA is a CA with binary alphabet under which the symbol at each site is updated to the symbol that is in majority among the neighbouring sites (see Fig.~\ref{fig:majority}).  The neighbourhood has to have an odd cardinality to avoid ties.
\begin{openproblem}
	Is every small positive perturbation of a one-dimensional majority CA ergodic?
	Is every perturbation of the two-dimensional nearest-neighbour majority CA with sufficiently small positive zero-range noise non-ergodic?
\end{openproblem}
\noindent For the one-dimensional case, Gray has outline a proof of ergodicity for the nearest-neighbour marjority CA under small symmetric zero-range noise~\cite{Gra87}. 
On the other hand, Toom has proven the non-ergodicity of sufficiently small perturbations of the two-dimensional majority CA with the NEC-neighbourhood (see Example~\ref{exp:toom}).
It is conjectured that in two dimensions, the non-ergodicity holds also for the symmetric nearest-neighbour majority rule.

\begin{figure}[!ht]
\begin{scriptsize}
\begin{center}
\begin{tabular}{p{0.45\textwidth}@{}p{0.05\textwidth}@{}p{0.45\textwidth}}
	~~$\varepsilon=0$	&& ~~$\varepsilon=0.01$\\
	\boxed{\includegraphics[width=0.435\textwidth]{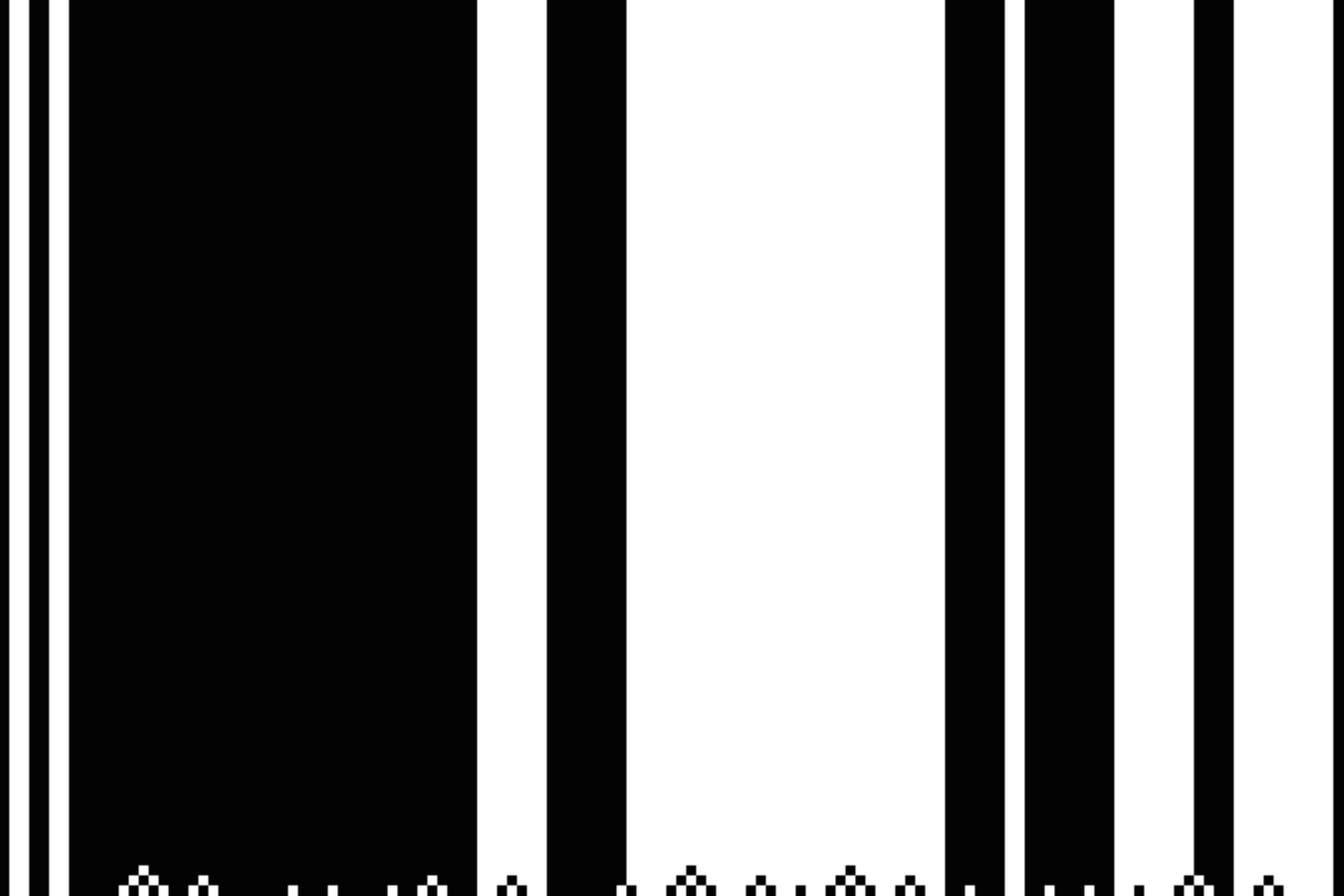}}
	&& \boxed{\includegraphics[width=0.435\textwidth]{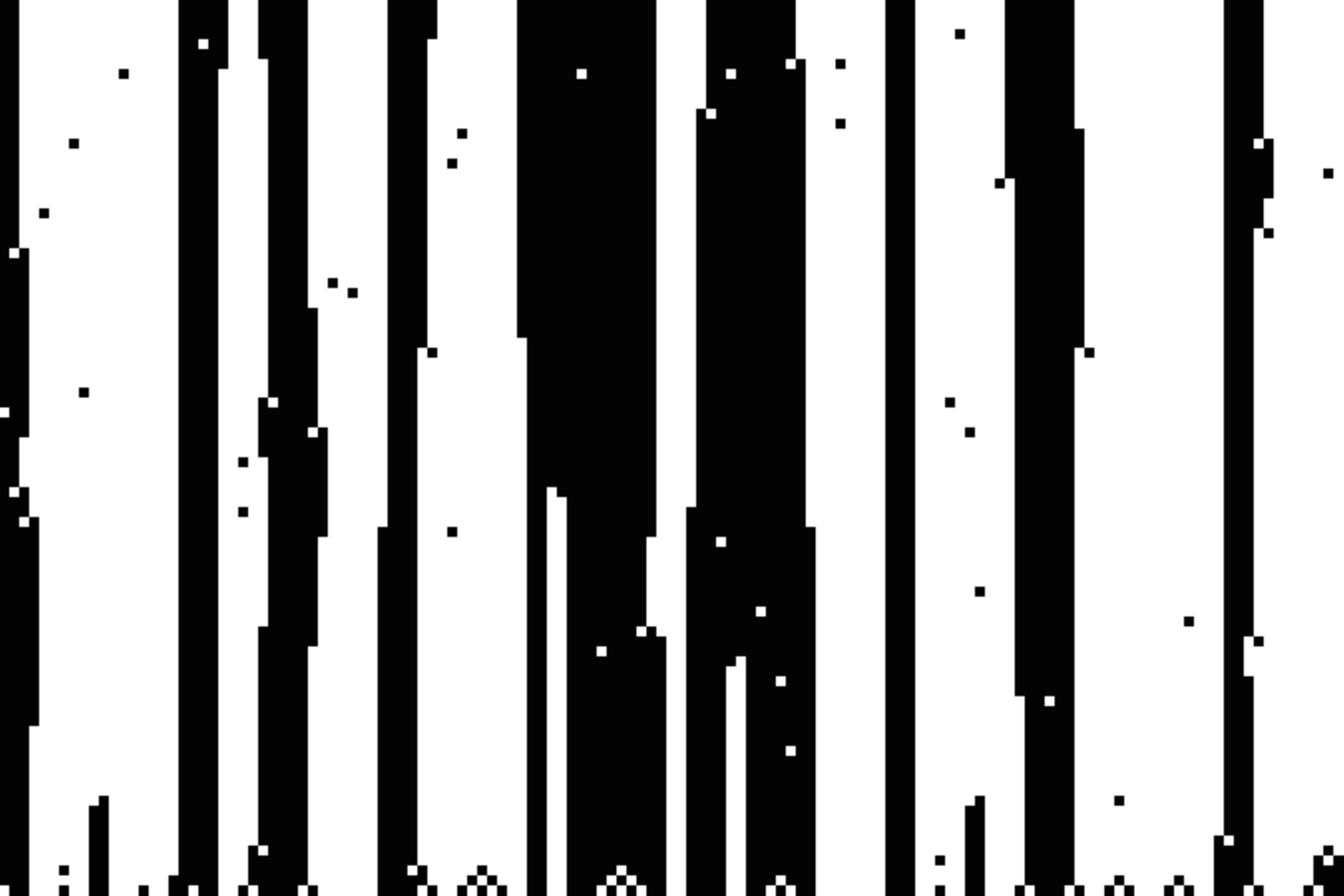}} \\ [1ex]
	\multicolumn{3}{p{0.92\textwidth}}{~~The local rule is given by
	$F(x)_i\isdef\majority(x_{i-1},x_i,x_{i+1})$.
	The noisy version appears to be ergodic.}
\end{tabular}
\end{center}
\end{scriptsize}

\caption{%
	Space-time diagrams of the majority rule perturbed by a memoryless noise with uniform replacement distribution and error probability $\varepsilon$. Time goes upwards.}
	\label{fig:majority}
\end{figure}


We end with posing two widely open-ended problems.

\begin{openproblem}
	Study the continuity of the unique invariant measure of ergodic perturbations of CA as a function of the noise parameters.
\end{openproblem}

\begin{openproblem}
	Identify classes of CA that remain non-ergodic in presence of sufficiently small noise.
\end{openproblem}
\noindent See~\cite{Too80} for a class of two-dimensional examples, and~\cite{Gac86,Gac01} for a one-dimensional example.

\bibliographystyle{plainurl}
\bibliography{bibliography}

\end{document}